%% file: todanmk-final.tex
\documentclass[reqno,12pt]{amsart}
\usepackage{amsthm,amssymb}
\usepackage{epsfig}
\usepackage{xypic}
\usepackage[arrow]{xy}
\usepackage{young}

\usepackage{latexsym,cite,epsf,graphics}
\usepackage[a4paper,height=20cm,top=26mm,bottom=26mm,left=26mm,right=26mm]{geometry}
\usepackage{graphicx,color}
\usepackage{tabmac1}
\usepackage{hyperref}

\numberwithin{equation}{section}

\newtheorem*{statement}{Theorem}
\newtheorem{thm}{Theorem}[section]
\newtheorem{cor}[thm]{Corollary}
\newtheorem{lem}[thm]{Lemma}
\newtheorem{prop}[thm]{Proposition}
\theoremstyle{definition}
\newtheorem{definition}[thm]{Definition}
\newtheorem{remark}[thm]{Remark}
\newtheorem{example}[thm]{Example}
\newtheorem{conjecture}[thm]{Conjecture}

\newcommand\remind[1]{{\tt *** #1 ***}}

\def\sign{{\rm sign}}
\def\mS{{\mathfrak{S}}}
\def\mult{{\rm mult}}
\def\aa{{\mathbf{a}}} 
\def\bb{{\mathbf{b}}}
\def\jdt{{\mathrm{jdt}}}
\def\oV{{\mathring {\mathcal V}}}
\def \tv{{\tilde v}}
\def \tl{{a}}

\def\Z{{\mathbb Z}}
\def\R{{\mathbb R}}
\def\C{{\mathbb C}}
\def\P{{\mathbb P}}
\def\L{{\mathcal L}}
\def\D{{\mathcal D}}
\def\Pic{{\mathrm{Pic}}}
\def\tW{{\tilde W}}
\def\mM{{\mathcal M}}
\def\mL{{\mathcal L}}

\def\C{{\mathbb C}}
\def\Z{{\mathbb Z}}
\def\D{{\mathcal{D}}}
\def\OO{{\mathcal{O}}}
\def\be{{\mathbf{e}}}
\def\p{{\mathbf{p}}}
\def\wt{{\rm wt}}
\def\a{{\mathfrak{a}}}
\def\b{{\mathfrak{b}}}
\def\c{{\mathfrak{c}}}
\def\d{{\mathfrak{d}}}

\begin{document}

\baselineskip 16pt

\title[Generalized discrete Toda lattices]
{Toric networks, geometric $R$-matrices and generalized discrete Toda lattices}

\author[Rei Inoue]{Rei Inoue}
\address{Rei Inoue, Department of Mathematics and Informatics,
   Faculty of Science, Chiba University,
   Chiba 263-8522, Japan.}
\email{reiiy@math.s.chiba-u.ac.jp}
\thanks{R.I. was partially supported by JSPS KAKENHI Grant Number
26400037.}

\author[Thomas Lam]{Thomas Lam}
\address{Thomas Lam, Department of Mathematics, 
University of Michigan, Ann Arbor, MI 48109, USA.}
\email{tfylam@umich.edu}
\thanks{T.L. was partially supported by NSF grants DMS-1160726, DMS-1464693, and a Simons Fellowship.}

\author[Pavlo Pylyavskyy]{Pavlo Pylyavskyy} 
\address{Pavlo Pylyavskyy, School of Mathematics, University of Minnesota, 
Minneapolis, MN 55414, USA.}
\email{ppylyavs@umn.edu}
\thanks{P. P. was partially supported by NSF grants DMS-1148634, DMS-1351590, and Sloan Fellowship.}

\date{August 31, 2015 (Final version: June 19, 2016)}

\subjclass[2000]{}

\keywords{}

\begin{abstract}
We use the combinatorics of toric networks and the double affine geometric $R$-matrix to define a three-parameter family of generalizations of the discrete Toda lattice.  We construct the integrals of motion and a spectral map for this system.  The family of commuting time evolutions arising from the action of the $R$-matrix is explicitly linearized on the Jacobian of the spectral curve.  The solution to the initial value problem is constructed using Riemann theta functions.
\end{abstract}

\maketitle

\section{Introduction}

We study a generalization of the discrete Toda lattice parametrized by a triple of integers $(n,m,k)$,
which corresponds to a network on a torus with
$n$ horizontal wires, $m$ vertical wires and $k$ {\em shifts} at the horizontal boundary.  An example of the kind of toric network that we consider is given in Figure \ref{fig:toda3}.

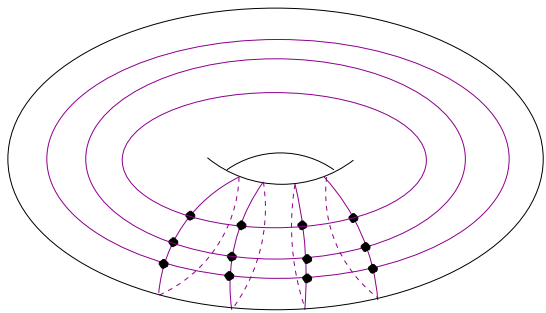
\begin{figure}[h!]
    \begin{center}
    \scalebox{0.8}{\input{toda3.pstex_t}}
    \end{center}
    \caption{The toric network for $n=3$, $m=4$ and $k=0$.}
    \label{fig:toda3}
\end{figure}

The phase space $\mM$ of our system is the space of parameters $q_{ij} \in \C$ assigned to each of the crossings of the wires (a $3 \times 4 = 12$ dimensional space for the example in Figure \ref{fig:toda3}). 
We construct two families of commuting discrete time evolutions acting on the parameters $q_{ij}$, together generating an action of $\Z^m \times \Z^N$, where $N := \gcd(n,k)$.  These time evolutions act as birational transformations of the phase space.  The purpose of this paper is to study the algebro-geometrical structure of these maps and to solve the corresponding initial value problem.

\medskip

The rational transformations of the phase space come from the {\it {affine geometric $R$-matrix}}, acting on the parameters of two adjacent parallel wires, either horizontal or vertical.  The affine geometric $R$-matrix arises in the theory of affine geometric crystals \cite{BK, KNO1,KNO2}, being a birational lift of the {\it {combinatorial $R$-matrix}} of certain tensor products of Kirillov-Reshetikhin crystals for $U_q({\widehat{\mathfrak{sl}}_n})$ \cite{KKMMNN}. Geometric $R$-matrices also arises independently in the study of Painlev\'e equations \cite{KNY} and total positivity \cite{LP}; see also \cite{Ki,Et}.

The geometric $R$-matrix satisfies the Yang-Baxter relation and we show in Theorem \ref{thm:dynamics} that the action of the geometric $R$-matrix generates commuting birational actions of two affine Weyl groups $W$ and $\tW$: one swapping horizontal wires, and one swapping vertical wires.  The commutativity was proved in \cite{KNY} for the case $k=0$; we extend it here to arbitrary $k$.

The $\Z^m \times \Z^N$ time-evolutions of our dynamical system come from the subgroup of translation elements in the corresponding affine Weyl group.  In the case $(n,m,k) = (n,2,n-1)$, a rational map given by one of the $\Z^{m=2}$ actions corresponds to the discretization of the well-known $n$-periodic Toda lattice equation, studied in \cite{HTI}.  

The dynamics that we study come from positive birational maps, and they tropicalize to the box-ball system \cite{TS} and the combinatorics of jeu-de-taquin (see \S \ref{sec:tab}).

\medskip
To study the dynamics of our generalized discrete Toda lattice, we construct a spectral map
\begin{align*}
\phi:\mM &\longrightarrow \{\text{plane algebraic curve}\} \times \Pic^g(C_f) \times \mathcal{S}_f \times R_O \times R_A \\
(q_{ij}) &\longmapsto (C_f, \D, (c_1,\ldots,c_M), O , A)
\end{align*}
where $C_f$ is a spectral curve, $\D \in \Pic^g(C_f)$ is degree $g$ divisor, and the remaining data is explained in \S \ref{sec:eigenvector}.  We show that when appropriately restricted, the spectral map is an injection.  Such spectral data is frequently encountered in the theory of integrable systems, and our approach follows that of van Moerbeke and Mumford \cite{vMM}.  Van Moerbeke and Mumford used similar spectral data to study periodic difference operators.

The heart of our work is the calculation of the double affine geometric $R$-matrix action in terms of the spectral data: we show that the translation subgroup $\Z^m \times \Z^N$ acts as constant motions on the Jacobian of $C_f$, while the symmetric subgroups of $W$ and of $\tW$ act by permuting the additional data $R_O$ and $R_A$ (which are certain special points on $C_f$).  

We explicitly invert (for $N = 1$) the spectral map $\phi$ using Riemann theta functions, and give a solution to the initial value problem.  This extends work of Iwao \cite{Iwao07,Iwao10}, who studied the initial value problem of the $\Z^m$ action in the $(n,m,n-1)$ and $(n,m,0)$ cases.  We relate our theta function solutions to the octahedron recurrence \cite{Spe} via Fay's trisecant identity\cite{Fay73}.  We also give an interpretation of our dynamics in terms of dimer model transformations \cite{GonchaKenyon13}.

\medskip
\noindent {\bf Outline.}
The outline of this paper is as follows:
in \S \ref{sec:netw}, we introduce a family of commuting actions 
on the toric network, generalizating the 
discrete Toda lattice.  
We briefly summarize the main results of this paper in \S \ref{subsec:main}, and give some examples and applications in \S \ref{subsec:example}.
In \S \ref{sec:Lax}, we introduce a space $\mL$ of Lax matrices for our system, and construct the spectral curves $C_f$, where $f \in \C[x,y]$.  We also study certain special points on our spectral curves, and analyze some of the singular points.
In \S \ref{sec:eigenvector}, we study the spectral map which sends a point in the phase space to the spectral curve $C_f$, a divisor $\D$ on $C_f$, and some additional data.  We modify the strategy in \cite{vMM} to fit our situation.  In \S \ref{sec:proofLax} and \S \ref{sec:proofeigenvector}
we present the proofs of results in \S \ref{sec:Lax} and 
\S \ref{sec:eigenvector} respectively.  In particular, the coefficients of the spectral curves (which are integrals of motion) are described explicitly in terms of the combinatorics of these networks.
In \S \ref{sec:actions}, we study the vertical $\Z^m$ actions, the horizontal $\Z^N$ actions, and the snake path actions.
We show that all the actions preserve the spectral curve, 
and the commuting $m+N$ actions induce constant motions on 
the Picard group of $C_f$, through the map $\phi$. 
In \S \ref{sec:theta}, we solve, for the case of $N=1$, the initial value problem 
for the commuting actions by constructing the inverse of $\phi$ 
in terms of the Riemann theta function.
Our method extends the strategy in \cite{Iwao10}.  For $N > 1$, a solution is given that relies on a technical condition.
In \S \ref{sec:fay}, we show that the theta function solution for the system satisfies the octahedron recurrence,
by specializing Fay's trisecant identity for the Riemann theta function.  In \S \ref{sec:transpose}, we study the symmetry between
the network and its transposition obtained by swapping the roles of the
vertical and the horizontal wires. 
In \S \ref{sec:cluster}, we realize the $R$-matrix transformation
on a toric network in terms of transformations of the honeycomb dimer model on a torus~\cite{GonchaKenyon13}.  An explicit interpretation of the $R$-matrix as a cluster transformation will be the subject of future work \cite{ILP}, so we do not elaborate on the relation to cluster structures here.

\medskip
\noindent
{\bf Future directions.}
There are several systematic ways to construct integrable rational maps 
using combinatorial objects on surfaces, such as directed networks, electrical networks, and bipartite graphs~\cite{LP,GonchaKenyon13,GSTV}.  In particular, the $R$-matrix of the present work has an ``electrical" analogue \cite{LPElec}.  It would be interesting to study these discrete-time dynamical systems from the view point of algebraic geometry (Cf.~\cite{FockMarsha14}).
In all these cases, the rational transformations are additionally positive, and can be tropicalized to piecewise-linear maps that generate a discrete-time and discrete-space dynamical system.  It is natural to consider the tropical counterpart of our work using tropical geometry as in\cite{InoueTakenawa}, where
tropical curves and tropical theta functions are used to study 
the piecewise-linear map arising from the $(n,m,k) = (n,2,n-1)$ case.

\medskip
\noindent
{\bf Acknowledgments.}  We thank Rick Kenyon for helpful discussions.
We also thank the anonymous referee for kind comments which improved 
the manuscript.

\section{Dynamical system from a toric network} 
\label{sec:netw}

We study the action of the geometric $R$-matrix on an array of variables.  The geometric $R$-matrix generates an action of a product of commuting (extended) affine symmetric groups, acting as birational transformations.

\subsection{The affine geometric $R$-matrix} \label{sec:Rmatrix}
For a vector $\aa = (a_1,\ldots,a_n)$, let $\aa^{(i)}:= (a_{i+1},a_{i+2},\ldots,a_n,a_1,\ldots,a_{i})$.
Let $\aa = (a_1,\ldots,a_n)$ and $\bb= (b_1,\ldots,b_n)$ be two vectors.  Define the {\it {energy}} $E(\aa,\bb)$ to be
$$E(\aa,\bb) = \sum_{i=0}^{n-1} \left(\prod_{j=1}^i b_{j} \prod_{j=i+2}^{n} a_{j} \right).$$
Define the affine geometric $R$-matrix to be the transformation $R:(\aa,\bb) \mapsto (\bb',\aa')$, given by
$$b_{i}' = b_{i}\frac{E(\aa^{(i)}, \bb^{(i)})}{E(\aa^{(i-1)}, \bb^{(i-1)})}   \qquad  \text{and} \qquad
a_{i}' = a_{i}\frac{E(\aa^{(i-1)}, \bb^{(i-1)})}{E(\aa^{(i)}, \bb^{(i)})}.$$
We will usually think of $R$ as a birational transformation of $\C^n \times \C^n$.

It is easy to see 
\begin{align}\label{eq:q1q2}
  \prod_{i=1}^n a_{i} = \prod_{i=1}^n a_{i}'
  \qquad  \text{and} \qquad
  \prod_{i=1}^n b_{i} = \prod_{i=1}^n b_{i}'.
\end{align}

\subsection{Description of dynamics}\label{sec:dynamics}
The $R$-matrix can be used to act on a rectangular array of variables, by acting on consecutive pairs of rows or of columns.  

Let $\a,\b,\c$ be three positive integers, and let $0 \leq \d <\c$ be a nonnegative integer satisfying $\gcd(\c,\d)=1$. Note that we allow $\d=0$ if and only if $\c=1$. 
Let $\d^{-1}$ be the unique number in the range $0 \leq \d^{-1} <\c$ such that $\d\d^{-1}=1 \mod \c$.  For a positive integer $a$, we define $[a] := \{1,2,\ldots,a\}$.

We consider an array of $\a \times \b \times \c$ variables $\{q_{a,b,c}\}_{a \in [\a], b \in [\b], c \in [\c]}$. We consider two distinct ways to arrange them in a two-dimensional rectangular array: 
\begin{align}\label{eq:q-barq}
q_{i,j} = q_{a,b,c} \text { for } i=a, \; j=b+\b(c-1)
\end{align}
and 
\begin{align}\label{eq:q-tildeq}
\tilde q_{i,j} = q_{a,b,c} \text { for } i=\b+1-b, \; j=\a \d^{-1} c - a +1 \mod \a\c.
\end{align}
By convention, $q_{i,j}$ denotes the entry in the {\it $i$-th column} and {\it $j$-th row}. 
We call these arrays $Q$
and $\tilde Q$.  The first array $Q$ has dimensions $\b\c \times \a$ and the second array $\tilde Q$ has dimensions $\a\c \times \b$. 
Note that changing $\a \mapsto \b$, $\d \mapsto \d^{-1}$, $q_{a,b,c} \mapsto q_{1-b,1-a,\d^{-1}c}$ swaps $Q$ and $\tilde Q$.

\begin{example}\label{ex:2221}
 Let $(\a,\b,\c,\d)= (2,2,2,1)$. The two arrays are as follows.
 \begin{align}
Q := \left(\begin{array}{cc} 
q_{1,1,1} & q_{2,1,1} \\
q_{1,2,1} & q_{2,2,1} \\
q_{1,1,2} & q_{2,1,2}\\
q_{1,2,2} & q_{2,2,2} 
\end{array} \right), 
\qquad
\tilde Q := \left(\begin{array}{cc} 
q_{2,2,1} & q_{2,1,1}\\
q_{1,2,1} & q_{1,1,1}\\
q_{2,2,2} & q_{2,1,2}\\
q_{1,2,2} & q_{1,1,2}
\end{array} \right).
\end{align}
\end{example}

\begin{example}\label{ex:3232}
 Let $(\a,\b,\c,\d)= (3,2,3,2)$. The two arrays are as follows.
 \begin{align}
Q := \left(\begin{array}{ccc} 
q_{1,1,1} & q_{2,1,1} & q_{3,1,1}\\
q_{1,2,1} & q_{2,2,1} & q_{3,2,1}\\
q_{1,1,2} & q_{2,1,2} & q_{3,1,2}\\
q_{1,2,2} & q_{2,2,2} & q_{3,2,2}\\
q_{1,1,3} & q_{2,1,3} & q_{3,1,3}\\
q_{1,2,3} & q_{2,2,3} & q_{3,2,3}
\end{array} \right), 
\qquad 
\tilde Q := \left(\begin{array}{cc} 
q_{3,2,1} & q_{3,1,1}\\
q_{2,2,1} & q_{2,1,1}\\
q_{1,2,1} & q_{1,1,1}\\
q_{3,2,3} & q_{3,1,3}\\
q_{2,2,3} & q_{2,1,3}\\
q_{1,2,3} & q_{1,1,3}\\
q_{3,2,2} & q_{3,1,2}\\
q_{2,2,2} & q_{2,1,2}\\
q_{1,2,2} & q_{1,1,2}
\end{array} \right).
\end{align}
\end{example}

The arrays $Q$ and $\tilde Q$ correspond to a network on a torus in 
the following way. 
Let $G$ a network embedded into a torus, as illustrated in Figure~\ref{fig:networkG}. 
Each rectangle of red lines denotes the fundamental domain of the torus.
There are $\a$ vertical ``wires" (directed upwards), forming $\a$ simple curves in the torus with the same homology class, and $\b \c$ horizontal wires (directed to the right), which form $\b$ simple curves on the torus with another homology class: the top $\b \d$ right ends of horizontal wires cross the top edge and come out on the other side. 
We set $q_{ij}$ at the crossing of the $i$-th vertical and the $j$-th horizontal wires, then $Q$ corresponds to the configuration of the $q_{ij}$ on $G$.
When we see $G$ from the inside of the torus, the roles of the 
vertical and horizontal wires are interchanged;
there are $\b$ vertical wires forming $\b$ simple curves, and 
$\a \c$ horizontal wires forming $\a$ simple curves. 
The top $\a \d^{-1}$ ends of horizontal wires cross the top edge 
and come out on the other side. 
The array $\tilde Q$ corresponds to this ``from-inside" configuration,
where $\tilde q_{ij}$ is placed at 
the crossing of the $i$-th vertical and the $j$-th horizontal wires.

See Figure~\ref{fig:3232} for the network in the case of Example~\ref{ex:3232},
where the lines with the same color in the two pictures are identical
simple curves in $G$.

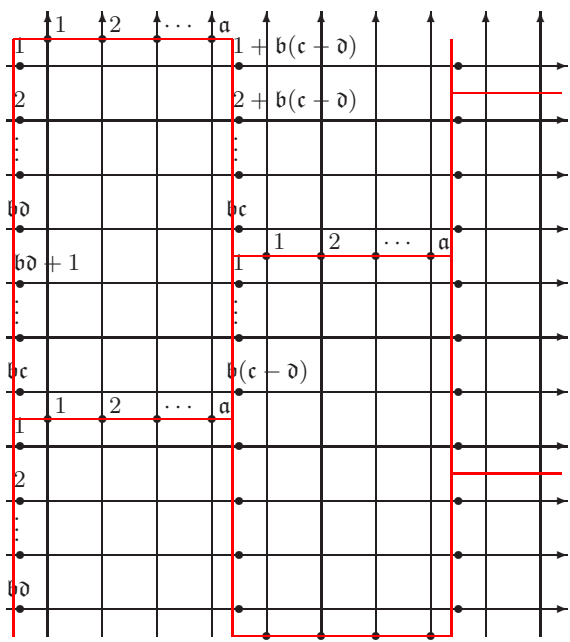
\begin{figure}[h]
\unitlength=0.9mm
\begin{picture}(100,95)(0,-8)

\multiput(10,-8)(8,0){10}{\vector(0,1){92}}
\multiput(4,76)(0,-8){11}{\vector(1,0){82}}

\multiput(10,80)(8,0){4}{\circle*{1}}
\multiput(10,24)(8,0){4}{\circle*{1}}
\multiput(42,48)(8,0){4}{\circle*{1}}
\multiput(42,-8)(8,0){4}{\circle*{1}}
\multiput(6,76)(0,-8){11}{\circle*{1}}
\multiput(38,76)(0,-8){11}{\circle*{1}}
\multiput(70,76)(0,-8){11}{\circle*{1}}

\put(11,81){\scriptsize $1$}
\put(19,81){\scriptsize $2$}
\put(27,81){\scriptsize $\cdots$}
\put(35,81){\scriptsize $\a$}

\put(11,25){\scriptsize $1$}
\put(19,25){\scriptsize $2$}
\put(27,25){\scriptsize $\cdots$}
\put(35,25){\scriptsize $\a$}

\put(43,49){\scriptsize $1$}
\put(51,49){\scriptsize $2$}
\put(59,49){\scriptsize $\cdots$}
\put(67,49){\scriptsize $\a$}

\put(5,78){\scriptsize $1$}
\put(5,70){\scriptsize $2$}
\put(5,62){\scriptsize $\vdots$}
\put(4,54){\scriptsize $\b \d$}
\put(5,46){\scriptsize $\b \d+1$}
\put(5,38){\scriptsize $\vdots$}
\put(4,30){\scriptsize $\b \c$}
\put(5,22){\scriptsize $1$}
\put(5,14){\scriptsize $2$}
\put(5,6){\scriptsize $\vdots$}
\put(4,-2){\scriptsize $\b \d$}

\put(37,78){\scriptsize $1+\b(\c-\d)$}
\put(37,70){\scriptsize $2+\b(\c-\d) $}
\put(37,62){\scriptsize $\vdots$}
\put(36,54){\scriptsize $\b \c$}
\put(37,46){\scriptsize $1$}
\put(37,38){\scriptsize $\vdots$}
\put(36,30){\scriptsize $\b (\c-\d)$}

{\color{red} 
\multiput(5,80)(0,-56){2}{\line(1,0){32}}
\multiput(37,48)(0,-56){2}{\line(1,0){32}}
\multiput(69,72)(0,-56){2}{\line(1,0){16}}
\multiput(5,80)(32,0){3}{\line(0,-1){88}}

}
\end{picture}
    \caption{Toric network}
    \label{fig:networkG}
\end{figure}

\begin{figure}[h]
\unitlength=0.9mm
\begin{picture}(100,85)(0,5)

\put(-8,80){$Q:$}

\multiput(10,30)(8,0){3}{\vector(0,1){54}}
\multiput(4,76)(0,-8){6}{\vector(1,0){30}}

\multiput(10,80)(8,0){3}{\circle*{1}}
\multiput(10,32)(8,0){3}{\circle*{1}}
\multiput(6,76)(0,-8){6}{\circle*{1}}
\multiput(30,76)(0,-8){6}{\circle*{1}}

\multiput(11,81)(0,-48){2}{\scriptsize $1$}
\multiput(19,81)(0,-48){2}{\scriptsize $2$}
\multiput(27,81)(0,-48){2}{\scriptsize $3$}

\put(5,78){\scriptsize $1$}
\put(5,70){\scriptsize $2$}
\put(5,62){\scriptsize $3$}
\put(5,54){\scriptsize $4$}
\put(5,46){\scriptsize $5$}
\put(5,38){\scriptsize $6$}

\put(29,78){\scriptsize $3$}
\put(29,70){\scriptsize $4$}
\put(29,62){\scriptsize $5$}
\put(29,54){\scriptsize $6$}
\put(29,46){\scriptsize $1$}
\put(29,38){\scriptsize $2$}


\put(50,80){$\tilde Q:$}

\multiput(68,6)(8,0){2}{\vector(0,1){78}}
\multiput(62,76)(0,-8){9}{\vector(1,0){22}}

\multiput(68,80)(8,0){2}{\circle*{1}}
\multiput(68,8)(8,0){2}{\circle*{1}}
\multiput(64,76)(0,-8){9}{\circle*{1}}
\multiput(80,76)(0,-8){9}{\circle*{1}}

\multiput(69,81)(0,-72){2}{\scriptsize $1$}
\multiput(77,81)(0,-72){2}{\scriptsize $2$}

\put(63,78){\scriptsize $1$}
\put(63,70){\scriptsize $2$}
\put(63,62){\scriptsize $3$}
\put(63,54){\scriptsize $4$}
\put(63,46){\scriptsize $5$}
\put(63,38){\scriptsize $6$}
\put(63,30){\scriptsize $7$}
\put(63,22){\scriptsize $8$}
\put(63,14){\scriptsize $9$}

\put(79,78){\scriptsize $4$}
\put(79,70){\scriptsize $5$}
\put(79,62){\scriptsize $6$}
\put(79,54){\scriptsize $7$}
\put(79,46){\scriptsize $8$}
\put(79,38){\scriptsize $9$}
\put(79,30){\scriptsize $1$}
\put(79,22){\scriptsize $2$}
\put(79,14){\scriptsize $3$}

{\color{green}
\multiput(4,75)(0,-16){3}{\line(1,0){30}}
\put(77,6){\line(0,1){78}}
}

{\color{red}
\put(9,30){\line(0,1){54}}
\multiput(60,59)(0,-24){3}{\line(1,0){22}}
}
\end{picture}
    \caption{Network for the case of $(3,2,3,2)$}
    \label{fig:3232}
\end{figure}
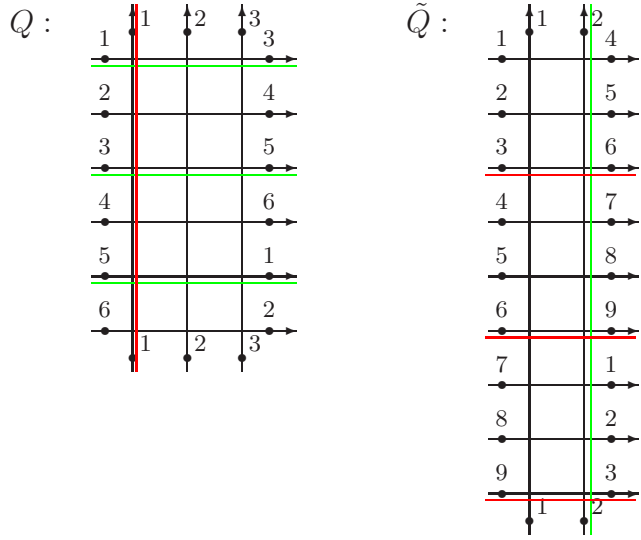

\begin{definition}\label{def:action}
Let $A = \{a_{i,j}\}_{i \in [s], j \in [r]}$ be an $r \times s$ array.
Let $\aa_i = (a_{i,j})_{j \in [r]}$,
and write $\aa_i^T$ for the transpose of $\aa_i$.
For $1 \leq \ell \leq s-1$, define the array $s_\ell(A)$ by 
\begin{align}\label{eq:s-ell}
s_{\ell}(A) = \left(\aa_1^T, \ldots, \aa_{\ell-1}^T, (\aa'_{\ell+1})^T,
 (\aa'_{\ell})^T,\aa_{\ell+2}^T, \ldots, \aa_s^T \right)
\end{align}
if $A = (\aa_1^T, \ldots, \aa_{\ell-1}^T, \aa_{\ell}^T, \aa_{\ell+1}^T, \aa_{\ell+2}^T, \ldots, \aa_s^T)$,
and $R(\aa_\ell,\aa_{\ell+1}) = (\aa'_{\ell+1},\a'_{\ell})$.
Also define an array $\pi(A)$ by the formula
\begin{align}\label{eq:pi}
\pi(A)_{i,j} = a_{i,j-1}.
\end{align}
Here we consider the first index modulo $r$ and the second index modulo $s$.  
\end{definition}

Let $\mM \simeq \C^{\a \b \c}$ be the phase space of our dynamical system
and regard $(q_{a,b,c})_{a \in [\a], b \in [\b], c \in [\c]}$ as 
coordinates on $\mM$. Applying the above definition to the array $Q$, we obtain operators $s_{1}, \ldots, s_{\a-1}$ and $\tilde \pi$ acting on $\mM$.  
Similarly, applying this definition to the array $\tilde Q$, we obtain operators $\tilde s_{1}, \ldots, \tilde s_{\b-1}$ and $\pi$ acting on $\mM$.  We emphasize that \eqref{eq:pi} for $Q$ gives the operation $\tilde \pi$, while $\eqref{eq:pi}$ for $\tilde Q$ gives the operation $\pi$.

The following result generalizes a result of Kajiwara, Noumi and Yamada \cite{KNY}.
\begin{thm}\label{thm:dynamics}
The operators $s_{\ell}$ and $\pi$ generate an action of an extended affine symmetric group $W = (\Z/\a\c\Z) \ltimes \widehat{\mathfrak S}_{\a}$ on $\mM$. Here $\Z/\a\c\Z$ is the cyclic group of order $\a \c$ and $\widehat{\mathfrak S}_{\a}$ is the affine symmetric group (the Coxeter group) of type $\tilde A_{\a-1}$.  Similarly, the operators $\tilde s_{\ell}$ and $\pi$ form an action of an extended affine symmetric group $\tilde W = (\Z/\b\c\Z) \ltimes \widehat{\mathfrak S}_{\b}$ on $\mM$, where $(\Z/\b\c\Z)$ is the cyclic group of order  $\b \c$. Furthermore, these two actions commute.
Precisely, on $\mM$ the following relations hold.  
\begin{align*}
  & s_{\ell}  s_{\ell+1}  s_{\ell}
  = s_{\ell+1}  s_{\ell}  s_{\ell+1},
  \qquad 
   s_{j}  s_{\ell} =  s_{\ell}  s_{j}
  \quad (|j-\ell| > 1),  \qquad s_\ell^2 = 1,
  \\
  & \pi  s_{\ell+1} =  s_{\ell} \pi, \qquad
\pi^{\a \c} = 1,
\end{align*}
\begin{align*}
    &\tilde s_{\ell} \tilde s_{\ell+1} \tilde s_{\ell}
  =\tilde s_{\ell+1} \tilde s_{\ell} \tilde s_{\ell+1},
  \qquad 
  \tilde s_{j} \tilde s_{\ell} = \tilde s_{\ell} \tilde s_{j}
  \quad (|j-\ell| > 1), \qquad \tilde s_\ell^2 = 1,
  \\
  & \tilde \pi \tilde s_{\ell+1} = \tilde s_{\ell}  \tilde \pi, \qquad
 {\tilde \pi}^{\b \c} = 1, 
\end{align*}
\begin{align*}
  & s_{j} \tilde s_{\ell} = \tilde s_{\ell}  s_{j}, 
  \qquad
   \tilde \pi  s_{\ell} =  s_{\ell}  \tilde \pi,
  \qquad 
   \pi \tilde s_{\ell} = \tilde s_{\ell} \pi,
  \qquad
  \tilde \pi \pi = \pi \tilde \pi.
\end{align*}  
In the above formulae, the index of $s_\ell$ is taken modulo $\a$, and $s_0$ is defined by the equation $ \pi s_{1} = s_{0} \pi$.  Similarly, the index of $\tilde s_\ell$ is taken modulo $\b$, and $\tilde s_0$ is defined by the equation $\tilde \pi \tilde s_{1} = \tilde s_{0} \tilde \pi$.
\end{thm}

The proof of Theorem \ref{thm:dynamics} is delayed to \S \ref{sec:dynamics_proof}.

\begin{example}
 In Example~\ref{ex:2221} we have
$$ s_1(Q)_{1,1,1}= q_{2,1,1} \frac{q_{1,1,2}q_{1,2,2}q_{1,1,1} + q_{2,2,1}q_{1,2,2}q_{1,1,1} +q_{2,2,1}q_{2,1,2}q_{1,1,1} +q_{2,2,1}q_{2,1,2}q_{2,2,2}}
 {q_{1,2,1}q_{1,1,2}q_{1,2,2} + q_{2,1,1}q_{1,1,2}q_{1,2,2} +q_{2,1,1}q_{2,2,1}q_{1,2,2} + 
 q_{2,1,1}q_{2,2,1}q_{2,1,2}},$$
$$\tilde s_1(\tilde Q)_{1,1,1}= q_{1,2,1} \frac{q_{1,1,1}q_{2,1,2}q_{1,1,2}+q_{1,1,1}q_{2,1,2}q_{2,2,1} + q_{1,1,1}q_{1,2,2}q_{2,2,1} + 
q_{2,2,2}q_{1,2,2}q_{2,2,1}}
{q_{2,1,2}q_{1,1,2}q_{2,1,1}+q_{2,1,2}q_{1,1,2}q_{1,2,1} + q_{2,1,2}q_{2,2,1}q_{1,2,1} + 
q_{1,2,2}q_{2,2,1}q_{1,2,1}},$$ 
$$ \tilde \pi(Q)_{1,1,1} = q_{1,2,1}, \qquad  \pi(\tilde Q)_{1,1,1} = q_{2,1,2}.$$
We leave it for the reader to verify that $s_1$ and $\pi$ commute with $\tilde s_1$ and $\tilde \pi$, as an easy computational exercise. 
\end{example}

We can present the extended affine symmetric groups $W$ and $\tilde W$ as follows. The finite symmetric group $\mS_{\a}$ acts naturally on the lattice $\Z^{\a}$, fixing the subgroup generated by the vector $(1,1,\ldots,1)$.  Thus we have an action of $\mS_\a$ on the quotient $\Z^{\a}/\Z(\c,\c,\ldots,\c)$.  Then we have that $W = \mS_{\a} \ltimes (\Z^{\a}/\Z(\c,\c,\ldots,\c))$.  The commutative normal subgroup $\Z^{\a}/\Z(\c,\c,\ldots,\c)$ is generated by $e_u$ for $1 \leq u \leq \a$:
\begin{align}\label{eq:Za-action}
e_u = (s_u \cdots s_{\a-1})(s_{u-1} \cdots s_{\a-2})
        \cdots (s_1 \cdots  s_{\a-u}) {\pi}^u.
\end{align}
The element $e_u$ is identified with the vector $\be_u = (1,\ldots,1,0,\ldots,0) \in \Z^a$ with $u$ $1$-s.  Note that $e_\a = \tilde \pi^\a$ satisfies $e_\a^\c = 1$ agreeing with the fact that $(\c,\c,\ldots,\c) = 0$ in $\Z^{\a}/\Z(\c,\c,\ldots,\c)$.
Similarly, we have $\tilde W = \mS_\b\ltimes (\Z^{\b}/\Z(\c,\c,\ldots,\c))$,
where the commutative normal subgroup $\Z^{\b}/\Z(\c,\c,\ldots,\c)$ is generated by 
$\tilde e_u$ for $1 \leq u \leq \b$:
\begin{align}\label{eq:Zm-action}
\tilde e_u = (\tilde s_u \cdots \tilde s_{\b-1})
        (\tilde s_{u-1} \cdots \tilde s_{\b-2})
        \cdots (\tilde s_1 \cdots \tilde s_{\b-u}) {\tilde \pi}^u.
\end{align}

Now, the operators $e_u$ and $\tilde e_u$ all commute, and thus give an action of $\Z^{\a}/\Z(\c,\c,\ldots,\c) \times \Z^{\b}/\Z(\c,\c,\ldots,\c)$ on $\mM$!  We will think of this as a discrete dynamical system with $\a + \b$ different time evolutions.  We shall find a complete set of integrals of motion for this system, and study the initial value problem.

\subsection{Change of indexing}

Instead of the quadruple $(\a,\b,\c,\d)$, we shall also index our systems with triples $(n,m,k)$, given by
$$n=\b \c, \qquad  m=\a, \qquad k = \b \d.$$
We shall also set 
$$
N := \gcd(n,k) = \b, \qquad n^\prime := n/N = \c, \qquad k^\prime := k / N = \d,  \qquad M:= \gcd(n,m+k).$$ 

The quadruple $(\a,\b,\c,\d)$ can be recovered via
$$\a=m, \qquad \b = \gcd(n,k), \qquad \c=n/\gcd(n,k), \qquad \d = k/\gcd(n,k).$$The involution $Q \longmapsto \tilde Q$ is associated with the following changes of indices:
$$(\a,\b,\c,\d) \longmapsto (\b,\a,\c,\d^{-1}),$$  
$$(n,m,k) \longmapsto (m n^\prime,N, \bar k^\prime m),$$
where $\bar k' := (k^\prime)^{-1}$ is taken modulo $n^\prime$.

Through \eqref{eq:q-barq} or \eqref{eq:q-tildeq}
we identify $q = (q_{a,b,c}) \in \mM$ with 
$Q = (q_{i,j}) \in \mathrm{Mat}_{n,m}(\C)$ or 
$\tilde{Q} = (\tilde q_{i,j}) \in \mathrm{Mat}_{m n^\prime,N}(\C)$. 
Correspondingly, the network $G$ has $m$ vertical wires (directed upwards) 
which form $m$ simple curves on the torus with the same homology class, and  
$n$ horizontal wires (directed to the right) which form $N$ 
simple curves on the torus with another homology class.
 

\subsection{Main results}\label{subsec:main}

Let $\mathcal{M} \simeq \C^{\a \b \c} = \C^{mn}$ be the phase space 
where the ring $\mathcal{O}(\mM)$ of regular functions on $\mM$ is 
generated by $q_{i,j} ~(i \in [m], j \in [n])$.  In \S\ref{sec:Lax}, we shall define a map $\psi: \mathcal{M} \to \C[x,y]$.  Every coefficient of $f(x,y) = \psi(q)$ is a regular function on $\mM$.  We then have (see Corollary \ref{cor:L-WW}) the following result.

\begin{statement}
The actions of the affine symmetric groups $W$ and $\tilde W$ on $\mathcal{M}$ preserve 
each fiber $\psi^{-1}(f)$ for $f \in \psi(\mM)$.  In particular, the coefficients of $f(x,y)$ are integrals of motion of the commuting $\Z^m$ and $\Z^N$ actions.
\end{statement}
In \S \ref{sec:spectral}, we give a combinatorial description of every coefficient of $f$, as generating functions of path families on a network on the torus.

Now fix a generic $f \in \psi(\mM)$ such that the affine plane curve  
$\{(x,y) \mid f(x,y)=0\} \subset \C^2$ is smooth (except for $(0,0)$), 
and let $C_f$ be the smooth completion of the affine curve.
We shall call $C_f$ the {\it spectral curve}.  In \S\ref{sec:Lax}, we define distinguished {\em special points},  $P$, $A_u ~(u \in [m])$ and $O_u~(u \in [N])$ on $C_f$.  We apply the results and techniques of van Moerbeke and Mumford \cite{vMM} to establish the following; see Theorem \ref{thm:phi}.

\begin{statement}
Fix a generic $f \in \psi(\mM)$.  There is an injection
$$
\phi : \psi^{-1}(f) \hookrightarrow \Pic^{g}(C_f) \times \mathcal{S}_f \times R_O \times R_A,
$$
where 
\begin{enumerate}
\item
$g$ is the genus of $C_f$, and $\Pic^{g}(C_f)$ is the Picard group of $C_f$,
of degree $g$;
\item
$\mathcal{S}_f \simeq (\C^\ast)^{M-1} \subset (\C^\ast)^M$;
\item
$R_O$ is a finite set of cardinality $N!$, identified with the $N!$ orderings of $\{O_1,O_2,\ldots,O_N\}$; and
\item
$R_A$ is a finite set of cardinality $m!$, identified with the $m!$ orderings of $\{A_1,A_2,\ldots,A_m\}$.
\end{enumerate}
\end{statement}
In fact, our Theorem \ref{thm:phi} identifies the image of $\phi$.

Define the divisors $\mathcal{A}_u := uP -\sum_{i=1}^u A_i$ and $\mathcal{O}_u := uP -\sum_{j=N-u+1}^N O_j$.
The commuting time evolutions $e_u$ \eqref{eq:Za-action} 
and $\tilde e_u$ \eqref{eq:Zm-action} can be described as follows
(Theorem \ref{thm:finite} and Theorem \ref{thm:commuting-actions}).

\begin{statement}
Suppose $\phi(q) = ([\mathcal{D}], (c_1,\ldots,c_M),O,A)$.  Then we have 
\begin{align*}
  &\phi(e_u (q)) = ([\mathcal{D} - \mathcal{A}_u], 
  (c_{u+1},\ldots,c_M,c_1,\ldots,c_u),O,A) \quad \text{for $u=1,\ldots,m$},
  \\ 
  &\phi(\tilde e_u (q)) = ([\mathcal{D} + \mathcal{O}_u],
  (c_{M-u+1},\ldots,c_M,c_1,\ldots,c_{M-u}),O,A) \quad \text{for $u=1,\ldots,N$}.
\end{align*}
Furthermore, the finite symmetric subgroups $\mS_N \subset \tilde W$ and $\mS_m \subset W$ act naturally on $R_O$ and $R_A$ respectively and do not affect the rest of the spectral data.
\end{statement}
In other words, the time evolutions $e_u$ and $\tilde e_u$ are linearized on $\Pic^{g}(C_f)$.  Our approach to Theorem \ref{thm:commuting-actions} is similar to that of Iwao \cite{Iwao07}.

When $N = 1$, we give in Theorem \ref{thm:N=1} an explicit formula for the inverse to the map $\phi$, which also explicitly solves the initial value problem for our dynamics.  For $t \in \Z^m$, let $(q^t_{j,i})$ denote the point in the phase space after time evolution in the direction $t$.
\begin{statement}
When $N=1$, we have a formula
$$  
  q^t_{j,i}
  =
  C \,
  \frac{\theta_{i}^{t+\be_{j}}(P) \, \theta_{i-1}^{t+\be_{j-1}}(P)}
  {\theta_{i}^{t+\be_{j-1}}(P) \, \theta_{i-1}^{t+\be_{j}}(P)},
$$
where $\theta_i^t(P)$ is a particular value of the Riemann theta function, and $C$ is a constant depending only on $(c_1,\ldots,c_M)$, $O$, and $A$.
\end{statement}

\subsection{Examples}
\label{subsec:example}

\subsubsection{Discrete Toda lattice}

Let $(\a, \b, \c, \d)=(2,1,n,n-1)$ (i.e. $(n,m,k) = (n,2,n-1)$).
We set $q_{i,j} = q_{a,1,c}$ for $i=a, ~j=c$, and 
regard $q := (q_{i,j})_{i \in [2], j \in [n]}$ as a coordinate
of $\mM \simeq \C^{2n}$. 
The resulting $\Z^2$-action on $\mM$ 
is generated by ${e}_1$ and ${e}_2$,
where $e_2$ simply acts on $\mM$  as
$ e_2 (q_{i,j}) = q_{i,j+1}$.
As for $e_1$, when we define $q^t :=  e_1^t(q)$ for $q \in \mM$,
the action of $e_1$ is rewritten as a system of 
difference equations:
\begin{align}\label{eq:d-Toda}
  \begin{cases}
  q_{1,j}^{t+1} q_{2,j}^{t+1} = q_{2,j}^t q_{1,j+1}^t,
  \\
  q_{1,j+1}^{t+1} + q_{2,j}^{t+1} = q_{2,j+1}^t +  q_{1,j+1}^t.
  \end{cases}
\end{align}
This is the discretization of $n$-periodic Toda lattice equation 
studied in \cite{HTI}. 
Actually, we recover the original Toda lattice equation
$\frac{d^2}{dt^2}x_j = \mathrm{e}^{x_{j+1}-x_j} - \mathrm{e}^{x_{j}-x_{j-1}}$
by setting $q_{1,j}^t = 1 + \delta \frac{d}{dt}x_j$ and
$q_{2,j}^t = \delta^2 \mathrm{e}^{x_{j+1}-x_j}$, 
and taking the limit $\delta \to 0$.
Here we set $q_{\ast,j}^t = q_{\ast,j}(\delta t)$ for $\ast = 1,2$
and $x_j = x_j(t)$.

A simple generalization of the discrete Toda lattice is
the case of $(\a, \b, \c, \d)=(m,1,n,n-1)$.
Its initial value problem was studied by Iwao \cite{Iwao07,Iwao10} 
by applying \cite{vMM}.
He also studied the similar problem in the case of 
$(\a, \b, \c, \d) = (m,1,n,0)$ in \cite{Iwao9}.

\subsubsection{Tropicalization and tableaux}\label{sec:tab}

Let $T_1$ and $T_2$ be two semistandard tableaux of rectangular shapes. Define $T_1 \otimes T_2$ to be the concatenation of the two into a skew tableaux by placing $T_2$ North-East of 
$T_1$.   Let $\jdt(T_1 \otimes T_2)$ denote the straight shape tableau obtained by performing jeu-de-taquin on $T_1 \otimes T_2$.  The following result is well-known.

\begin{lem}
There is a unique pair of rectangular semistandard tableaux $T_1'$ and $T_2'$ such that $T_i'$ has the same shape as $T_i$ (for $i = 1,2$), and
 $\jdt(T_1 \otimes T_2) = \jdt(T_2' \otimes T_1')$. 
\end{lem}

\begin{example}
Suppose 
$$
T_1 = \tableau[sY]{1&2\\3&3} \qquad \text{and} \qquad T_2 = \tableau[sY]{1&2&3}.$$ 
Then one has 
$$
T_2' = \tableau[sY]{1&3&3}\qquad \text{and} \qquad T_1' = \tableau[sY]{1&2\\2&3}.$$
\end{example}

The transformation $R(T_1,T_2)=(T_2',T_1')$ is called the {\it {combinatorial $R$-matrix}}.  It appears in the theory of crystal graphs as the isomorphism map between tensor products of Kirillov-Reshetikhin 
crystals. The following property is well-known.

\begin{thm}
The combinatorial $R$-matrix is an involution. Furthermore, it satisfies the braid relation:
 $$(R \otimes Id)(Id \otimes R)(R \otimes Id) = (Id \otimes R)(R \otimes Id)(Id \otimes R).$$
\end{thm}

Now, let $\{q_{a,b,c}\}_{a \in [\a], b \in [\b], c \in [\c]}$ be an array of nonnegative integers.  Define $Q = (q_{i,j})$ and $\tilde Q = (\tilde q_{i,j})$ as before:
$$q_{i,j} = q_{a,b,c} \text { for } i=a, \; j=b+\b(c-1)$$ and $$\tilde q_{ij} = q_{a,b,c} \text { for } i=\b+1-b, \; j=\a \d^{-1} c - a +1.$$
We create single-row tableaux from $Q$ and $\tilde Q$ as follows. 
For each $i \in [\a]$, let $T_i$ be the single-row tableau with $q_{i,j}$ $j$-s.  For each $i \in [\b]$, let $\tilde T_i$ be the single-row tableau with $\tilde q_{i,j}$ $j$-s.  We view $T =  T_1 \otimes \dotsc \otimes T_\a$ and $\tilde T = \tilde T_1 \otimes \dotsc \otimes \tilde T_\b$ as tropical analogues of $Q$ and $\tilde Q$.

Let us define an action of $W$ on $q_{a,b,c}$ as follows. For $T = T_1 \otimes \dotsc \otimes T_{\a}$ and $1 \leq \ell \leq \a-1$ let
$$s_{\ell}(T) = T_1 \otimes \dotsc \otimes T_{\ell+1}' \otimes T_{\ell}' \otimes \dotsc \otimes T_{\a}.$$  
That is, we apply the combinatorial $R$-matrix to the $\ell$-th and $(\ell+1)$-st factors. In addition, let $\tilde \pi$ act on $T$ by applying Schutzenberger's {\it {promotion}} operator \cite{Sch} 
to each factor $T_i$.  Similarly, we define $\tilde s_\ell$ and $\pi$ acting on $\tilde T$.

\begin{thm}
The operators $s_{\ell}$ and $\pi$ form an action of the extended affine symmetric group $W$ on $q_{a,b,c}$. Similarly,
operators $\tilde s_{\ell}$ and $\tilde \pi$ form an action of the extended affine symmetric group $\tilde W$ on $q_{a,b,c}$. These two actions commute.
These two operations are the tropicalizations of the rational actions of Theorem \ref{thm:dynamics}.
\end{thm}

Here tropicalization refers to the formal operation of substitution
$$\C \mapsto \Z, \qquad \times \mapsto +, \qquad \div \mapsto -, \qquad + \mapsto \min.$$
Under this substitution, the geometric $R$-matrix becomes a piecewise-linear involution of $\Z^n \times \Z^n$, which is the combinatorial $R$-matrix. 
In other words, the $W$ dynamics we are considering in this paper is a birational lift of the dynamics of repeated application of the combinatorial $R$-matrix on a sequence of $\a$ single-row tableaux, arranged in a circle. 

\begin{example}
Let $(\a,\b,\c,\d)=(2,2,2,1)$ as in Example \ref{ex:2221}. Take $T = 22222344 \otimes 1234$. Then $\tilde T = 122222344 \otimes 134$.  We have
$$ s_1(T) = 2234 \otimes 12222344, s_1(\tilde T) = 111122334 \otimes 134;$$
$$\tilde \pi(T) = 11111233 \otimes 1234, \tilde \pi(\tilde T) = 123 \otimes 122222344;$$
$$\tilde s_1(T) = 11112334 \otimes 1233, \tilde s_1(\tilde T) = 124 \otimes 122223344;$$
$$\pi (T) = 1234 \otimes 22222344, \pi(\tilde T) = 111112334 \otimes 234.$$
\end{example}

\begin{remark}
 The corresponding commuting crystal actions in the case $\c=1$ were considered by Lascoux in \cite{Las} and by Berenstein-Kazhdan in \cite{BK}.
\end{remark}

\subsubsection{Box-ball systems}

The box-ball system is an integrable cellular automaton introduced by Takahashi
and Satsuma \cite{TS}. It is described by an algorithm to move finitely many 
balls in an infinite number of boxes aligned on a line,
where a consecutive array of occupied boxes is regarded as a {\it soliton}.   
This system is related to both of the previous two examples;
the global movements of solitons are equivalent to 
the tropicalization of the discrete Toda lattice \eqref{eq:d-Toda}. 
The symmetry of the system is explained by the crystal base theory, 
where the dynamics of bolls is induced by the action of the combinatorial $R$-matrix.

See \cite{IKT} for a comprehensive review of 
the combinatorial and tropical aspects of the box-ball system.

\section{Lax matrix and spectral curve}
\label{sec:Lax}
\subsection{Lax matrix}
Fix integers $n,m,k$ such that $n \geq 2$, $m \geq 1$ and
$1 \leq k \leq n$.  From now on we shall mainly use $q := (q_{ij})_{i \in [m], j \in [n]}$ as a coordinate
of the phase space $\mM$.
We identify $q \in \mM$ with an $m$-tuple of $n$ by $n$ matrices 
$Q:=(Q_i(x))_{i=1,\ldots,m}$ with a spectral parameter $x$,
where 
\begin{align}\label{eq:Q}
&Q_i(x) := \left(\begin{array}{cccc} 
q_{i1} & 0 & 0 &x\\
1&q_{i2} &0 &0 \\
0&\ddots&\ddots&0 \\
0&0&1&q_{in}
\end{array} \right).
\end{align}

Let $\mL$ be the set of $n$ by $\infty$ scalar matrices  
$A := (a_{ij})_{1 \leq i \leq n,  \,j \in \Z}$ satisfying the following conditions:
\begin{align}
  a_{ij} = 
  \begin{cases} 
   1 & j-i = -m-k \\
   a_{ij} \in \C & -m-k+1 \leq j-i \leq -k \\
   0 & \text{otherwise}.
   \end{cases} 
\end{align}
In particular, $A$ has finitely many nonzero entries.  For $A \in \mL$, we define an $n$ by $n$ matrix 
$L(A;x) = (l(x)_{ij})_{1 \leq i,j \leq n}$ by
\begin{align}\label{eq:L-A}
  l(x)_{ij} = \sum_{\ell \in \Z} a_{i,j- \ell n} \, x^\ell.
\end{align}  
We may identify $\mL$ with $\C^{mn}$, and identify $A \in \mL$ with $L(A;x)$.

Now define a map $\alpha : \mathcal{M} \to \mL$ by
$$
\alpha: Q = (Q_1,Q_2,\ldots,Q_m) \longmapsto L(x) := Q_1(x) Q_2(x) \cdots Q_m(x)P(x)^k,
$$
where
\begin{align}\label{eq:P}
P(x) := \left(\begin{array}{cccc} 
0 & 0 & 0 &x\\
1&0 &0 &0 \\
0&\ddots & \ddots&0 \\
0&0&1&0
\end{array} \right).
\end{align}
We call $L(x)$ the \emph{Lax matrix}.  Our approach to the study of $L(x)$ is close to that of van Moerbeke and Mumford \cite{vMM}.  

We give a combinatorial description of the Lax matrix by using {\em highway paths} on the network $G$. 
We introduce the ``$x$-line'' as illustrated in Figure \ref{fig:toda1}, 
where the top $k$ right ends of horizontal wires cross the $x$-line (the top edge) and come out on the other side. 
There are $n$ sources labeled $1,2,\ldots, n$ on the left and $n$ sinks labeled $1,2,\ldots,n$ on the right.  Each sink $j$ is connected to the source $j+k \mod n$ by a wire, as illustrated in the figure.  There are $mn$ intersection points between the vertical wires and the horizontal wires, which we call the {\it $(i,j)$-crossroads}, where $i = 1,2,\ldots,m$ indexes the vertical wire, and $j = 1,2,\ldots,n$ indexes the horizontal wire.  
Let us denote by $G'$ the \emph{cylindrical network} obtained by gluing only the upper and lower edges of Figure \ref{fig:toda1}.  In the torus network $G$, source $i$ and sink $i$ are the same point.  In the cylindrical network $G'$, source $i$ and sink $i$ are distinct points.

\begin{figure}[h!]
    \begin{center}
    \scalebox{0.8}{\input{toda1.pstex_t}}
    \end{center}
    \caption{The toric network $G$, with a coil and a snake path shown.}
    \label{fig:toda1}
\end{figure}

Let us introduce the notion of highway paths following \cite{LP}.  
A {\it highway path} $p$ is a directed path in the network $G$ (or $G'$) with the following property: at any of the crossroads, if the path is traveling upwards, it must turn right.  We shall only consider highway paths that start at one of the sources $1,2,\ldots, n$ and end at one of the sinks $1,2,\ldots,n$.  
The weight $\wt(p)$ of a highway path $p$ is defined as follows.
Every time $p$ passes the $x$-line it picks up the weight $x$.  Every time $p$ goes through the $(i,j)$-crossroad, it picks up the weight $q_{ij}$ or $1$, according to Figure~\ref{fig:highway}.  The weight $\wt(p)$ is the product of all these weights.  The condition that a highway path must turn right when travelling up into a crossroad is indicated in the Figure~\ref{fig:highway}: we can think of such a turn as giving weight 0.  
Finally, a highway path $p$ may use no edges.  In this case, we consider the path $p$ to start at some source $i$, and end at the sink $i$.  We declare such paths to be {\it abrupt}, and have weight $\wt(p) = -y$.

\begin{figure}[h!]
\unitlength=0.9mm
\begin{picture}(80,30)(20,0)

\multiput(0,15)(30,0){4}{\vector(1,0){20}}
\multiput(10,5)(30,0){4}{\vector(0,1){20}}

\thicklines
\linethickness{0.45mm}
\put(0,15){\line(1,0){10}}
\put(10,15){\vector(0,1){10}}

\put(30,15){\vector(1,0){20}}

\put(70,5){\line(0,1){10}}
\put(70,15){\vector(1,0){10}}

\put(100,5){\vector(0,1){20}}


\put(-20,-2){weights:}
\put(9,-2){$1$}
\put(39,-2){$q_{ij}$}
\put(69,-2){$1$}
\put(99,-2){$0$}
\end{picture}
    \caption{Highway paths}
    \label{fig:highway}
\end{figure}

The following lemma gives a highway-path description of the Lax matrix,
which is a variant of the results of \cite{LP}. 
It follows directly from the definitions.

\begin{lem}\label{lem:entry}
Let $L(x) = \alpha(q)$ for $q \in \mM$. 
For $1 \leq i,j \leq n$, we have
$$
\mbox{$(i,j)$-th entry of $L(x)-y$} = \sum_p \wt(p),
$$ 
where the summation is over highway paths in $G'$ from source $i$ to sink $j$.
\end{lem}

\subsection{Spectral curve and Newton polygon}\label{sec:spectral}
Define a map $\psi : \mM \to \C[x,y]$  
as the composition of $\alpha: \mM \to \mL$ and the map $\beta: \mL \to \C[x,y]$,
$$
  Q = (Q_1,Q_2,\ldots,Q_m) \stackrel{\alpha}{\longmapsto} L(x) = Q_1(x) Q_2(x) \cdots Q_m(x)P(x)^k 
  \stackrel{\beta}{\longmapsto} \det(L(x) - y).
$$
Consequently, for $Q = (Q_1,\ldots,Q_m) \in \mM$
we have an affine plane curve $C'_{\psi(Q)}$ in $\C^2$, given by 
the zeros of $\psi(Q)$.  We call this curve the \emph{spectral curve}.

Each term of $\psi(Q)$ corresponds to
the weight of specific highway paths as follows.
We say that a pair of paths is {\it noncrossing} if no edge is used twice,
and that a family of paths is noncrossing if every pair of paths is 
noncrossing. 
Suppose $\p = \{p_1,p_2,\ldots,p_n\}$ is an unordered noncrossing family of $n$ paths in $G'$ using all the sources and all the sinks.  
The non-abrupt paths in $\p$ induce a bijection of a subset $S \subset [n]$ with itself.  We let $\sign(\p)$ denote the sign of this permutation.
The following theorem is a reformulation \cite{MIT}.  
In our language, the proof is very similar to \cite[Theorem 3.5]{TP}.

\begin{thm} \label{thm:mit}
We have
$$
f(x,y)=\det(L(x)-y) = \sum_{\p = \{p_1,p_2,\ldots,p_n\}} {\sign(\p)} \wt(p_1) \wt(p_2) \cdots \wt(p_n),
$$
where the summation is over noncrossing (unordered) families of $n$ paths in $G'$ using all the sources and all the sinks.
In other words, the coefficient $x^ay^b$ in $f(x,y)=\det(L(x)-y)$ counts (with weights) families of $n$ paths that
\begin{itemize}
 \item do not cross each other;
 \item cross the $x$-line exactly $a$ times;
 \item contain exactly $b$ abrupt paths and $n-b$ non-abrupt paths.
\end{itemize}
The overall resulting sign of the monomial $x^ay^b$ is $(-1)^{(n-b-1)a+b}$.
\end{thm}

For $f(x,y) = \sum_{i,j} a_{i,j} x^i y^j \in \C[x,y]$,
we write $N(f) \subset \R^2$ for the Newton polygon of $f$.
This is defined to be the convex hull of the points $\{(i,j) \mid a_{ij} \neq 0\}$.  
It is important for us to identify the {\it lower hull} and {\it upper hull} of $N(f)$.  The former (resp. latter) is the set of edges of $N(f)$ such that the points directly below (resp. above) these edges do not belong to $N(f)$.  We exclude vertical or horizontal edges from the lower and upper hull.

\begin{prop} \label{prop:hull}
For generic $Q = (Q_1,\ldots,Q_m) \in \mM$, the Newton polygon $N(\psi(Q))$ is the triangle with
vertices $(0,n)$, $(k,0)$ and $(k+m,0)$.
In $N(\psi(Q))$, the lower hull (resp. upper hull)  
consists of one edge with vertices $(k,0)$ and $(0,n)$
(resp. $(k+m,0)$ to $(0,n)$).
\end{prop}
See \S~\ref{subsec:prop:hull} for the proof.

\subsection{Special points on the spectral curve}\label{subsec:special-pts}

For $f(x,y) \in  \C[x,y]$ an irreducible polynomial, let $C'_f= \{(x,y) \mid f(x,y) = 0\} \subset \C^2$ be the corresponding plane curve, and let
$\overline{C'_f} \subset \P^2(\C)$ denote its closure.  Let $C_f$ be the normalization of $\overline{C'_f}$, with a map $C_f \to \overline{C'_f}$ that is a resolution of singularities.
We declare points on $C_f$ to be {\it special} if either (1) either $x$ or $y$ is 0, or (2) the point does not lie over $C'_f$ (that is, $x$ or $y$ is $\infty$).  

For $f(x,y) \in \psi(\mathcal{M})$,
define a polynomial $g_f(x,y)$ (resp. $h_f(x,y)$) by 
$f(x,y) = g_f(x,y) + {\rm other~ terms}$ 
(resp. $f(x,y) = h_f(x,y) + {\rm other~ terms}$), 
where $g_f(x,y)$ (resp. $h_f(x,y)$) consists of 
those monomials lying on the lower hull (resp. the upper hull) of $N(f)$.
For this $f(x,y)$ we also define 
\begin{align}\label{eq:fc}
f_c:= \sum_{(j,i) \in L_c} f_{i,j},
\end{align}
where 
$$
L_c = \{(j,i) 
~|~ (m+k) i + n j = n (m+k)-M \}.
$$  
We define $\sigma_r \in \mathcal{O}(\mathcal{M})$ by 
\begin{align}\label{def:sigma_r}
  \sigma_r(q) = \prod_{i=1}^m \prod_{j=0}^{n'-1} q_{i,r+jN},
\end{align}
for $1 \leq r \leq N$, and  
define $\epsilon_r \in \mathcal{O}(\mathcal{M})$ by
\begin{align}\label{def:epsilon_r}
\epsilon_r(q):= \prod_{j=1}^n q_{rj}
\end{align}
for $1 \leq r \leq m$.  

\begin{remark}
 Despite the seeming dissimilarity, $\sigma_r(q)$ and $\epsilon_r(q)$ have the same nature. Indeed, visualize variables $q_{ij}$ as associated to crossings of two families of parallel wires on a torus, as it was done in Section \ref{sec:netw}.
 Then $\sigma_r(q)$ is the product of parameters on the $r$-th horizontal wire (out of $N$), while $\epsilon_r(q)$ is the product of parameters on the $r$-th vertical wire (out of $m$). In particular, the symmetry between $Q$ and $\tilde Q$ 
 from Section \ref{sec:netw} switches the $\sigma_r(q)$ and the $\epsilon_r(q)$ into each other. 
\end{remark}

\begin{lem}\label{lem:Q}
For $f(x,y) \in \psi(\mathcal{M})$, we have
\begin{align}
&g_f(x,y) = \prod_{r=1}^N ((-y)^{n'} + \sigma_r x^{k'}),
\label{eq:lower-f}
\\
\label{eq:upper-f}
&h_f(x,y) = ((-y)^{\frac{n}{M}}+x^{\frac{k+m}{M}})^M.
\end{align}
\end{lem}

See \S~\ref{subsec:lem:Q} for the proof.

In the following, for $f(x,y) = \psi(q), ~q \in \mM$ we study the special points of $C_f$ 
related to polynomials $g_f(x,y)$, $h_f(x,y)$ or $f(x,0)$. 

We have $f(x,0) = \det P(x)^k \prod_{i=1}^m \det Q_i(x)$, 
where $\det Q_i(x) = \prod_{j=1}^n q_{ij} + (-1)^{n-1} x = \epsilon_i(q) + (-1)^{n-1} x$.  Thus the roots of $f(x,0)=0$ are exactly the $(-1)^{n} \epsilon_r(q)$, where $\epsilon_r(q)$ is defined by \eqref{def:epsilon_r}.  It is also clear that $\epsilon_1,\ldots,\epsilon_m$ depend only on $f(x,y)$.
The following lemma is obtained immediately.
 
\begin{lem}\label{lem:A}
Suppose that $\epsilon_1,\ldots,\epsilon_m$ are distinct and nonzero.  Then there are $m$ special points $A_i =  ((-1)^{n} \epsilon_i(q),0)$ on $C_f$ with $y = 0$ and $x$ is nonzero.
\end{lem}


The point $(0,0)$ on $C'_f$ is usually singular, whenever $k \geq 2$.

\begin{lem}\label{lem:O}
Suppose $\sigma_1,\ldots,\sigma_N$ are distinct and nonzero.  Then there are $N$ points of $C_f$ lying over $(0,0) \in C'_f$.
\end{lem}
\begin{proof}
By \eqref{eq:lower-f}, the meromorphic function $y^{n'}/x^{k'}$ takes the $N$ distinct values 
$(-1)^{n'+1}\sigma_r$ for $r = 1,2,\ldots,N$ as $(x,y) \to (0,0)$.   
So there are at least $N$ points on $C_f$.  
On the other hand, looking at $N(f)$ we see that analytically near $(0,0)$, the polynomial $f$ can factor into at most $N$ pieces.
\end{proof}
Let $O_1,O_2,\ldots, O_N$ denote the the special points of Lemma \ref{lem:O}.  
Near $O_r$ there is a local coordinate $u$ such that  \begin{equation*}
(x,y) \sim (u^{n'},-(-\sigma_r)^{1/n'}u^{k'}).
\end{equation*}

It turns out that $C_f' \subset \P^2(\C)$ has only one point at $\infty$.  
Due to the polynomial $h_f(x,y)$, in homogeneous coordinates, this point is
$$
P' = \begin{cases} [1:0:0] &\mbox{if  $n > k+m$}, \\
[1:1:0] & \mbox{if $n = k+m$},\\
[0:1:0] & \mbox{if $n < k+m$}.
\end{cases}
$$
To compute this, we first homogenize $f(x,y)$ to get $F(x,y,z) = z^d f(x/z,y/z)$where $d := \max(n,m+k)$. Then we solve $F(x,y,0) = 0$.  
Recall that $f_c$ was defined in \eqref{eq:fc}. 

\begin{lem} \label{lem:inf}
Suppose that $f_c \neq 0$.  Then there is a unique point $P \in C_f$ lying over $P'$.
\end{lem}

See \S~\ref{subsec:lem:inf} for the proof.

\subsection{A good condition for the spectral curve}
Let $V_{n,m,k}$ be the subspace of $\C[x,y]$ given by
\begin{align}\label{eq:Vnmk}
  V_{n,m,k} 
  = 
  \left\{((-y)^{\frac{n}{M}}+x^{\frac{k+m}{M}})^M 
  + \sum_{i=0}^{n-1} y^{i} f_i(x) ~\Big|~
  f_i(x) = 
  \sum_{(j,i) \in L_{n,m,k}} 
  f_{i,j} x^j \in \C[x] \right\},
\end{align}
where we write $L_{n,m,k}$ for the set of lattice points in the 
convex hull of $\{(k,0),(0,n),(k+m,0) \}$ but not on the 
upper hull.  By Proposition \ref{prop:hull} and Lemma \ref{lem:Q}, 
we have that $\psi(\mM) \subset V_{n,m,k}$. 

\begin{definition}\label{def:P:curve}
Define the subset $\mathcal{V} \subset V_{n,m,k}$ as the set of 
$f(x,y) \in V_{n,m,k}$, satisfying the conditions:
\begin{enumerate}
\item
$f(x,y)$ is irreducible;
\item 
$C'_f$ is smooth; 
\item
$f_c \neq 0$;
\end{enumerate}
and such that the special points on $C_f$ consist exactly of:
\begin{enumerate}
\item[(4)]
$m$ distinct points $A_1,A_2,\ldots,A_m$ where $A_r :=  ((-1)^{n} \epsilon_i,0)$ and $\epsilon_i \neq 0$;
 \item[(5)]
$N$ distinct points $O_1,O_2,\ldots, O_N$ lying over $(0,0)$, where near $O_r$ 
there is a local coordinate $u$ such that  
\begin{align}\label{eq:o_r} 
(x,y) \sim (u^{n'},-(-\sigma_r)^{1/n'}u^{k'})
\end{align}
and $\sigma_r \neq 0$;
\item[(6)]
a single point $P$ lying over the line at infinity $\P^2(\C) \setminus \C^2$.
\end{enumerate}
\end{definition}

For $f \in \mathcal{V}$, the genus $g$ of $C_f$ is given by
\begin{align} \label{eq:genus}
 g = \frac{1}{2}\left( (n-1)m - M -N + 2 \right).
\end{align}
Indeed, it follows from Pick's formula and Proposition \ref{prop:hull} that the number of interior lattice points of $N(f)$ is equal to the right hand side.  This formula for the genus then follows from \cite[Corollary on p.6]{Kho}.  Alternatively, the genus can be computed using the Riemann-Hurwitz formula, as in \cite{vMM}.

\begin{prop}\label{P:curve}
For $f \in \mathcal{V}$, we have that $\beta^{-1}(f) \neq \emptyset$.
Moreover, the set $\oV := \mathcal{V} \cap \psi(\mathcal{M})$
is a Zariski-dense subset of $\psi(\mathcal{M})$.     
\end{prop}

We give the proof in \S~\ref{subsec:P:curve}.
Informally, the last statement of Proposition \ref{P:curve} states that most curves in $\psi(\mathcal{M})$ satisfy the ``niceness" conditions listed in Definition \ref{def:P:curve}.

From Definition~\ref{def:P:curve}, it follows that for $f \in \mathcal{V}$, the meromorphic functions
$x$ and $y$ on $C_f$ satisfy
\begin{align}
  (x) = n^\prime \sum_{i=1}^N O_i - n P, 
  \qquad 
  (y) = k^\prime \sum_{i=1}^N O_i + \sum_{i=1}^m A_i -(m+k) P.
\end{align}

\begin{remark}
The condition that $A_1,\ldots,A_m$ (resp. $O_1,O_2,\ldots, O_N$) are distinct imply that the quantities $\epsilon_1,\ldots,\epsilon_m$ (resp. $\sigma_1,\ldots,\sigma_N$) are distinct.  Many of our main results still apply after a modification even when this condition does not hold.
\end{remark}

\section{Proofs from Sections \ref{sec:netw} and \ref{sec:Lax}}
\label{sec:proofLax}

\subsection{Proof of Proposition \ref{prop:hull}}
\label{subsec:prop:hull}

We use the interpretation of $f(x,y)$ given by Theorem \ref{thm:mit}. Let $\p = (p_1,\ldots,p_n)$ be a family of noncrossing highway paths in $G'$, as in Theorem \ref{thm:mit}, and let $\wt(\p) = \wt(p_1) \cdots \wt(p_n)$.

There are $k+m$ opportunities for a highway path in $\p$ to pick up the weight $x$: from each of the $m$ vertical wires and from the $k$ horizontal wires crossing $x$-line.  For a fixed power $y^b$ (where $b = 0,1,\ldots,n$), or equivalently a fixed number of abrupt paths, we will bound the maximal and the minimal possible value of the exponent $a$ of $x$.

  \begin{figure}[h!]
    \begin{center}
    \scalebox{0.8}{\input{toda2.pstex_t}}
    \end{center}
    \caption{The red path crosses $x$-line less than the green one.}
    \label{fig:toda2}
\end{figure}

The first key observation is that we can consider $\p$ to be a family of noncrossing closed cycles on the toric network $G$.  An abrupt path $p$ is simply the ``cycle" that starts and ends at the vertex $i$ ($=$ source $i$ and sink $i$ identified) in $G$ and does not move.

Lift such a highway cycle $C$ to the universal cover, as shown in Figure \ref{fig:toda2}. We obtain a path that starts and ends at two vertices labeled with the same integer $i \in \{1,2,\ldots,n\}$, but $\ell$ periods (of $m$ vertical lines each) to the right of the original source.  More precisely, if $C$ is obtained by gluing paths $p_{i_1},\ldots,p_{i_\ell}$ in $G'$, then $C$ has length $\ell$.  The $x$-line now has a staircase-like shape as shown in Figure \ref{fig:toda2}.  We claim that if $C$ has length $\ell$, then it crosses the $x$-line at least $\ell k/n$ times and at most $\ell (k+m)/n$ times.

To see this, observe that if we lift the ending vertex several periods (consisting of $n$ horizontal wires) up, we preserve the length and we increase the 
number of crossings with the $x$-line.  Similarly, if we lower the ending vertex several periods down, we preserve the length and we decrease the number of crossings of $x$-line.  Thus, the smallest and the largest ratios of the quantity (crosses of $x$-line/length) is achieved for the lowest and the highest possible highway paths, shown in Figure \ref{fig:toda2} in purple and green.
The former is the horizontal path, while the latter is an alternating right-up staircase path. For these paths, the ratios are exactly $k/n$ and $(k+m)/n$.

\subsection{Proof of Lemma~\ref{lem:Q}}
\label{subsec:lem:Q}

One of the consequences of the proof of Proposition \ref{prop:hull} is that the monomials for which the lower bound $k/n$ of the ratio is reached are 
the ones coming from
horizontal highway paths on the universal cover.   Let us call each such closed cycle a {\it {coil}}. For example, the purple line in Figure \ref{fig:toda2} represents a coil passing through source $i$, as well as passing through the $n'$ other vertices between $1$ and $n$ that have residue $i$ 
modulo $N = \gcd(k,n)$. Thus the terms of $g(x,y)$ are formed in the following way: for each of the $N$ coils we decide whether to include it 
into our family of paths, or to make all paths starting at its sources abrupt. The second choice corresponds to the contribution $(-y)^{n'}$.  The first choice gives 
$(\prod_{i=1}^m \prod_{j=1}^{n'} q_{i,r+jN}) x^{k'}$, which is the weight of that coil. Thus, the $r$-th coil contributes 
the factor of $\left((\prod_{i=1}^m \prod_{j=1}^{n'} q_{i,r+jN}) x^{k'} + (-y)^{n'}\right)$, and \eqref{eq:lower-f} follows. 

Similarly, the monomials for which the upper bound $(k+m)/n$ of the ratio is reached are 
the ones coming from right-up staircase paths on the universal cover.  Let us call each such closed cycle a {\it {snake path}}.  For example, the 
green line in Figure \ref{fig:toda2} represents a snake path passing through source $i$, as well as through all the $n/M$ other vertices between $1$ and $n$ that have residue $i$ 
modulo $M = \gcd(m+k,n)$.  Thus the part contributing to the upper hull of $f(x,y)$ is a product of factors $x^{(m+k)/M}+(-y)^{n/M}$, where 
the term $(-y)^{n/M}$ corresponds to choosing to have abrupt paths, while $x^{(m+k)/M}$ corresponds to choosing to have the snake path. 
Thus we obtain \eqref{eq:upper-f}.

\subsection{Proof of Lemma~\ref{lem:inf}}
\label{subsec:lem:inf}

We analyze the singularity at $P' \in \P^2(\C)$ on the chart 
$y \neq 0$, that is, $\{[x:1:z] ~|~ x,z \in \C\}$.
Let $\tilde{C}_{\overline{f}}$ be the affine curve given by $\overline{f}(x,z) := F(x,1,z) = 0$.    
 
(i) $n < k+m$: using \eqref{eq:upper-f}
we can write $\overline{f}(x,z)$ as
$$
  \overline{f}(x,z)
  = ((-1)^{\frac{n}{M}}z^{\frac{k+m-n}{M}}
    +x^{\frac{k+m}{M}})^M 
    + \sum_{j=1}^{\max(m,m+k-n)} z^j \overline{f}_j(x)
$$
where $\deg_x \overline{f}_j(x) \geq 1$.
Then, when $m+k-n = 1$, $\tilde{C}_{\overline{f}}$ is smooth at $(0,0)$ since 
$\partial \overline{f}/\partial z|_{(0,0)} = (-1)^{n/M} \neq 0$.
When $m+k-n \geq 2$, the point $(0,0)$ is singular.
Define $K := \C(x)[z]/(\overline{f}(x,z))$.
We consider all valuations $v: K \twoheadrightarrow \Z$ on $K$ satisfying that $v(x) > 0$ and $v(z) > 0$.  
We shall show that such valuations are unique, 
from which the uniqueness of $P$ follows.
Define $\overline{h}_f(x,z):=((-1)^{\frac{n}{M}}z^{\frac{k+m-n}{M}}
    +x^{\frac{k+m}{M}})^M$ which is the part of 
$\overline{f}(x,z)$ consisting of monomials
lying on the lower hull of $N(\overline{f})$,
corresponding to the upper hull of $N(f)$. 
Then we see that 
$v(\overline{f}(x,z)) = v(\overline{h}_f(x,z)) =
M \cdot v(x^{\frac{k+m}{M}}
+ (-1)^{\frac{n}{M}}z^{\frac{k+m-n}{M}})$.
To be consistent with $v(\overline{f}(x,z)) = v(0)= \infty$,
the only choice we have is 
$v(x) = \frac{k+m-n}{M}$ and $v(z) = \frac{k+m}{M}$.

(ii) $n = k+m$: 
we can write $\overline{f}(x,z)$ as
$$
  \overline{f}(x,z)
  = (x-1)^n + \sum_{k=1}^m z^k \overline{f}_k(x-1)
$$
where $\overline{f}_k(x-1) \in \C[x-1]$.
Thus $P'$ is smooth point on $C_{\overline{f}}$ if $\overline{f}_1(0) \neq 0$.
Looking at the Newton polygon $N(f)$ and \eqref{eq:fc}, it follows that 
$\overline{f}_1(0) = \sum_{i=0}^{n-1} f_{i,n-1-i} = f_c \neq 0$.

(iii) $n >k+m$: this case is nearly the same as the case $n < k+m$.

\subsection{Proof of Proposition \ref{P:curve}}\label{subsec:P:curve}
The first statement follows from Theorem~\ref{thm:eta} in the next section.  
In the following, we prove the second statement.  

We first show that $\mathcal{V}$ contains a Zariski-dense and open subset of $V_{n,m,k}$.  A Zariski-open subset of $V_{n,m,k}$ consists of $f(x,y)$ with Newton polygon $N(f)$ given by Proposition \ref{prop:hull}.  This polygon is not a non-trivial Minkowski sum of two other polygons, so $f(x,y)$ is irreducible.  Similarly the conditions that $C'_f$ is smooth and $f_c \neq 0$ are Zariski-open conditions on $V_{n,m,k}$.  By Lemma \ref{lem:inf}, $f_c \neq 0$ implies that (6) in Definition \ref{def:P:curve} holds.  The calculations in the proofs of Lemmas \ref{lem:A} and \ref{lem:O} then imply that $\mathcal{V}$ contains a Zariski-open and dense subset of $V_{n,m,k}$.

It thus suffices to show that $\psi(\mM)$ contains a Zariski-dense subset of $V_{n,m,k}$.  We shall use the following result.

\begin{lem}\label{L:mMmL}
The set $\alpha(\mM)$ is Zariski-dense in $\mL$.
\end{lem}
\begin{proof}
This follows from \cite[Proof of Theorem 4.1]{LPgeom} where it is shown that the map $\alpha:\mM \to \mL$ is generically a $m!$ to $1$ map between two spaces of dimension $mn$.
\end{proof}

By the first statement of Proposition \ref{P:curve}, we have $\mathcal{V} \subset \beta(\mL)$.  By Lemma \ref{L:mMmL}, $\psi(\mM) = \beta(\alpha(\mM))$ contains a Zariski-dense subset of $\mathcal{V}$.  It follows that $\psi(\mM)$ is Zariski-dense in $V_{n,m,k}$.  This completes the proof of the second statement of Proposition \ref{P:curve}.

\begin{remark}
The nonconstant coefficients $f_{i,j}$ of $f(x,y)$ can be pulled back to functions on $\mM$ or $\mL$.  A consequence of our proof is that $\psi(\mM)$ is Zariski-dense in $V_{n,m,k}$, and therefore the functions $f_{i,j}$ are algebraically independent on $\mM$ or $\mL$.  It would be interesting to obtain a direct proof of this.
\end{remark}

\subsection{Proof of Theorem \ref{thm:dynamics}}
\label{sec:dynamics_proof}

 We employ the technique first introduced in \cite{LP}. 
 \begin{figure}[h!]
    \begin{center}
    \input{wire11.pstex_t}
    \end{center}
    \caption{Crossing merging and crossing removal moves with vertex weights shown.}
    \label{fig:wire11}
\end{figure}
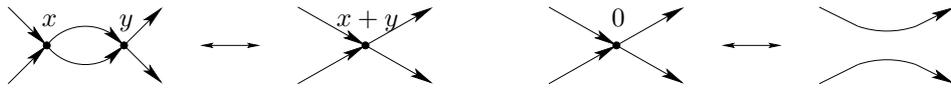
 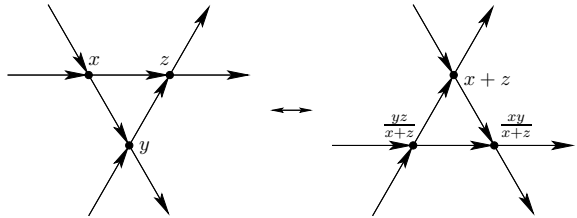
\begin{figure}[h!]
    \begin{center}
    \scalebox{.7}{\input{wire8.pstex_t}}
    \end{center}
    \caption{Yang-Baxter move with transformation of vertex weights shown.}
    \label{fig:wire8}
\end{figure}
Specifically, we realize the geometric $R$-matrix transformation, dubbed the {\it {whirl move}} in \cite{LP}, as a sequence of local transformations on our toric network. 
The transformations we shall employ are of three kinds: crossing merging/unmerging, crossing creation/removal, shown in Figure \ref{fig:wire11}, and Yang-Baxter move shown in Figure \ref{fig:wire8}.

The whirl move $R$ occurs between two parallel wires adjacent to each other and wrapping around a local part of the surface that is a cylinder. The way it is realized as a sequence of local moves is illustrated in 
Figure \ref{fig:wire20}. First a crossing is created with weight $0$, and split into two crossings of weight $p$ and $-p$. One of them is pushed through the wires crossing our two distinguished wires,
until it comes out on the other side. As it is proven in \cite[Theorem 6.2]{LP}, there is at most one non-zero value of $p$ for which the end result is again a pair of crossings of weights $p$ and $-p$,
and thus those two can be canceled out. The resulting action on the weights along two parallel wires is exactly the whirl move $R$, and does not depend on where the original auxiliary crossing was 
created. 
\begin{figure}[h!]
    \begin{center}
    \input{wire20a.pstex_t}
    \end{center}
    \caption{For a unique choice of the weight $p$, the weight that comes out on the other side after passing through all horizontal wires is also $p$; the resulting transformation of $x$ and $y$ is exactly the whirl move.}
    \label{fig:wire20}
\end{figure}
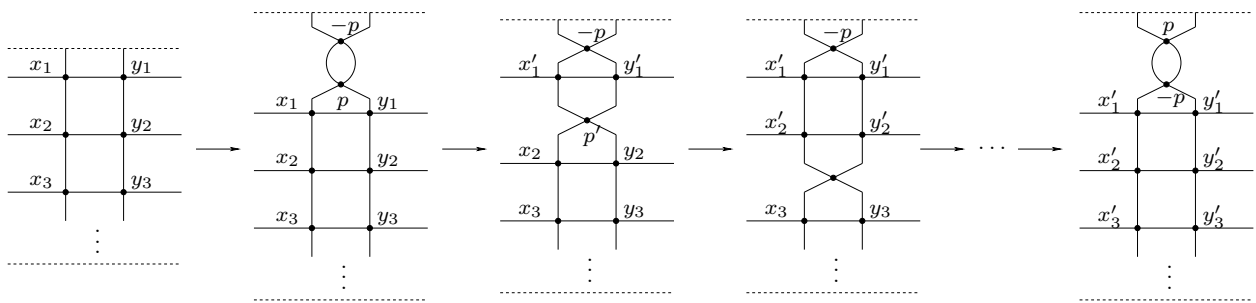
The value of $p$ does depend on the location $j$ where the new crossings are created, and is given by 
$$p = \frac{\prod y_{i} - \prod x_{i}}
{E(\boldsymbol{x}^{(j)},\boldsymbol{y}^{(j)})}.$$
Here the variables $x_i$ and $y_i$ are weights along the two wires as 
in Figure \ref{fig:wire20}, and $E(\boldsymbol{x}^{(j)},\boldsymbol{y}^{(j)})$
is the energy of the cyclically shifted vectors $\boldsymbol{x}^{(j)} = (x_{j+1},x_{j+2},\ldots,x_{j})$ and 
$\boldsymbol{y}^{(j)} = (y_{j+1},y_{j+2},\ldots,y_{j})$ 
as introduced in \S~\ref{sec:Rmatrix}.

Now, the fact that $R$ satisfies the braid move, that is, $s_{\ell}  s_{\ell+1}  s_{\ell}  = s_{\ell+1}  s_{\ell}  s_{\ell+1}$ and $\tilde s_{\ell} \tilde s_{\ell+1} \tilde s_{\ell}
  =\tilde s_{\ell+1} \tilde s_{\ell} \tilde s_{\ell+1}$ is shown in \cite[Theorem 6.6]{LP}. Indeed, these relations happen at a local part of the surface (in our case, torus) that looks like a cylinder.
  
\begin{figure}[h!]
    \begin{center}
    \scalebox{0.6}{\input{wire25a.pstex_t}}
    \end{center}
    \caption{}
    \label{fig:wire25}
\end{figure}  
  
On the other hand, the commutativity of the $s_{\ell}$ and the $\tilde s_{\ell}$ does not follow from \cite[Theorem 12.2]{LP}. This is because in \cite{LP} we only considered the case when the pairs of parallel wires
intersect once. In our case on the torus however, it is common to have horizontal and vertical wires intersect more than once: this happens for any $\d \not = 1$. The proof in such a situation is essentially the same. 

Indeed, if we have two pairs of parallel wires crossing as in Figure \ref{fig:wire25}, but possibly more than once, we can realize each of the two corresponding $R$-moves by a sequence of local moves as above. 
It is a local check that performing one of the two sequences does not change 
the value of 
$p = \frac{\prod y_i - \prod x_i}{E(\boldsymbol{x}^{(j)},\boldsymbol{y}^{(j)})}$ 
one needs to perform the other, because each of the 
$\prod x_i$, $\prod y_i$ and $E(\boldsymbol{x}^{(j)},\boldsymbol{y}^{(j)})$ is not changed. 
It is also a local check that the two sequences commute once the parameters $p$ for each are chosen, this is \cite[Proposition 3.4]{LP}.
Thus, commutativity follows.

\section{Eigenvector map}
\label{sec:eigenvector}
In this section, we fix $f \in \mathcal{V}$ (see Definition~\ref{def:P:curve}) 
and consider the corresponding smooth curve $C_f$.  
For $j \in \Z$, 
we set $O_j := O_r$
if $j \equiv r$ mod $N$.
Similarly, for $j \in \Z$, 
we set $A_j := A_r$
if $j \equiv r$ mod $m$.

\subsection{Generalities}
A divisor $D = \sum_i n_i P_i$ on an algebraic curve $C$ 
is a finite formal integer linear combination of points $P_i$ on $C$.  
We write $D \geq 0$ if $n_i \geq 0$, and say that $D$ is {\it positive} in this case.  
The degree of $D$ is given by $\deg(D) = \sum_i n_i$.

Given $h$ a meromorphic function on $C$, we let 
$(h) = (h)_0 - (h)_\infty$ be the divisor of $h$.  Here, $(h)_0$ denotes the divisor of zeroes, and $(h)_\infty$ denotes the divisor of poles.
Two divisors $D_1, D_2$ are linearly equivalent, $D_1 \sim D_2$, 
if there exists a meromorphic function $h$ such that $(h) = D_1-D_2$.
We write $[D]$ for the equivalence class of $D$ with respect to linear equivalence.
The Picard group $\Pic(C)$ is the abelian group of divisors on $C$ modulo linear equivalence.
For $j \in \Z$, we write $\Pic^j(C)$ for the part of the Picard group of $C$ 
that has degree $j$, that is, 
$\Pic^j(C) := \{D : \text{ a divisor on $C$} ~|~ \deg(D) = j \}/\sim$.

To each divisor $D$ we associate a space of meromorphic functions 
$$\L(D) = \{f \mid (f) + D \geq 0\}.
$$
If $D = \sum_i n_i P_i \geq 0$, then in words, $\L(D)$ consists of meromorphic functions which are allowed to have poles of order $n_i$ at $P_i$.  We have $\dim(\L(D_1)) =  \dim(\L(D_2))$ if $D_1$ and $D_2$ are linearly equivalent.

\subsection{Positive divisors of a Lax matrix}\label{subsec:eigenmap}

A divisor $D \in \Pic^g(C_f)$ of degree $g$ is called {\it general} if $\dim \mathcal{L}(D)=1$.  
A divisor $D \in \Pic(C_f)$ is called {\it regular} with respect to the points $P$ and $O_j$ if $D$ is general and 
if $\dim \mathcal{L}(D+ kP-\sum_{j=n-k}^n O_j)=0$ for $k > 0$.  For brevity, we will sometimes just say that a divisor is regular when it is regular with respect to $P$ and the $O_j$.

Let us now fix $L(x) \in \beta^{-1}(f)$.  Let 
$$
\Delta_{i,j} := (-1)^{i+j} |L(x)-y|_{i,j}
$$ 
denote the (signed) $(i,j)$-th minor of the matrix $L(x) - y$.

\begin{thm}[cf. \cite{vMM}]\label{thm:double}
There exists a positive divisor $R$ of degree $2g+2$ supported on $S_{P,O} := \{P,O_1,\ldots,O_N\}$, and uniquely defined positive general divisors $D_1,D_2,\ldots,D_n,\bar D_1,\ldots,\bar D_n$ of degree $g$ such that for all $(i,j) \in [n]^2$, we have
$$
(\Delta_{i,j}) = D_j + \bar D_i + (j-i -1 )P + \sum_{r=j+1}^{i  -1} O_r - R.
$$
In addition, $D_1,\ldots,D_n$ have pairwise no common points, and $\bar D_1,\ldots,\bar D_n$ have pairwise no common points.
\end{thm}

See \S~\ref{pf:thm:double} for the proof.
Note that $\sum_{r=j}^i O_r$ is to be interpreted in a signed way: 
$$
  \sum_{r=j}^i O_r := \sum_{r=j}^\infty O_r - \sum_{r=i+1}^\infty O_r
  = 
  \begin{cases}
    \sum_{r=j}^i O_r & i \geq j,
    \\ 
    0 & i = j-1,
    \\
    -\sum_{r=i+1}^{j-1} O_r & i \leq j-2. 
  \end{cases}
$$

For $L(x) \in \beta^{-1}(f)$, define $g_i:=\Delta_{n,i}$.
Since $L(x)-y$ is singular
along $C_f$, the vector $g = (g_1,g_2,\ldots,g_n)^\bot$ (thought of as a vector with entries that are rational functions on $C_f$) is an eigenvector of $L(x)-y$. 
We define $g_i$ for $i \in \Z$ by $g_{i+n} = x^{-1}g_i$.  We also define $h_i := g_i/g_n$ for $i \in \Z$.  Thus $h_n = 1$ and $h_{i+n} = x^{-1}h_i$.  The vector $(h_1,h_2,\ldots,h_n)^\bot$ is also an eigenvector of $L(x)-y$.

\begin{definition}\label{def:D}
For $L(x) \in \beta^{-1}(f)$, define the divisor $\D = \D(L(x))$ on $C_f$ to be the minimum positive divisor satisfying
$$
  (h_i) + \D \geq \sum_{j=i+1}^n O_j - (n-i) P,
$$ 
that is, $h_i \in \mL(\D + (n-i) P - \sum_{j=i+1}^n O_j)$,
for $i=1,\ldots,n-1$.
\end{definition}

It follows from Theorem \ref{thm:double} that $\D = D_n$ from the theorem and that this divisor is uniquely determined by Definition \ref{def:D}.  In particular, $\D$ is a positive regular divisor of degree $g$ with respect to the points 
$P$ and $O_j$, and we have
\begin{align}\label{eq:hdiv}
(h_i) = D_i - \D - (n-i)P + \sum_{j=i+1}^n O_j.
\end{align}

\subsection{The eigenvector map}\label{subsec:eigen}

Let $R_A$ (resp. $R_O$) denote the set of orderings of the $m$ points $A_1,\ldots,A_m$ (resp. $O_1,\ldots,O_N$):
$$
R_A := \{\nu (A_1,\ldots,A_m) ~|~ \nu \in \mathfrak{S}_m \},
\qquad 
R_O := \{\tilde \nu(O_1,\ldots,O_N) ~|~ \tilde\nu \in \mathfrak{S}_N \}.
$$ 
By our assumption that $f \in \mathcal{V}$, the points $A_r$ (resp. $O_r$) are distinct, so $R_A$ has cardinality $m!$ and $R_O$ has cardinality $N!$.

We write $\nu_r ~(r=1,\ldots,m-1)$ (resp. $\tilde \nu_r ~(r=1,\ldots,N-1)$) 
for the generator of $\mathfrak{S}_m$ (resp. $\mathfrak{S}_N$),
which permutes the $r$-th and the $(r+1)$-st entry of 
$A \in R_A$ (resp. $O \in R_O$). 

By our assumption that $f \in \mathcal{V}$, we have $f_c \neq 0$.  Define $\mathcal{S}_f \subset  (\C^\ast)^M$ by
$$
  \mathcal{S}_f 
  = 
  \left\{(x_1,\ldots,x_M) \in (\C^\ast)^M ~|~ 
  \prod_{\ell=1}^M x_\ell = f_c \right\}.
$$ 

For $\ell=1, \ldots, M$, define $c_\ell$ to be the coefficient of the maximal power of $x$ in the $(\ell, \ell+1)$ entry of $L(x)^{n/M}$.  (Here, maximal power is the maximal power for a generic $L(x) \in \mL$.)  An explicit formula of 
$c_\ell$ for $L(x) \in \alpha(\mM)$ is given in Lemma \ref{lem:c-snake}.

We now define the {\it eigenvector map} $\eta_f$ for $f \in \mathcal{V}$ by
\begin{align}
  \label{eq:def-eta}
  \eta_f : \beta^{-1}(f) &\longrightarrow \Pic^g(C_f) \times \mathcal{S}_f \times R_O\\
          L(x) &\longmapsto ([\mathcal{D}],(c_1,\ldots,c_M),O),
\end{align}
where the ordering $O = (O_1,\ldots,O_N)$ is uniquely determined from $(\Delta_{i,j})$ by Theorem \ref{thm:double}.  For $f \in \oV$ we define the map $\phi_f$ by
\begin{align}
  \label{eq:def-phi}
  \phi_f : \psi^{-1}(f) &\longrightarrow \Pic^g(C_f) \times \mathcal{S}_f  \times R_O \times R_A \\
~ Q = (Q_1,\ldots,Q_m) &\longmapsto ([\mathcal{D}],(c_1,\ldots,c_M),O, A),
\end{align}
where $A = (((-1)^{n}\epsilon_1,0),\ldots,((-1)^{n}\epsilon_m,0))$, and $\epsilon_r(q)$ is given by \eqref{def:epsilon_r}.  That $(c_1,\ldots,c_M)$ lies in $\mathcal{S}_f $ is the content of Lemma \ref{lem:c-f} below.
We will omit the subscripts of $\eta_f$ and $\phi_f$ and just write 
$\eta$ and $\phi$ when no ambiguity will arise.

The following theorem should be compared to Theorem 1 in \cite{vMM}.
 
\begin{thm}\label{thm:eta}
We have a one-to-one correspondence between $L(x) \in \beta^{-1}(\mathcal{V})$
and the following data:
\begin{enumerate}
\item[(a)]
$f \in \mathcal{V}$,

\item[(b)]
$([\D], c, O) \in \Pic^g(C_f) \times \mathcal{S}_f \times R_O$ where
$\D$ is a positive divisor on $C_f$ of degree $g$,
and regular with respect to $P$ and the $O_j$.
\end{enumerate}
\end{thm}

See \S \ref{subsec:thm:eta} for the proof.
It is shown in \cite[Proof of Theorem 4.1]{LPgeom} that the map $\alpha:\mM \to \mL$ is generically $m!$ to $1$.  Let $\mathring{\mM}$ denote the open subset of $\mM$ where (a) the map $\alpha$ is $m!$ to $1$ (that is, $\mathring{\mM} = \alpha^{-1}(\alpha(\mathring{\mM})) \to \alpha(\mathring{\mM})$ is a $m!$ to $1$ map), and (b) the finite symmetric subgroup $\mS_m$ acting via the $R$-matrix is well-defined on $\mathring{\mM}$.

\begin{thm}\label{thm:phi}
Fix $f \in \oV$. 
We have an injection from $\psi^{-1}(f) \cap \mathring{\mM}$
to the collection of data $([\D], c, O, A) \in \Pic^g(C_f) \times \mathcal{S}_f \times R_O \times R_A$ such that
\begin{enumerate}
\item[(a)]
$\D$ is a positive divisor on $C_f$ of degree $g$. 
It is regular with respect to $P$ and $O_j$-s.

\item[(b)]
$c = (c_1,\ldots,c_M) \in \mathcal{S}_f$.

\item[(c)]
$O \in R_O$.

\item[(d)]
$A \in R_A$.
\end{enumerate}
\end{thm}
See \S \ref{subsec:thm:phi} for the proof.
We will usually just write $\psi^{-1}(f)$ instead of $\psi^{-1}(f) \cap \mathring{\mM}$ when no confusion arises (for example, when discussing the action of the $R$-matrix on the spectral data).

Let us perform a quick dimension count.  The dimension of $\mM$ or $\mL$ is equal to $mn$.  The dimension of $\Pic^g(C_f) \times \mathcal{S}_f$ is equal to $g+M-1$.  
The number of lattice points in the convex full of $\{(k,0),(0,n),(k+m,0) \}$ is equal to $g+N+m+M$.  Thus the dimension of $V_{n,m,k}$ \eqref{eq:Vnmk} 
is equal to $g+N+m-1$.  Using \eqref{eq:genus} we obtain $\dim(\mM) = \dim(\Pic^g(C_f) \times \mathcal{S}_f) + \dim(V_{n,m,k})$, consistent with Theorems \ref{thm:eta} and \ref{thm:phi}.

\section{Proofs from Section \ref{sec:eigenvector}}
\label{sec:proofeigenvector}

The first three subsections are devoted to introduce key lemmas for Theorem~\ref{thm:double} and \ref{thm:eta}.

\subsection{Special zeroes and poles of the eigenvector}

\begin{prop} \label{prop:Q}
The eigenvector $g = (g_i)_i^T$ of $L(x) \in \beta^{-1}(f)$ satisfies the following. 
\begin{enumerate}
\item[(i)] For any $j \in \Z$, 
the rational function $g_j/g_{j+1}$ on $C_f$ has a zero of order one at $O_{j+1}$.
\item[(ii)]
For any $j \in \Z$, 
the rational function $g_j/g_{j+1}$ on $C_f$ has a pole of order one at $P$.
\end{enumerate}
\end{prop}

To prove this proposition, we shall need the following generalization of Theorem \ref{thm:mit}, whose proof is essentially the same as that of Theorem \ref{thm:mit}. Suppose $\p = \{p_1,p_2,\ldots,p_{n-1}\}$ is an unordered noncrossing family of $n$ paths in $G'$ using all the sources except $i$, and all the sinks except $j$.  Identifying $[n] \setminus i \simeq [n-1] \simeq [n] \setminus j$ via the order-preserving bijection, we have that the non-abrupt paths in $\p$ induce a bijection of a subset $S \subset [n-1]$ with itself.  We let $\sign(\p)$ denote the sign of this permutation.

Let $|(L(x)-y)|_{i,j}$ denote the maximal minor of $(L(x)-y)$ complementary to the $(i,j)$-th entry. 

\begin{thm} \label{thm:minor}
We have
$$
|(L(x)-y)|_{i,j}= \sum_{\p = \{p_1,p_2,\ldots,p_{n-1}\}} \sign(\p) \wt(p_1) \wt(p_2) \cdots \wt(p_{n-1}),
$$
where the summation is over noncrossing (unordered) families of $n-1$ paths in $G'$ using all the sources except $i$, and all the sinks except $j$.  In other words, the coefficient of $x^a y^b$ in $|(L(x)-y)|_{i,j}$ counts (with weights) families of highway paths that 
\begin{itemize}
 \item start at all sources but $i$ and end at all sinks but $j$;
 \item do not cross each other;
 \item cross the $x$-line exactly $a$ times;
 \item contain exactly $b$ abrupt paths and $n-b-1$ non-abrupt paths.
\end{itemize}
\end{thm}

\begin{example}
 Let $(n,m,k)=(3,3,1)$. The Lax matrix is given as follows: $L(x)=$
$$
\left(\begin{array}{ccc} 
q_{1,1}x+q_{2,3}x+q_{3,2}x & q_{1,1}q_{2,1}x+ q_{1,1}q_{3,3}x+q_{2,3}q_{3,3}x & q_{1,1}q_{2,1}q_{3,1}x+x^2\\
q_{1,2}q_{2,2}q_{3,2}+x & q_{1,2}x+q_{2,1}x+q_{3,3}x & q_{1,2}q_{2,2}x+q_{1,2}q_{3,1}x+q_{2,1}q_{3,1}x\\
q_{1,3}q_{2,3}+q_{1,3}q_{3,2}+q_{2,2}q_{3,2} & q_{1,3}q_{2,3}q_{3,3}+x & q_{1,3}x +q_{2,2}x+q_{3,1}x
\end{array} \right)
$$
Consider the minor $|(L(x)-y)|_{1,2}=$  with rows $2,3$ and columns $1,3$. Here are some terms that appear in it:
$$|(L(x)-y)|_{1,2}= q_{2,2}x^2-q_{1,3}q_{2,3}q_{1,2}q_{2,2}x-xy - \cdots.$$
 The term $q_{2,2}x^2$ for example is formed by two paths: one starting at source $3$, turning at $q_{1,3}$, turning at $q_{1,2}$, going through at $q_{2,2}$, turning at $q_{3,2}$, turning at $q_{3,1}$ and thus finishing in
 sink $3$; the other a staircase path starting at source $2$, turning at each of $q_{1,2}$, $q_{1,1}$, $q_{2,1}$, $q_{2,3}$, $q_{3,3}$, $q_{3,2}$ and finishing at sink $1$. The first path contributes weight $q_{2,2}x$,
 while the second contributes weight $x$. Note that source $1$ and sink $2$ remain unused, as they should according to the theorem. Also note the the paths give bijection $2 \mapsto 1$ and $3 \mapsto 3$ between the 
 used sources and sinks. The induced permutation is the identity permutation, and that is why we have sign $+$ in front of this term. 
 
The term $-xy$ corresponds to one abrupt path of weight $-y$ from source $3$ to sink $3$, and one staircase path, the same as for the previous term. The term $-q_{1,3}q_{2,3}q_{1,2}q_{2,2}x$ corresponds to two paths 
 that induce the map $2 \mapsto 3$ and $3 \mapsto 1$ on sources and sinks, of weights $q_{1,2}q_{2,2}x$ and $q_{1,3}q_{2,3}$ respectively. The minus sign arises since the induced permutation in $S_2$ is a transposition. 
\end{example}

Let us prove Proposition~\ref{prop:Q}.
(i) If $N = 1$, the result is easy: $g_j/g_{j+n}$ vanishes to order $n$ at $O_1$.  We also have a shift automorphism $s: \mL \to \mL$ (see the proof of Proposition \ref{prop:Q}(ii)) which sends $O_1$ to $O_1$ and pulls $g_j/g_{j+1}$ back to $g_{j+1}/g_{j+2}$.  We will thus assume $N >1$.

 For each $i$, define
$$v^{(i)} := \left((-1)^1 |L(x)-y|_{i,1}, \ldots, (-1)^n |L(x)-y|_{i,n}\right)^T.$$
We shall think of $v^{(i)}$ as a vector whose entries lie in the coordinate ring of the affine plane curve $\tilde C_f$.  Like the vector $g = (g_1,g_2,\ldots,g_n)^T$, the vectors $v^{(i)}$ are (nonzero) eigenvectors of the matrix $L(x)$.  The matrix $L(x)$ is singular along the curve $\tilde C_f$, but generically it has rank $n-1$.  Thus for any $i$, the vectors $g$ and $v^{(i)}$ are multiples of each other.  More precisely, $g/v^{(i)}$ is a rational function on $C_f$.  To show that $g_j/g_{j+1}$ has a zero at $O_{j+1}$, we shall calculate using a convenient choice of $v^{(i)}$.

Choose $v = v^{(j)}$. Then $v_j =|L(x)-y|_{j,j}$ and $v_{j+1} =|L(x)-y|_{j,j+1}$. 
Thus, $v_j$ counts the families of paths that start at all sources but $j$ and end at all sinks but $j$,
as in Theorem \ref{thm:minor}.  

Let us make the substitution $(x,y) \sim (u^{n'},-(-\sigma_j)^{1/n'}u^{k'})$ inside $v_j$, and let $v'_j(x,y)$ denote the terms of in $v_j(x,y)$ that give the lowest degree in $u$ after the substitution.  Call this degree $d$.  The terms in $v_j(x,y)$ are obtained by either taking abrupt paths, or by taking coils.  In the proof of Lemma~\ref{lem:Q} we defined the $N$ coils in the network $G$, which we denote $C_1,C_2,\ldots,C_{N}$.  To obtain a path family $\p$ contributing to $v'_j$, instead of $C_j$, we include the $n'-1$ abrupt paths which use the vertices $j'$ where $j \neq j'$ but $j \equiv j' \mod N$.  For each of the coils $C_t$ where $t \neq j$, we can either include the coil (that is, include the $n'$ paths in $G'$) in $\p$, or we use the corresponding $n'$ abrupt paths instead.  In particular, considering $C_{j+1}$ we see that $v'_j$ has a factor of $(\sigma_{j+1}x^{k'} + (-y)^{n'})$. This factor vanishes under the substitution $(x,y) \sim (u^{n'},-(-\sigma_{j+1})^{1/n'}u^{k'})$, thus creating a zero at the point $O_{j+1}$ of order at least $d+1$.

For $v_{j+1}$, we count families of paths that start at all sources but $j$ and end at all sinks but $j+1$.  Let $v'_{j+1}(x,y)$ denote the terms of lowest degree in $v_{j+1}(x,y)$.  This lowest degree is again equal to $d$ as well (we caution that if $j = n$, then we take $v_{j+1}:=x^{-1} v_1$).  The calculation of $v'_{j+1}(x,y)$ is similar to that of $v_j(x,y)$, except that instead of a single ``incomplete'' coil with index $j$ modulo $N$, we have two incomplete coils $C_j$ and $C_{j+1}$ with indices $j$ and $j+1$ modulo $N$.  We obtain 
$$
v'_{j+1}(x,y) = a(q) x^\alpha y^\beta \prod_{t \neq j,j+1} (\sigma_{t}x^{k'} + (-y)^{n'})
$$
where $a(q)$ is some nonzero polynomial in $q_{ab}$-s, and $\alpha, \beta$ are nonnegative integers, and the product is over $t \in \{1,2,\ldots N\}$ not equal to $j$ or $j+1$ modulo $N$.  As a result, we see that $v_{j+1}$ vanishes at $O_{j+1}$ of order exactly $d$.

Thus $g_j/g_{j+1}$ has a zero at $O_{j+1}$.  The fact that $g_j/g_{j+1}$ vanishes to exactly order $1$ follows since we know that $g_j/g_{j+n}$ vanishes to the order of $n'$ at each $O_i$.

(ii) For $L(x) \in \beta^{-1}(f)$, we have 
rational functions $g_i/g_{i+1}$ on $C_f$.
To emphasize the dependence of $g_i/g_{i+1}$ on $L$,
we write $g_i/g_{i+1}(L)$. 
Let $\varsigma^\ast : \mL \to \mL$ be the shift map given by $L(x) \mapsto P(x) L(x) P(x)^{-1}$. (We will introduce the shift operator $\varsigma$ on $\mM$ in \S~\ref{sec:shift}.) Then $\beta(\varsigma^\ast(L(x))) = \beta(L(x))$ is invariant under $r$.
Let $Y$ be a Zariski dense subset of $\beta^{-1}(\oV)$
such that $\varsigma^\ast(Y) = Y$.
Then there exists a nonnegative integer $a_i$ given by
$$
a_i = \max\{ \mathrm{ord}_P(g_i/g_{i+1}(L))_\infty ~|~ L \in Y \},
$$
where $\mathrm{ord}_P(g_i/g_{i+1}(L))_\infty$ denotes the 
order of the pole of $g_i/g_{i+1}(L)$ at $P$. 
Since $g_i/g_{i+1}(\varsigma^\ast(L)) = g_{i-1}/g_{i}(L)$ as rational function on 
$C_{\beta(L)}$,
we conclude that for all $i$, we have 
$\mathrm{ord}_P(g_i/g_{i+1}(L))_\infty = a := \min_i\{a_i\}$.
But we know that $\mathrm{ord}_P(g_{i}/g_{i+n}(L))_\infty=n$
for all $i$ and $L \in \mL$.  Thus we must have $a=1$.
Since this value does not depend on the choice of $Y$, the claim follows.

\subsection{A holomorphic differential}
Let $\zeta$ be the differential form on the curve $C_f$ given by
$$
  \zeta = \frac{x^{k-1} dy}{\frac{\partial f}{\partial x}}.
$$

\begin{lem}\label{lem:zeta}
  The divisor of the differential form $\zeta$ is supported on $S_{P,O} = \{P,O_1,\ldots,O_N\}$: 
  $$
    (\zeta) = (k'-1) \sum_{i=1}^N O_i + \left((n-1)m-k-M \right) P.
  $$
  When $g \geq 1$, it is holomorphic.
\end{lem}

\begin{proof}
By using local expansions of $f$ around $O_i$ and $P$, we get 
\begin{align*}
  &\left(\frac{\partial f}{\partial x}\right) 
  = (n'-k'n) \sum_{i=1}^N O_i + n(k+m-1) P,
  \\
  &(dy) = (k'-1) \sum_{i=1}^N O_i - (m+k+M)P.
\end{align*}
From these and $(x)$ in \S 4.3, we obtain $(\zeta)$.

Now we prove that if $g \geq 1$ then $(n-1)m-k-M \geq 0$.
Let $p$ be the number of integer points inside a parallelogram in $\R^2$,
whose vertices are $(0,n)$, $(k,0)$, $(k+m,0)$ and $(m,n)$.
We have $p = 2g + M-1$, since the parallelogram is composed of 
the Newton polygon $N(f)$ and its copy, sharing the upper hull of $N(f)$.
Then the claim is equivalent to that if $g \geq 1$, then $p + N-1 \geq M+k$.

When $k=n$, $N=n$ follows. Then we have $p+ N-1 - (M+k) = p - (M+1)$.
It reduces to $2g-2$, which is non-negative when $g \geq 1$. 
 
When $m+k=n$, $M=n$ follows. Then we have $g = \frac{1}{2}(n(m-1)-m-N+1)$
which is negative if $m = 1$. So we assume $m \geq 2$.
We have $p+N-1=m(n-1)$ and $M+k = 2n-m$. 
Thus, we obtain $p+ N-1 - (M+k) \geq 0$ when $m \geq 2$.

When $n > k$, 
we prove the claim by choosing $M+k$ different points from $p$ points 
in the parallelogram.
When $m+k < n$ (resp. $m+k > n$), we can choose $M-1$ points on the upper hull
of $N(f)$, $k$ points; $(i,n-i)$ for $i=1,\ldots,k$ inside the upper (resp. lower) 
triangle, and one point inside the lower (resp. upper) triangle.

Finally the claim follows.
\end{proof}

\subsection{About $(c_1,\ldots,c_M) \in \mathcal{S}_f$}

Recall that $c_\ell$ is the coefficient of the maximal power of $x$ in the $(\ell, \ell+1)$ entry of $L(x)^{n/M}$ defined in \S \ref{subsec:eigen}.

\begin{lem}\label{lem:c-snake}
When $L(x) = \alpha(q=(q_{ij}))$, we have 
$$c_{\ell} = \sum_{i+j \equiv \ell+1 \mod M} q_{ij}.$$
\end{lem}

\begin{proof}
By Lemma \ref{L:mMmL}, it suffices to prove the lemma for $f(x,y) \in \psi(\mM)$.
We apply Lemma \ref{lem:entry} to interpret the $c_{\ell}$ as counting almost snake paths in $G$, that is, paths that turn at all steps but one.  The almost snake paths $p$ that are counted start at vertex $\ell$ and end at vertex $\ell+1$, and when drawn in the cylindric network $G'$, they are disjoint unions of $n/M$ paths.  In other words, $p$ wraps around the picture Figure \ref{fig:toda1} horizontally $n/M$ times.  Such paths $p$ are completely determined by the single vertex that the path goes through.  With $\ell$ fixed, the weights of these vertices are exactly the $q_{i,j}$ with $i+j \equiv \ell+1 \mod M$.
\end{proof}

Recall that we defined $f_c$ in \eqref{eq:fc}.
\begin{lem}\label{lem:c-f}
When $L(x) \in \beta^{-1}(f)$, we have 
\begin{align}\label{eq:prodc-f}
  \prod_{\ell=1}^M c_\ell= f_c = \pm \sum_{(j,i) \in L_c} f_{i,j}.
\end{align}
Thus the product of all $c_\ell$ is constant on $\psi^{-1}(f)$. 
\end{lem}

\begin{proof}
We first show that 
\begin{equation}
\label{eq:deg}
{\text{degree of $f_{i,j}$ in the $q_{r,s}$}}  = (m+k) n - ((m+k) i + n j).
\end{equation}
Indeed, lift a family $\p$ of paths that contributes to $f_{i,j}$ (as in Theorem \ref{thm:mit}) to the universal cover, as in Figure \ref{fig:toda2}.  By Theorem \ref{thm:mit}, the family $\p$ includes exactly $i$ abrupt paths and $n-i$ non-abrupt paths.  Let us suppose that the sources used by the $n-i$ non-abrupt paths are $(1/2, z_1), \ldots, (1/2, z_{n-i})$ and the sinks are $(m+1/2, w_1), \ldots, (m+1/2, w_{n-i})$, where $1 \leq z_r \leq n$ and $1 \leq w_r \leq n+m$.  To agree with our convention in Figure \ref{fig:toda1} that source and sink labels increase as we go down, we will take the $y$-coordinate to increase as we go down in Figure \ref{fig:toda2}; but otherwise, the coordinates we are using are the usual Cartesian coordinates.

For each $r$, let $w'_r \in [1-k,n-k]$ be chosen so that $w'_r \equiv w_r \mod n$. 
From Theorem \ref{thm:mit} it also follows that $$\sum_{r=1}^{n-i} (w_{r} - w'_{r})/n = i,$$ as this is the number of times the $x$-line would be crossed by the paths. 
Also, we know that $\{w_1, \ldots, w_{n-i}\} = \{z_1-k, \ldots, z_{n-i}-k\}  \mod n$ and thus $\{w'_1, \ldots, w'_{n-i}\} = \{z_1-k, \ldots, z_{n-i}-k\}$. This is because our collection of non-abrupt paths must form a collection of cycles in $G$. 
The number of the $q_{r,s}$ a path from $(1/2, z_r)$ to $(m+1/2, w_s)$ picks up is equal to $z_r+m-w_s$, and summing over all the non-abrupt paths in $\p$ we get 
\begin{align*}
{\text{degree of $f_{i,j}$ in the $q_{r,s}$}} &= m(n-j)+\sum_{r=1}^{n-i} z_r -\sum_{s=1}^{n-i} w_s \\
&= m(n-j)+\sum_{r=1}^{n-i} z_r -\sum_{s=1}^{n-i} w'_s + jn \\
&=m(n-i) + (n-i)k -jn = (m+k) n - ((m+k) i + n j),
\end{align*}
as claimed.  This degree is equal to $M$ when $(j,i) \in L_c$.

Now we prove the lemma. Let us partition all the edges in the network $G$ into $M$ snake paths, $H_1, \ldots, H_M$.  Suppose $p$ is a closed path in $G$.  Then we can partition $p$ into {\it {snake intervals}} -- maximal contiguous segments of $p$ that follow one of the $H_\ell$.  Such snake intervals are connected by a {\it horizontal step} going through a vertex (and picking up its weight), and moving from a snake path $H_\ell$ to the next snake path $H_{\ell+1}$.  It is easy to see that the vertices which separate $H_{\ell}$ from $H_{\ell +1}$ are exactly the 
ones labelled with $q_{ij}$ where $i+j \equiv \ell+1 \mod M$.

In order for $p$ to be closed, it must consist of $c$ horizontal steps, where $c$ is a multiple of $M$, since $p$ must come back to the snake path it started at.  When $(j,i) \in L_c$, \eqref{eq:deg} shows that all families of toric paths that contribute to $f_{i,j}$ consist of just a single 
closed path in $G$ that picks up exactly one of the $q_{ij}$ from each of the sets $T_\ell = \{q_{ij} \mid i+j \equiv \ell+1 \mod M\}$, for $\ell = 1, \ldots, M$.  Using Lemma \ref{lem:c-snake}, we see that each such term appears exactly once in the product $\prod_{\ell=1}^M c_\ell$.

We need to show that this is not only an injection, but a bijection. For each term contributing to 
$\prod_{\ell=1}^M c_\ell$, it suffices to find a closed path in $G$ that has this weight.
By \eqref{eq:deg}, such a closed path would necessarily contribute to $f_{i,j}$ for $(j,i) \in L_c$.

For each $q_{r,s}$ appearing in our chosen term of  
$\prod_{\ell=1}^M c_\ell$, we draw a horizontal step through the vertex labelled $q_{r,s}$ in the network $G$. From the endpoint of each such horizontal step (say from $H_\ell$ to $H_{\ell+1}$), attach a snake interval in $H_{\ell+1}$ leading to the start point of the horizontal step which goes from $H_{\ell+1}$ to $H_{\ell+2}$.  Some of these snake intervals may be empty. When we glue everything together, we get the desired closed path.  
\end{proof}

\subsection{Sketch proof of Theorem \ref{thm:double}}\label{pf:thm:double}
The proof is analogous that of \cite{vMM}.  We explain the differences.  
When $n-k \leq m \leq 2n-k$,
the triple of integers $(N, M, M^\prime)$ in van Moerbeke and Mumford's work \cite{vMM} corresponds to the parameters
$(n,n-k,m+k-n)$ in our case.  

One first shows that there is a positive regular divisor of degree $g$, denoted $\D$, satisfying Definition \ref{def:D}.  The correspondence $L(x) \mapsto \D$ is the main construction in \cite[Theorem 1]{vMM}.  Their results do not formally apply since our Lax matrix is not {\it regular} in their terminology.  Nevertheless, the properties of $L(x) \mapsto \D$ is proved in the same manner as \cite[Lemmas 3 and 4]{vMM}, where Proposition~\ref{prop:Q} takes the place of \cite[Lemma 2]{vMM}.  

The divisors $D_i$ and $\bar D_j$ of Theorem \ref{thm:double} are obtained by shifting and transposing the Lax matrix.  \cite[Proposition 1]{vMM} then says, in our terminology, 
$$(\Delta_{i,j} \zeta) = D_j + D'_i + (j-i -1)P + \sum_{r=j+1}^{i-1} O_r$$
where $\zeta$ is the differential of Lemma \ref{lem:zeta}.
We now take $R = (\zeta)$ to obtain Theorem \ref{thm:double}.  The statement about common points is on \cite[p.117]{vMM}.

\subsection{Proof of Theorem \ref{thm:eta}}\label{subsec:thm:eta}

We shall also use the following lemma which is proved in a similar way to \cite[Lemma 5]{vMM}.

\begin{lem} \label{lem:D}
Suppose $\D$ is a positive divisor on $C_f$ of degree $g$. 
which is regular with respect to the points $P$ and $O_j$.  We have
\begin{enumerate}
\item [(i)] $\dim \mathcal{L}(\mathcal{D}+ (n-i)P-\sum_{j=i+1}^n O_j)=1$
for $i \in \Z$.  
\item[(ii)] 
Suppose for each $i \in \Z$, we have fixed a nonzero element $h_i \in \mathcal{L}(\mathcal{D}+ (n-i)P-\sum_{j=i+1}^n O_j)$.
Then for $i_1 \leq i_2$, and any $h \in \mathcal{L}(\mathcal{D}+ (n-i_1) P-\sum_{j=i_2+1}^n O_j)$, there are unique scalars $b_{i_1},\ldots,b_{i_2}$ 
such that 
$$
  h = \sum_{i=i_1}^{i_2} b_i h_i,
$$ 
with $b_{i_1}, b_{i_2} \neq 0$.
\end{enumerate} 
\end{lem}

Let us prove the theorem.
Due to Theorem~\ref{thm:double}, Lemma~\ref{lem:c-f} and the assumption
on $\mathcal{V}$, for $L(x) \in \beta^{-1}(\mathcal{V})$ 
we have $f := \beta(L(x)) \in \mathcal{V}$,
and $\eta_f(L(x)) = ([\D],c,O)$ satisfying (b).
In the following we construct the inverse of $\eta_f$ \eqref{eq:def-eta}
for $f \in \mathcal{V}$.
Around $P \in C_f$ we have the local expansion
$(x,y) = (x_0 u^{-n}, y_0 u^{-m-k})$ as $u \to 0$.
We take $c=(c_1,\ldots,c_M) \in \mathcal{S}_f$, 
and recursively define $d_i ~(i \in \Z)$ by 
$$
  d_n = 1, \qquad d_i = c_\ell\, d_{i+1} \quad \text{ for } i \equiv \ell \mod M.
$$ 
Let $h_i \in \mathcal{L}(\mathcal{D}+ (n-i)P-\sum_{j=i+1}^n O_j)$
have an expansion around $P$ as 
$h_i = d_i u^{-n+i} + \cdots$.

Since
$y \, h_i \in \mathcal{L}(\mathcal{D}+ (n+m+k-i) 
P-\sum_{j=-k+i+1}^n O_j)$,
using Lemma \ref{lem:D}, we have unique scalars $b_{i,j}$ satisfying
$$
  y \, h_i = \sum_{j=n-m-k}^{n-k} x \, b_{i,i+j} h_{i+j},
  \qquad b_{i,i-m-k}, b_{i,i-k} \neq 0.
$$ 
By expanding it around $P$, we get 
$b_{i,i-m-k} = y_0 d_i / d_{i-m-k} = y_0 / f_c^{\frac{m+k}{M}}$ 
independent of $i$. 
Define an infinite matrix $A = (a_{ij})_{i \in [n], j \in \Z}$ by
$$
  a_{i,i+j} = 
  \begin{cases}
  b_{i,i+j} \displaystyle{\frac{f_c^{\frac{m+k}{M}}}{y_0}}
  & \text{for } -m-k \leq j \leq -k, \\
  0 & \text{otherwise}.
  \end{cases}
$$
By using \eqref{eq:L-A} for this $A$,
we obtain $L(A;x) \in \beta^{-1}(f)$.

\subsection{Proof of Theorem \ref{thm:phi}}\label{subsec:thm:phi}
Fix $L(x) \in \beta^{-1}(f)$, and assume that $\alpha^{-1}(L)$ contains  $Q \in \mathring{\mM}$.  By Lemma \ref{lem:R-matrix}, the action of the finite symmetric group $\mS_m$ via the $R$-matrix action of \S \ref{sec:dynamics} preserves the Lax matrix $L(x)$.  (See \S \ref{sec:actions} for more discussion of this action.)  Thus $\mS_m \cdot Q \subseteq \alpha^{-1}(L)$.
By \eqref{eq:q1q2}, the $\mS_m$ action induces the permutation action on the $m$ quantities $\epsilon_1(q),\ldots,\epsilon_m(q)$, which corresponds exactly to the permutation action on $R_A$.  Also, for $f \in\oV$, these $m$ quantities $\epsilon_1(q),\ldots,\epsilon_m(q)$ are distinct.  So $|\mS_m \cdot Q| = m! =|\alpha^{-1}(L)|!$ implying that $\mS_m \cdot Q = \alpha^{-1}(L)$.  Thus the map $\alpha^{-1}(L) \to (L,R_A)$ is an injection, and the claim follows.

\section{Actions on $\mathcal{M}$}
\label{sec:actions}

We study the family of commuting extended affine symmetric group actions on $\mathcal{M}$, 
introduced in \S~\ref{sec:dynamics}.  We also study another family of actions on $\mM$, which we call the {\it snake path actions}.
The main result in this section is Theorem~\ref{thm:commuting-actions}. 

Throughout \S \ref{sec:actions}--\ref{sec:fay}, we fix $f \in \oV$.
For simplicity, we write $Q_i := Q_i(x)$ in the rest of this section.

\subsection{The behaviour of spectral data under the extended affine symmetric group actions}\label{sec:WWonM}

Recall the definitions of energy and $R$-matrix from \S \ref{sec:Rmatrix}.  Let $A(x)$ and $B(x)$ be two $n$ by $n$ matrices of the shape \eqref{eq:Q}, with diagonal entries $\aa = (a_1,\ldots,a_n)$ and 
$\bb = (b_1,\ldots,b_n)$.  Define the {\it {energy}} by $E(A,B): = E(\aa,\bb)$ and the {\it $R$-matrix} to be the transformation 
$(A(x),B(x)) \mapsto (B'(x),A'(x))$, where $B'(x)$ and $A'(x)$ have diagonal entries $\bb'$ and $\aa'$ respectively, and $R(\aa,\bb) = (\bb',\aa')$. 
The following is proved for example in \cite[Corollary 6.4]{LP}.
\begin{lem}\label{lem:R-matrix}
Suppose $R(A(x),B(x)) = (B'(x),A'(x))$.  Then we have $A(x)B(x)=B'(x)A'(x)$.
\end{lem}
Indeed, Lemma \ref{lem:R-matrix} and the condition that $R$ is non-trivial uniquely determine the $R$-matrix as a birational transformation.

Now we reformulate the $W \times \tilde W$ actions introduced in \S~\ref{sec:dynamics}, in our current terminology.
The action of the extended affine symmetric group $W$ is generated by 
$s_i ~(1 \leq i \leq m-1)$ and $\pi$, where for $1 \leq i \leq m-1$, $s_i$ acts by 
$$
s_i: (Q_1,Q_2,\ldots,Q_m) \longmapsto 
  (Q_1,\ldots,Q'_{i+1},Q'_i,\ldots,Q_m)
$$
where $(Q'_{i+1},Q'_i)$ is the $R$-matrix image of $(Q_i,Q_{i+1})$.  
The operator $\pi$ acts as 
$$
\pi: (Q_1,Q_2,\ldots,Q_m) \longmapsto (Q_2,Q_3,\ldots,Q_m,P(x)^kQ_1P(x)^{-k}).
$$
The operator $e_u$ acts by moving, via the $R$-matrix, the last $u$ terms of 
$$
(Q_{u+1},\ldots,Q_m,P(x)^{k}Q_1P(x)^{-k},\ldots,P(x)^kQ_uP(x)^{-k})
$$
to the first $u$ positions, keeping them in order.

Recall from \S\ref{sec:dynamics} that we also have $\tilde Q = (\tilde Q_1,\ldots,\tilde Q_N) = (\tilde q_{i,j})$ coordinates on $\mM$,
where $\tilde Q_i := \tilde Q_i(x)$ is an $mn'$ by $mn'$ matrix:
$$
\tilde Q_i(x) = \left(\begin{array}{cccc} 
\tilde q_{i,1} & 0 & 0 &x\\
1&\tilde q_{i,2} &0 &0 \\
0&\ddots&\ddots&0 \\
0&0&1&\tilde q_{i,mn'}
\end{array} \right).
$$
Similarly, 
the action of $\tW$ generated by $\tilde s_r ~(1 \leq r \leq N-1)$
and $\tilde \pi$ on $\tilde Q = (\tilde Q_1,\ldots,\tilde Q_N) \in \mM$
is described as follows: 
for $1 \leq r \leq N-1$, $\tilde s_r$ acts by 
$$
\tilde s_r: (\tilde Q_1,\ldots, \tilde Q_N) \longmapsto (\tilde Q_1,\ldots,\tilde Q'_{r+1},\tilde Q'_r,\ldots,\tilde Q_N)
$$
where $(\tilde Q'_r,\tilde Q'_{r+1})$ is the $R$-matrix image of $(\tilde Q_{r+1},\tilde Q_r)$.  
The operator $\tilde \pi$ acts as
$$
\tilde \pi: (\tilde Q_1,\tilde Q_2,\ldots,\tilde Q_N) \longmapsto 
(\tilde Q_2,\tilde Q_3,\ldots,\tilde Q_N,\tilde P(x)^{m \bar k'}\tilde Q_1
\tilde P(x)^{-m \bar k'}),
$$
where $\tilde P(x)$ is the $mn'$ by $mn'$ version of $P(x)$ \eqref{eq:P}.

Recall that we denote by $\mS_m \subset W$ (resp. $\mS_N \subset \tilde W$) the finite symmetric group generated by $s_1,\ldots,s_{m-1}$ (resp. $\tilde s_1,\ldots,\tilde s_{N-1}$).  Then $W$ is generated by $\mS_m$ and $\pi$, and $\tilde W$ is generated by $ \mS_N$ and $\tilde \pi$.

Recall that in \S~\ref{subsec:main}
we define divisors $\mathcal{O}_u$ and $\mathcal{A}_u$ on $C_f$ as follows:
\begin{align}
\label{def:O}
&\mathcal{O}_{u} = uP-\sum_{i=N-u+1}^N O_i,
\qquad \mathcal{A}_u = uP-\sum_{i=1}^u A_i,
\end{align}
for $u \in \Z_{\geq 0}$.  
Let $\tau$ acts on $\mathcal{S}_f$ by $(c_1,\ldots,c_M) \mapsto (c_M, c_1,\ldots,c_{M-1})$.

\begin{thm}\label{thm:finite}
\begin{enumerate}
\item[(i)] 
The actions of $\mS_m$ and $\pi$ on $\psi^{-1}(f)$ induce the following
transformations on $\Pic^{g}(C_f) \times \mathcal{S}_f \times R_O \times R_A$:   for $r=1, \ldots, m-1$, 
$s_r$ induces  
\begin{align}\label{eq:vertical-s}
([\mathcal{D}],c,O,A)
\mapsto ([\mathcal{D}],c,O,\nu_{r}(A)),
\end{align}
and $\pi$ induces
\begin{align}\label{eq:vertical-pi}
([\mathcal{D}],c,O,A) \mapsto 
([\mathcal{D} - \mathcal{A}_1],\tau^{-1}(c),O, \nu_{m-1} \nu_{m-2} \cdots \nu_{1}(A)).
\end{align}
\item[(ii)] 
The actions of $\mS_{N}$ and $\tilde\pi$ on $\psi^{-1}(f)$ induce
the following transformations 
on $\Pic^{g}(C_f) \times \mathcal{S}_f \times R_O \times R_A$: 
for $r=1, \ldots, N-1$, $\tilde s_r$ induces 
\begin{align}\label{eq:horizontal-s}
([\mathcal{D}],c,O,A)
\mapsto ([\mathcal{D}],c,\tilde \nu_{N-r}(O),A),
\end{align}
and $\tilde \pi$ induces
\begin{align}\label{eq:horizontal-pi}
([\mathcal{D}],c,O,A) \mapsto 
([\mathcal{D} + \mathcal{O}_1],\tau(c), \tilde \nu_1 \cdots \tilde \nu_{N-1}(O),A).
\end{align}
\end{enumerate}
\end{thm}

Theorems \ref{thm:finite} will be proved in \S\ref{sec:prooffinite}.
The following theorem states that the commuting $\Z^m$ and $\Z^N$ actions on $\Pic^g(C_f)$ are linearized by the spectral map $\phi$.

\begin{thm}\label{thm:commuting-actions}
The following diagrams are commutative:
\begin{align}\label{eq:comm-e}
\xymatrix{
\psi^{-1}(f)  \ar[d]_{e_u} \ar[r]^{\phi \qquad \qquad} 
& \Pic^{g}(C_f) \times \mathcal{S}_f \times R_O \times R_A\ 
\ar[d]^{(-[\mathcal{A}_u],\,\tau^{-u}, \,id,\,id)} 
\\
\psi^{-1}(f) ~~ \ar[r]_{\phi \qquad \qquad} 
& \Pic^{g}(C_f) \times \mathcal{S}_f  \times R_O \times R_A, \\
}
\end{align}
for $u=1,\ldots,m$, and 
\begin{align}\label{eq:comm-ve}
\xymatrix{
\psi^{-1}(f)  \ar[d]_{\tilde e_u} \ar[r]^{\phi \qquad \qquad} 
& \Pic^{g}(C_f) \times \mathcal{S}_f \times R_O \times R_A
\ar[d]^{(+[\mathcal{O}_u],\,\tau^{u},\,id,\,id)} 
\\
\psi^{-1}(f) ~~ \ar[r]_{\phi \qquad \qquad} 
& \Pic^{g}(C_f) \times \mathcal{S}_f  \times R_O \times R_A, \\
}
\end{align}
for $u=1,\ldots,N$.
Here  $-[\mathcal{A}_{u}]$ and $+[\mathcal{O}_u]$ respectively 
act on $\Pic^g(C_f)$ as
$[\mathcal{D}] \mapsto [\mathcal{D} - \mathcal{A}_{u}]$ and 
$[\mathcal{D}] \mapsto [\mathcal{D} + \mathcal{O}_{u}]$.
\end{thm}

\begin{proof}
We shall show \eqref{eq:comm-e}.
The proof of \eqref{eq:comm-ve} is done in the similar way by exchanging 
the rules of the special points $\{A_i\}_{i \in [m]}$ for that of 
$\{O_j\}_{j \in [N]}$.

From \eqref{eq:vertical-s} and \eqref{eq:vertical-pi}, it follows that
a non-trivial part is to show $e_u$ \eqref{eq:Za-action} induces the action on 
$\Pic^{g}(C_f) \times R_A$:
$$
([\mathcal{D}],A) \mapsto 
([\mathcal{D} - \mathcal{A}_u],A).
$$

Note that for $1 \leq i < j \leq m$,
$\nu_{j-1} \nu_{j-2} \cdots \nu_i \in \mathfrak{S}_m$ acts on $R_A$ as 
\begin{align}\label{eq:nuA}
(A_1,\ldots,A_m)
\mapsto (A_1,\ldots,A_{i-1},A_{i+1},\ldots,A_j,A_i,A_{j+1},\ldots,A_m).
\end{align}
Thus, \eqref{eq:vertical-pi} indicates that $\pi$ induces
$([\mathcal{D}],(A_1,\ldots,A_m)) \mapsto
([\mathcal{D}-\mathcal A_1],(A_2,\ldots,A_m,A_1))$, and $\pi^u$ induces
$$
([\mathcal{D}],(A_1,\ldots,A_m)) \mapsto
([\mathcal{D}-\mathcal A_u],(A_{u+1},\ldots,A_m,A_1,A_2,\ldots,A_m)).
$$
Further, \eqref{eq:vertical-s} denotes that $s_r$ does not change a point 
in $\Pic^g(C_f)$ and acts on $R_A$ as a permutation $\nu_r \in \mathfrak{S}_m$.
The rest part $(s_u \cdots s_{m-1})(s_{u-1} \cdots s_{m-2})
\cdots (s_1 \cdots  s_{m-u})$ of $e_u$ changes 
$(A_{u+1},\ldots,A_m,A_1,A_2,\ldots,A_m) \in R_A$ following 
(the inverse of) \eqref{eq:nuA} as
\begin{align*}
&(A_{u+1},\ldots,A_m,A_1,A_2,\ldots,A_m) 
\stackrel{\nu_1 \cdots  \nu_{m-u}}{\mapsto}
(A_1,A_{u+1},\ldots,A_m,A_2,\ldots,A_m)
\\
&\stackrel{\nu_2 \cdots  \nu_{m-u+1}}{\mapsto}
\cdots 
\stackrel{\nu_{u-1} \cdots \nu_{m-2}}{\mapsto}
(A_1,\ldots,A_{u-1},A_{u+1},\ldots,A_m,A_u)
\stackrel{\nu_u \cdots \nu_{m-1}}{\mapsto}
(A_1,\ldots,A_m). \qedhere
\end{align*}
\end{proof}

We remark that $e_u$ and $\tilde e_u$ act on the sets $R_A$ and $R_O$ as the identity transformations.

\subsection{The shift operator}\label{sec:shift}
We define a shift operator $\varsigma$ acting on $\mM$ as 
\begin{align}
\varsigma: (Q_i)_{1 \leq i \leq m} \mapsto (P(x) Q_i P(x)^{-1})_{1 \leq i \leq m}.
\end{align}
It is easy to see that it induces an action $\varsigma^\ast: \mL \to \mL$ given by 
$\varsigma^\ast : L(x) \mapsto P L(x) P^{-1}$.  The following result generalizes \cite[Proposition~2.6]{Iwao07}.

\begin{prop}\label{prop:shift}
The following diagram is commutative:
\begin{align}\label{eq:comm-s}
\xymatrix{
\psi^{-1}(f)  \ar[d]_{\varsigma} \ar[r]^{\phi \qquad \qquad } 
& \Pic^{g}(C_f) \times \mathcal{S}_f \times R_O \times R_A~~\ 
\ar[d]^{(+[P-O_N],\,\tau,\,\tilde \nu_1 \cdots \tilde \nu_{N-1}, \,{id})} 
\\
\psi^{-1}(f) ~~ \ar[r]_{\phi \qquad \qquad} 
& \Pic^{g}(C_f) \times \mathcal{S}_f \times R_O \times R_A, \\
}
\end{align}
where $+[P-O_N]$ acts on $\Pic^g(C_f)$ as
$[\mathcal{D}] \mapsto [\mathcal{D} + P - O_N]$.
\end{prop}

\begin{proof}
Suppose $Q \in \psi^{-1}(f)$ and $\phi (Q) = ([\mathcal{D}],c,O,A)$.
Then 
$\phi \, \circ \, \varsigma (Q) = ([\D'],\tau(c),\tilde \nu_1 \cdots \tilde \nu_{N-1}(O),A)$,
where $\D'$ is some positive divisor of degree $g$.
We set $(O'_1,\ldots,O'_N) := \tilde \nu_1 \cdots \tilde \nu_{N-1}(O_1,\ldots,O_N) = (O_N,O_1,\ldots,O_{N-1})$.

Recall that $g=(g_1,g_2,\ldots,g_n)^T$ denotes the eigenvector of $\alpha(Q)$.
An eigenvector of $\alpha \circ \varsigma(Q)$ is $(g'_1,g'_2,\ldots,g'_n)^T:= 
(x g_n, g_1,g_2,\ldots,g_{n-1})^T$.
Define  $h_i' := g_i'/g_n' = g_{i-1}/g_{n-1} = h_{i-1}/h_{n-1}$; these ratios do not depend on which eigenvector of $\alpha \circ \varsigma(q)$ we chose.  Then by \eqref{eq:hdiv} 
\begin{align*}
(h'_i) = (h_{i-1})-(h_{n-1}) &= (D_{i-1} - \D - (n-i+1)P + \sum_{j=i}^n O_j) - (D_{n-1} - \D - P +  O_n) \\&= D_{i-1} - D_{n-1} - (n-i)P + \sum_{j=i+1}^{n}O'_j.
\end{align*}
By the uniqueness of $\D'$ it follows that the divisor $\D'$ (resp. $D'_i$) is equal to $D_{n-1}$ (resp. $D_{i-1}$).  Furthermore, $\D'-\D \sim P - O_n$, as required.

\end{proof}

\subsection{$W \times \tilde W$ actions on Lax-matrix}
For $r=1, \ldots, N-1$, $i=0, \ldots, n'-1$ and $s =0,\ldots,m-1$, 
using the energy in \S~\ref{sec:Rmatrix} define
$$E^{(s)}_{r+ki}: = E(\tilde P(x)^{-mi-s} \tilde Q_r \tilde P(x)^{mi+s}, \tilde P(x)^{-mi-s} \tilde Q_{r+1} \tilde P(x)^{mi+s}).$$
Define an $n$ by $n$ matrix $B_{s,r}$ by:
$$
  B_{s,r} = \mathbb{I}_n 
    + \sum_{i=0}^{n'-1} \kappa_{r+ki}^{(s)} E_{r+ki,r+ki+1},
  \quad \kappa_{r+ki}^{(s)} = \frac{\sigma_{r+1}(q)-\sigma_r(q)}{E^{(s)}_{r+ki}},
$$
where $(E_{a,b})_{c,d} = \delta_{a,c} \delta_{b,d}$ 
for $a,b,c,d \in \Z / n \Z$, and $\mathbb{I}_n $ denotes the identity matrix.

\begin{lem}\label{lem:vertical-L}
\begin{enumerate}
\item[(i)] 
The action of $s_r ~(r=1,\ldots,m-1)$ on $\mM$ induces 
the identity map on $\mL$, 
and the action of $\pi$ on $\mM$ induces 
the adjoint transformation
$$
{\pi}^\ast : L(x) \mapsto Q_1^{-1} L(x) Q_1
$$
on $\mL$.
\item[(ii)] 
The action of $\tilde s_r$ and $\tilde \pi$ on $\mM$ induces 
adjoint transformations on $\mL$, given by
$$
\tilde{s}_r^\ast : L(x) \mapsto (B_{0,r})^{-1} L(x) B_{0,r},
$$
for $r=0, \ldots, N-1$, and
$$
{\tilde \pi}^\ast : L(x) \mapsto P(x) L(x) P(x)^{-1}.
$$
\end{enumerate}
\end{lem}

\begin{proof}
(i) Due to Lemma~\ref{lem:R-matrix}, we have $\alpha \circ s_r(Q) = \alpha(Q)$ for $1 \leq r \leq m-1$.
The induced action of $\pi$ is 
\begin{align*}
  \pi^\ast (Q_1 \cdots Q_m P(x)^k) 
  &= Q_2 \cdots Q_m P(x)^k Q_1, 
\end{align*}
and the claim follows. 
\\
(ii) By \cite[Theorem 6.2]{LP},
the action of $\tilde s_r$ on $\mM$ is given by
\begin{align}\label{eq:actionB}
  (Q_s)_{s=1,\ldots,m} \mapsto 
  ((B_{s,r})^{-1} Q_s \,B_{s+1,r})_{s=1,\ldots,m}.
\end{align}
By definition $E^{(s)}_{r+ki}$ satisfies
$E^{(s+m)}_{r+ki} = E^{(s)}_{r+k(i-1)}$,
and we have $B_{s+m,r} = P(x)^k B_{s,r} P(x)^{-k}$.
The claim follows.

The second part is easy to see, as the transformation $\tilde \pi$ just cycles 
the indices inside the $Q_i$ (Definition \ref{def:action}). 
\end{proof}

Since similar matrices have the same
characteristic polynomial, we obtain the following: 
 
\begin{cor}\label{cor:L-WW}
For each $q \in \mM$, $\psi(q)$ is invariant under the action of 
$W \times \tilde W$. In particular, the coefficients of $f(x,y) = \psi(q)$ are conserved quantities for $W \times \tilde W$.
\end{cor}

\begin{lem}\label{lem:pi-c}
Each $c \in \mathcal{S}_f$ is invariant under actions of $s_i$ and $\tilde s_i$. 
The cyclic shift $\pi$
acts on $\mathcal{S}_f$ via 
$$\tau^{-1}: (c_1, \ldots, c_M) \mapsto (c_2, \ldots, c_M, c_1),$$
and $\tilde \pi$ acts on $\mathcal{S}_f$ via 
$$\tau: (c_1, \ldots, c_M) \mapsto (c_M, c_1, \ldots, c_{M-1}).$$
\end{lem}

\begin{proof}
From the definition of the cyclic group action \eqref{eq:pi},
it follows that $\tilde \pi(Q)_{ij} = q_{i,j-1}$ and $\pi(Q)_{ij} = q_{i+1,j}$. Thus, using Lemma~\ref{lem:c-snake}, we see that
$\tilde \pi$ induces 
$$
  c_\ell = \sum_{i+j \equiv \ell+1 \mod M} q_{i,j}
  ~\mapsto \sum_{i+j \equiv \ell+1 \mod M} q_{i,j-1}
  =  \sum_{i+j \equiv \ell \mod M} q_{i,j} = c_{\ell -1},
$$  
and that $\pi$ induces $c_{\ell} \mapsto c_{\ell+1}$ similarly.
\end{proof}

\subsection{Eigenvector change under conjugation}
For a vector $v = (v_1,v_2,\ldots,v_n)$ of rational functions on $C_f$, define
$$
D(v) = \text{(common zeroes of $v_i$)} - \text{(common poles of $v_i$)} + R + O_n
$$
where $R$ is the divisor supported on $S_{P,O} = \{P,O_1,\ldots,O_N\}$,
which appears in Theorem \ref{thm:double}. (See Lemma~\ref{lem:zeta}
for the explicit formula of $R$.)
The following result is immediate.

\begin{lem}\label{L:scale}
Let $f$ be a rational function on $C$.  Then $D((fv_1,\ldots,fv_n))$ is linearly equivalent to $D((v_1,\ldots,v_n))$.
\end{lem}

\begin{lem}\label{L:D}
The positive general divisor $\D = \D(L(x))$ of degree $g$ associated to $L(x)$ is equal to $D((\Delta_{1,n},\ldots,\Delta_{n,n}))$.
\end{lem}
\begin{proof}
By Theorem \ref{thm:double}, $\D = D_n$ belongs to the common zeroes.  Also by Theorem \ref{thm:double}, $\bar D_1,\ldots,\bar D_n$ have no common points, so they do not contribute to $D((\Delta_{1,n},\ldots,\Delta_{n,n}))$.  
\end{proof}

\begin{lem}\label{L:Qv}
Suppose $M = M(x,y)$ is a $n \times n$ matrix whose entries are polynomials in $\C[x,y]$, thought of as rational functions on $C_f$ with poles supported at $P$.  Let $D(M \cdot v) - D(v) = D_+ - D_-$, where $D_+$ and $D_-$ are
positive divisors. 
Then restricted to $ C_f \setminus P$, we have that $(\det M)_0 - D_+$ is a positive divisor.  Also, $D_-$ is supported at $P$.
\end{lem}
\begin{proof}
Let $p \in C_f \setminus P$.  Then the entries of $M$ are regular at $p$.  We shall show that the multiplicity of $p$ in $D_+$ is less than the multiplicity of $p$ as a zero in $\det M$.  Since $v'_i = \sum_{j} M_{ij}(x,y) v_j$, and $M_{ij}$ is regular at $p$, it is clear that $\mult_p D(Mv) \geq \mult_p D(v)$, where $\mult_p$ denotes the multiplicity of a divisor at a point $p$.  We also have $$v_i = \sum_{j} (M^{-1})_{ij}(x,y) v'_j = \dfrac{1}{(\det M)(x,y)}\sum_j \pm |M|_{i,j}(x,y) v'_j.$$
Since $|M|_{i,j}(x,y)$ is regular at $p$, we have $\mult_p D(v) \geq \mult_p(Mv) - \mult_p(\det M)$.  Both claims follow.
\end{proof}

Suppose $L'(x) = Q_1^{-1}L(x)Q_1$ and $\D' = \D(L'(x))$.  We shall compute $\D' - \D$ up to linear equivalence.  Note that $v = (\Delta_{1,n},\ldots,\Delta_{n,n})^T$ is an eigenvector of $L(x)^T-y$, and $L'(x)^T = Q_1^TL(x)^T(Q_1^{-1})^T$, so an eigenvector of $L'(x)^T-y$ is equal to $v' = Q_1^Tv$, where
$$
Q_1^T= Q_1^T(x) = 
\left(\begin{array}{cccc} 
q_{1,1} & 1 & 0 &0\\
0&q_{1,2} &1 &0 \\
0&\ddots&\ddots&1 \\
x&0&0&q_{1,n}
\end{array} \right).
$$
Extend the vector $v$ into an infinite vector $\tv$ by $v_{i+n} = xv_i$.  As in \S\ref{sec:Lax}, let $A = (a_{i,j})_{i \in [n], j \in \Z}$ be the infinite ``unfolded" version of $L^T(x)$, satisfying $(L^T)_{i,j}(x) = \sum_\ell {a}_{i,j+\ell n} x^\ell$.  (Note that $A$ is a matrix of scalars, but $\tv = (\ldots,v_{-1},v_0,v_1,\ldots)^T$ is a matrix of functions.)  Then we have 
\begin{equation}\label{E:unfolded}
A \cdot \tv = y \tv.
\end{equation} 

Define $w = (v_1,v_2,\ldots,v_{k+m})^T$ and $w' = (v'_1,v'_2,\ldots,v'_{k+m})^T$.  We claim that there exists a $(k+m) \times (k+m)$ matrix $M = M(y)$ such that $w' = M(y) \cdot w$.  Since $v' = Q_1^T v$, we have $v'_i = q_i v_i + v_{i+1}$ (where $q_i$ is extended periodically if $i \geq n$).  We now write $v_{k+m+1}$ in terms of $v_1,\ldots,v_{k+m}$ using \eqref{E:unfolded}.  The matrix $A$ is supported on the $m+1$ diagonals $k,k+1,\ldots,k+m$.  The matrix $A - y$ is supported on the $(k+m+1)$ diagonals $0,1,\ldots,k+m$.  Thus \eqref{E:unfolded} gives
\begin{equation}\label{E:km1}
a_{1,k+1}v_{k+1} + \cdots + a_{1,k+m+1} v_{k+m+1} = y v_{1}.
\end{equation}
Also note that $a_{1,k+m+1} = 1$.  So
$$
M(y) = \left(\begin{array}{ccccccc} 
q_{1,1} & 1 & 0 & \cdots & \cdots&0&0\\
0&q_{1,2} &1 &0&0&0&0 \\
\vdots & & \ddots & \ddots & \\
0&\cdots && q_{1,k+1} & 1\\
\vdots &  && & \ddots & \ddots \\
0 &\cdots && & &q_{k+m-1}&1 \\
y&0&\cdots&-a_{1,k+1}& \cdots & -a_{1,k+m-1} &q_{1,k+m}-a_{1,k+m}
\end{array} \right).
$$

\begin{lem}\label{lem:time}
With the above conventions, we have $$\D' \sim \D +A_1 - P.$$
\end{lem}
\begin{proof}
We shall show that $D(v') - D(v) = A_1 - P$.  This suffices by Lemmas \ref{L:scale} and \ref{L:D}.  
From Lemma \ref{L:Qv}, we have $D(v') - D(v) = D_+ - D_-$ where $D_{\pm}$ are positive divisors, with $D_+$ when restricted to $C_f \setminus P$ supported on $\{(\det Q^T)(x) = 0\}$, and $D_-$ supported at $P$.  Since $D(v')$ is a positive divisor of degree $g$
by the construction of $L'(x)^T$, we have that $D_+$ and $D_-$ have the same degree. 

We now calculate $D_-$.  By Theorem \ref{thm:double}, we have that $\mult_P(v_i) = C - i$ for some integer $C$ (we use the convention that negative multiplicity is a pole), and this formula still holds for the infinite vector $(\ldots,v_{-1},v_0,v_1,\ldots)$.  Since $v'_i = q_i v_i + v_{i+1}$, we have $\mult_P(v'_i) = C - (i+1)$.  Thus $D_- = P$.  It follows that $D_+$ is a single point in $C_f \setminus P$.

We shall show that $D_+$ must be supported on $A_1$.
By Theorem \ref{thm:double}, 
the common zeros of $v_i$ and $v_j$ for any $i \neq j$ are only $D_n$ except for 
the points in $S_{P,O}$. Thus we have $D(w) = D_n + R'$ where $R'$ is
supported on $S_{P,O}$.  
But $\{(\det Q^T)(x) = 0\}$ does not intersect $S_{P,O}$.  Using Lemma \ref{L:Qv} now applied to $w' = M(y)w$, we see that $D_+$ must be supported on $\{\det M(y) = 0\}$.

By Lemma \ref{L:A} below, we conclude that $D_+$ is a multiple of $A_1$.  Since $D(v') - D(v)$ is a divisor of degree 0, we must have
$$
D(v') - D(v) = A_1 - P,
$$
as required.
\end{proof}

\begin{lem}\label{L:A}
The intersection of $(\det Q^T)(x) = 0$ and $(\det M(y)) = 0$ is the single point $A_1$.  
\end{lem}
\begin{proof}
We check that $\det M(y) = \pm y$.  
It is clear that 
$$
  \det M(y) = (-1)^{m+1} y + q_{1,1} \cdots q_{1,k}
  \det \begin{pmatrix}
        q_{1,k+1} & 1 \\
        & q_{1,k+2} & 1 \\
        & & \ddots & \quad \ddots \\
        & & & q_{1,k+m-1} & 1\\
        -\tl_{1,k+1} & -\tl_{1,k+2} & \cdots & -\tl_{1,k+m-1} & q_{1,k+m} - \tl_{1,k+m} 
       \end{pmatrix}. 
$$
Write $u_i~(i=1,\ldots,m)$ for the $m$ columns of the above $m$ by $m$ matrix. We show that they are linearly dependent as
$$
  u_1 - q_{1,k+1} \big(u_2 - q_{1,k+2} (u_3 - \cdots (u_{m-1} - q_{1,k+m-1} u_m) \cdots ) \big) = 0,
$$
which reduces to
\begin{align}\label{eq:detM=0}
&-\tl_{1,k+1} + q_{k+1} \tl_{1,k+2} - q_{1,k+1} q_{1,k+2} \tl_{1,k+3} + 
\cdots \\
&+(-1)^{m} q_{k+1}q_{k+2} \cdots q_{k+m-1} \tl_{1,k+m} 
+ (-1)^{m+1} q_{k+1}q_{k+2} \cdots q_{k+m} = 0. \nonumber
\end{align}
By the definition of $A$, we have 
$$
  \tl_{1,j} = \begin{cases} 
              l_{j,1} & j=k+1,\ldots,n, \\
              l_{j-n,1} & j=n+1,\ldots,n+k, \\
              \end{cases}
$$ 
where $L = (l_{ij})_{i \in [n], j \in \Z}$ is the infinite unfolded version of $L(x)$.
On the network $G'$, $l_{k+i,1}$ is the weight generating function of paths 
from the source $k+i$ to the sink $k+1$. 
Note that each weight in $l_{k+i,1}$ is of degree $m+1-i$ in $q$-s.
We divide $l_{k+i,1}$ into two parts:
$$
  l_{k+i,1} = l_{k+i,1}' + l_{k+i,1}'',
$$
where $l_{k+i,1}'$ is the sum of the weights {\em with} $q_{1,k+i}$,
and $l_{k+i,1}'$ is the sum of the weights {\em without} $q_{1,k+i}$.
We claim that there is one-to-one correspondence between  
the paths contributing to $l_{k+i,1}'$ 
and the paths contributing to $l_{k+i+1,1}''$. 
Precisely, we have $q_{1,k+i} l_{k+i+1,1}'' = l_{k+i,1}'$.
Combining with $l_{k+1,1}'' = 0$ and $l_{k+m,1}' = q_{1,k+m}$, 
we obtain \eqref{eq:detM=0}.
\end{proof}

\subsection{Proof of Theorem \ref{thm:finite}}\label{sec:prooffinite}

(i) Due to the definition of $s_r$, Lemma~\ref{lem:vertical-L}(i) and  
Lemma~\ref{lem:pi-c},  
for $Q=(Q_1,\ldots,Q_m) \in \psi^{-1}(f)$ with 
$\phi (Q) = ([\mathcal{D}],c,O,(A_1,\ldots,A_m))$, we have
$$
\phi \circ s_r(Q_1,\ldots,Q_m) = ([\mathcal{D}],c,O,(A_1,\ldots,A_{r-1},
A_{r}(Q'),A_{r+1}(Q'),A_{r+2},\ldots,A_m)),
$$  
where $Q' = s_r(Q)$. From \eqref{eq:q1q2} we see that 
$\epsilon_{r}(q') = \epsilon_{r+1}(q)$ and $\epsilon_{r+1}(q') = \epsilon_{r}(q)$,
and \eqref{eq:vertical-s} follows.

It is easy to see that for $L(x) = \alpha(Q)$ we have 
$\eta \circ \pi(L(x)) = ([\D'], \tau^{-1}(c),O')$,
where $\D'$ is some divisor.  By Lemma \ref{lem:time}, we have $[\D] = [\D' + A_1 - P]$.   The claim follows.
\\
(ii) First we prove \eqref{eq:horizontal-s}.
Let $g = (g_1,\ldots,g_n)^t$ be the eigenvector of $L(x) \in \psi^{-1}(f)$,
and set $(g_1',\ldots,g_n')^t := (B_{0,r})^{-1} g$.
Define $h_i' := g_i' / g_n'$ for $i=1,\ldots,n-1$. 
Then we have 
$$g_j' 
= \begin{cases}
  g_{r+ki} - \kappa^{(0)}_{r+ki} g_{r+ki+1} & j \equiv r+ki \mod n   
  \\
  g_j & \text{otherwise}.
  \end{cases}
$$
We must show that $(h_j')_{\infty} = (h_j)_{\infty}$ 
for $r=1,\ldots,N-1$ and $j=1,\ldots,n-1$.  It is enough to consider the case of $r=1$ and $j=1$.
Due to Theorem \ref{thm:double} and \eqref{eq:hdiv}, there are positive divisors
$\mathcal{D}, \mathcal{D}'$
of degree $g$ such that
$(h_1)_\infty = (n-1) P + \mathcal{D}$,  
$(h_1')_\infty = (n-1) P + \mathcal{D}'$.
Since $h_1' = h_1 - \kappa^{(0)}_{1} h_2$ and $(h_1)_\infty > (\kappa^{(0)}_{1} h_2)_\infty = (h_2)_\infty$,
we get $(h_1')_\infty \leq (h_1)_\infty$.
Thus it follows that $\mathcal{D}' = \mathcal{D}$.
From \eqref{eq:actionB} and Lemma~\ref{lem:pi-c}
it follows that both $A \in R_A$ and $c \in \C^M$ are not changed by 
$\tilde s_r$. 
From the facts that $\sigma_{N+1-r}(q)$ is the product of all 
diagonal elements of $\tilde Q_r$, and that 
the $R$-matrix action changes $(\tilde Q_r, \tilde Q_{r+1})$ 
to $(\tilde Q_{r+1}',\tilde Q_r')$,
it follows that $\tilde s_r$ exchanges the $(N-r)$-th and the $(N-r+1)$-th
elements of $O \in R_O$.
Then we get \eqref{eq:horizontal-s}.

Since $\tilde{\pi}$ induces the shift $\varsigma^\ast$ of $L(x)$,
\eqref{eq:horizontal-pi} follows from Proposition~\ref{prop:shift}.

\subsection{Snake path actions}
Recall that $M=\gcd(n,k+m)$. 
The {\em snake path actions} are torus actions on $\mM$, considered in \cite{LP}.
Recall from \S~\ref{subsec:lem:Q} 
that a snake path is a closed path on $G$ that turns at every vertex.  
Thus it alternates between going up or going right.  
For $1 \leq s \leq M$ and $t \in \C^\ast$,
the action $T_s := T_s(t)$ on $\mathcal{M}$ is given by
\begin{align}
  \label{eq:snake-q}
  T_s(q)_{ij} = 
  \begin{cases}
  t \, q_{ij}  & \text{ if } j \equiv s-i+1 \mod M\\
  \displaystyle{\frac{1}{t}}\, q_{ij}  & \text{ if } j \equiv s-i \mod M\\
  q_{ij} & \text{ otherwise}.
  \end{cases}
\end{align}
Informally, $T_s(t)$ multiplies all left turns on a snake path by $t$, and all right turns by $1/t$.  

\begin{lem}\label{lem:snake-h}
The induced action $T^\ast_s ~(1 \leq s \leq M)$ on $\mL$ is given by
$$
T^\ast_s: ~ L(x) \mapsto D_s(t) L(x) D_s(t)^{-1},
$$  
where 
$$
  D_s(t) = \mathrm{diag}(d_i)_{1 \leq i \leq n},
  \quad 
  d_i = \begin{cases}
        t & \text{ if $i \equiv s \mod M$} \\
        1 & \text{ otherwise }.  
        \end{cases}
$$
\end{lem}

\begin{proof}
For simplicity we write $D_s := D_s(t)$ and $P := P(x)$.
We rewrite $D_s L(x) D_s^{-1}$ as
\begin{align}\label{eq:DLD}
  D_s L(x) D_s^{-1}
  &=
  \left( \prod_{i=1}^{m} (P^{-i+1} D_s P^{i-1}) Q_i (P^{-i} D_s P^{i})^{-1} 
  \right)
  \cdot (P^{-m} D_s P^{m}) P^k D_s^{-1}.
\end{align}
This is equal to $\alpha(T_s(Q))$ since $(P^{-m} D_s P^{m}) \cdot P^k D_s^{-1} = P^k$. 
\end{proof}

\begin{prop}\label{prop:snake-dj}
We have the following commutative diagram:
$$
\xymatrix{
\mM ~~\ar[d]_{{e}_u} \ar[r]^{T_s} 
& ~~\mM \ar[d]^{{e}_u} \\
\mM ~~~ \ar[r]_{\varsigma^{-u} \circ T_s \circ \varsigma^{u}} & ~~~\mM .\\
}
$$
\end{prop}

\begin{proof}
It is enough to prove
$((\varsigma^\ast)^{-u} \circ T^\ast_s \circ (\varsigma^\ast)^{u}) \circ e^\ast_u
= e^\ast_u \circ T^\ast_s$ acting on $\mL$.
For $L(x) = Q_1 \cdots Q_m P^k \in \mathcal{L}$, we have 
\begin{align*}
  L(x) &\stackrel{e^\ast _u}{\longmapsto} 
  Q_{u+1}\cdots Q_m P^k Q_1 \cdots Q_u
  \\ 
  &\stackrel{(\varsigma^\ast)^{-u} \circ \tilde{T}_s \circ (\varsigma^\ast)^{u}}{\longmapsto}
  (P^{-u} D_s P^u)Q_{u+1}\cdots Q_m P^k Q_1 \cdots Q_u(P^{-u} D_s P^u)^{-1}
  \\
  &= \prod_{i=u+1}^{m} (P^{-i+1} D_s P^{i-1}) Q_i (P^{-i} D_s P^{i})^{-1}
     \cdot (P^{-m} D_s P^{m}) P^k D_s^{-1} 
  \\
  & \qquad \qquad 
     \cdot
     \prod_{i=1}^{u} (P^{-i+1} D_s P^{i-1}) Q_i (P^{-i} D_s P^{i})^{-1}
  \\
  &= Q_{u+1}'\cdots Q_m' P^k Q_1' \cdots Q_u',
\end{align*} 
where $Q_i'= (P^{-i+1} D_s P^{i-1}) Q_i (P^{-i} D_s P^{i})^{-1}$.
On the other hand we have
\begin{align*}
  L(x) &\stackrel{T^\ast_s}{\longmapsto}
  D_s L(x) D_s^{-1}
  \stackrel{\eqref{eq:DLD}}{=} Q_1'\cdots Q_m' P^k
  \\
  &\stackrel{e^\ast_u}{\longmapsto} 
  Q_{u+1}' \cdots Q_m' P^k Q_1' \cdots Q_u'. 
\end{align*}
Thus the claim is obtained.
\end{proof}

From Lemma~\ref{lem:c-snake} and \ref{lem:snake-h} we obtain
 
\begin{prop}\label{prop:snake}
The following diagram is commutative:
\begin{align}\label{eq:comm-T}
\xymatrix{
\psi^{-1}(f)  \ar[d]_{T_s} \ar[r]^{\phi \qquad \qquad} 
& \Pic^{g}(C_f) \times \mathcal{S}_f \times R_O \times R_A~~\ \ar[d]^{(id,\,t_s,\,id, \,id)} 
\\
\psi^{-1}(f) ~~ \ar[r]_{\phi \qquad \qquad} 
& \Pic^{g}(C_f) \times \mathcal{S}_f \times R_O \times R_A, \\
}
\end{align}
where $t_s := t_s(t)$ acts on $\mathcal{S}_f$ by
\begin{align}
  \label{eq:snake-c}
  t_s(c_{\ell}) = 
  \begin{cases}
  t \, c_{\ell} & \text{ if } s=\ell \\
  \displaystyle{\frac{1}{t}}\, c_{\ell}  & \text{ if } s=\ell +1\\
  c_{\ell} & \text{ otherwise}.
  \end{cases}
\end{align}
\end{prop}

Thus the snake path actions $T_s(t)$ act transitively on $\mathcal{S}_f$.

\section{Theta function solution to initial value problem}
\label{sec:theta}

\subsection{Riemann theta function}
Fix $f \in \oV$. 
We fix a universal cover of $C_f$ and $P_0 \in C_f$, 
and define the Abel-Jacobi map $\iota$ by
$$
  \iota: C_f \to \C^g; ~ X \mapsto \left(\int_{P_0}^{X} \omega_1,\ldots,\int_{P_0}^{X} \omega_g \right),
$$  
where $\omega_1,\ldots,\omega_g$ is a basis of holomorphic differentials on $C_f$.
We also write $\iota$ for the induced map $\mathrm{Div}^0(C_f) \to \C^g$.
Let $\Omega$ be the period matrix of $C_f$, and write 
$\Theta(z ) := \Theta(z;\Omega)$ for the Riemann theta function:
$$
  \Theta(z;\Omega) 
  = 
  \sum_{{\bf m}\in \Z^g} 
  \exp \left(\pi \sqrt{-1} 
  {\bf m} \cdot (\Omega {\bf m}+ 2 z)\right),
  \quad z \in \C^g.  
$$
It is known to satisfy the quasi-periodicity:
$$
  \Theta(z + {\bf m} + {\bf n} \Omega)
  =  \exp \left(- \pi \sqrt{-1}\, {\bf n} \cdot (\Omega {\bf n} + 2z) \right)
     \Theta(z),
$$
for ${\bf m}, {\bf n} \in \Z^g$.

Let $\mathcal{D}_0 \in \mathrm{Div}^{g-1}(C_f)$ be the Riemann characteristic, 
that is, $2 \mathcal{D}_0$ is linearly equivalent to the 
canonical divisor $K_{C_f}$. 
Then the well known properties of the Riemann theta function gives

\begin{lem}
For any point $Y \in C_f$, the function 
$$
X \mapsto \Theta(\iota(X-Y)) 
$$
is a section of the line bundle $\OO(\D_0+Y)$ which has zeroes exactly at $\D_0$ and at $Y$, and no poles. For any positive divisor $\D$ of degree $g$, the function
$$
X \mapsto \Theta(\iota(X - \D + \D_0))
$$
is a section of the line bundle $\OO(\D)$ which has zeroes exactly at $\D$ and no poles.
\end{lem}

\subsection{The inverse of $\phi$}

For a positive divisor $\D$ of degree $g$, define the function $\psi_i ~(i \in \Z)$ on $C_f$ by
$$
\psi_i(X) 
= \frac{\Theta(\beta - \iota(\OO_{n-i}) + \iota(X - O_n)) 
  \prod_{\ell=i+1}^n \Theta(\iota(X - O_\ell))}{\Theta(\beta + \iota(X-O_n)) 
  \Theta(\iota(X-P))^{n-i}},
$$
where $\beta = \iota(\D_0 + O_n - \D)$.

\begin{lem}\label{lem:psi-h}
$\psi_i(X)$ is a meromorphic function on $C_f$.
When $\D$ is the positive divisor defined by Definition~\ref{def:D}, we have 
\begin{align}\label{eq:h-div}
  (\psi_i) = (h_i) 
  = 
  \mathcal{D}_i + \sum_{j=i+1}^n O_j 
  - \mathcal{D} - (n-i) P.
\end{align}
Thus $\psi_i(X)$ is equal to $h_i(X)$ up to a scalar.
\end{lem}
\begin{proof}
To check that $\psi_i(X)$ is a meromorphic function on $C_f$ we use the functional equation for $\Theta$, and note that the multiset of points (with signs and multiplicities) that appear in the numerator is equal to the same for the denominator.  
This multiset in additive notation is $\beta - (n-i)P + (n-i+1)X - O_n$.
For the second statement, we note that $\psi_i(X)$ has zeroes at $O_{i+1},O_{i+2},\ldots,O_{n}$, a pole of order $n-i$ at $P$, and also poles at $\D$.  By \eqref{eq:hdiv}, we have $(\psi_i) = (h_i)$.
\end{proof}
 
Now we study the inverse of $\phi$.  
We focus on the $\Z^m$ action. 
For $t=(t_1,t_2,\ldots,t_m) \in \Z^m$, 
let $q^t := (q_{ji}^t)_{ji} \in \psi^{-1}(f)$ be a configuration at time $t$.
For $L^t(x) := \alpha(q^t)$, let $g^t = (g^t_1,\ldots, g^t_n)^\bot$ be 
its eigenvector and set $h^t_i := g^t_i/g^t_n$.
Let $\phi(q^t)$ be 
$$
  (\D^t, (c^t_1,\cdots,c^t_M),(A_1,\ldots,A_m),(O_1,\ldots,O_N)),
$$
where $\D^t = \D - \sum_{j=1}^m t_j \mathcal{A}_j$ 
for some positive divisor $\D$ 
of degree $g$.
Define
$\be_j := (\underbrace{1,\ldots,1}_{j},\underbrace{0,\ldots,0}_{m-j})$
for $j=0,\ldots,m-1$, and recursively set 
$\be_{j+m} = (1,\ldots,1) + \be_j$.
Define functions $\theta^t_i$ and $\Psi_{i}^{t}$ on $C_f$ by
$$
\theta^{t}_{i}(X) = \Theta(\beta^t + \iota(X-O_n 
- \OO_{n-i})),
$$
where $\beta^t = \iota(\D_0 + O_n - \D^t)$, and by 
$$
\Psi_{i}^{t+\be_j}(X) 
= 
\frac{\theta_{i-N}^{t+\be_{j-1}}(X) \,\theta_{i}^{t+\be_{j}}(X)}
{\theta_{i}^{t+\be_{j-1}}(X) \, \theta_{i-N}^{t+\be_{j}}(X)}. 
$$
By Lemma \ref{lem:psi-h}, there is a constant $b_i^{t+\be_j}$ depending on $C_f$, such that
we have
$$
  \Psi_{i}^{t+\be_j}(X) 
  = 
  b_i^{t+\be_j} 
  \frac{h_{i-N}^{t+\be_{j-1}}(X) \, h_{i}^{t+\be_{j}}(X)}
  {h_{i}^{t+\be_{j-1}}(X) \, h_{i-N}^{t+\be_{j}}(X)}.
$$

\begin{lem}(Cf. Lemma~3.1 \cite{Iwao10})\label{lem:Psi-OP}
We have 
\begin{align}
  \label{eq:Psi-O}
  &\Psi_{i-1}^{t+\be_j}(P) = \Psi_{i}^{t+\be_j}(O_i) 
  = b_i^{t+\be_j} \frac{q_{j,i-N}^t}{q_{j,i}^t},
  \\
  \label{eq:Psi-P}
  &\Psi_{i}^{t+\be_j}(P) 
  = 
  b_i^{t+\be_j} 
  \frac{d_{i-N-1}^{t+\be_{j}} \, d_{i}^{t+\be_{j}}}
   {d_{i-1}^{t+\be_{j}} \, d_{i-N}^{t+\be_{j}}},
\end{align}
where $d_{i}^{t}$ is a coefficient of the leading term of 
$h_{i}^{t}(X)$ around $X=P$.
\end{lem}

\begin{proof}
The first equality of \eqref{eq:Psi-O} follows from
$\theta^t_i(O_i) = \theta^t_{i-1}(P)$ and $O_i = O_{i-N}$. 
By Lemma~\ref{lem:vertical-L} we have 
$g^{t+\be_{j-1}} = Q_j^t \, g^{t+\be_{j}}$ which gives
$$
  \Psi_{i}^{t+\be_j}(X) 
  = 
  b_i^{t+\be_j} 
  \frac{(h_{i-N-1}^{t+\be_{j}}(X) + q_{j,i-N} \, h_{i-N}^{t+\be_{j}}(X))  
  h_{i}^{t+\be_{j}}(X)}
  {(h_{i-1}^{t+\be_{j}}(X) + q_{j,i} \, h_{i}^{t+\be_{j}}(X))
   h_{i-N}^{t+\be_{j}}(X)}.
$$
Then, from \eqref{eq:h-div} the second equalities of \eqref{eq:Psi-O}
and \eqref{eq:Psi-P} follow.
\end{proof}

\begin{lem} We have 
\begin{align}\label{eq:sneke-d}
  \frac{d_{i}^{t+\be_{j}}}{d_{i+1}^{t+\be_{j}}}
  = 
  c^t_{i+j}, \qquad i+j \equiv \ell \mod M,
\end{align}  
where we extend $c_i^t$ to $i \in \Z$ by setting $c_i^t := c_\ell^t$ 
if $i \equiv \ell \mod M$ for $1 \leq \ell \leq M$. 
\end{lem}

\begin{proof}
It is enough to show the case of $j=0$, $1 \leq i \leq M$.
Then the other cases follow from Lemma~\ref{lem:pi-c}.

We will show that the coefficient of the leading term in
$\frac{g_i}{g_{i+1}}$ at $X=P$ is equal to $c_i$.
To obtain the ratio $\frac{g_i}{g_{i+1}}$ we may choose to compute
maximal minors of $L(x)-y$ with respect to any row. Let us choose
row $i+1$. Then by Theorem \ref{thm:minor}, we are counting families
of paths that do not start at source $i+1$ and do not end at sink $i$
(for $g_i$) or sink $i+1$ (for $g_{i+1}$).

Since we are computing at $X=P$, the only terms that contribute are
the ones lying on the upper hull of the Newton polygon, that is, the
edge with slope $-n/(m+k)$.
In terms of path families this means that we only take paths that are
as close to a snake path as possible. In the case of $g_{i+1}$ this
means that we take a subset of closed snake paths. Each of them picks
up no weight, and thus overall we just get a constant.
In the case of $g_{i}$ one of the closed staircase paths skips one
term, and thus we get a long path starting at source $i$, winding
around the torus and ending at sink $i+1$.
Such path essentially by definition picks up weight $c_i$.
The other snake paths may or may not appear, thus creating only a
constant factor in front.
\end{proof}

Note that \eqref{eq:sneke-d} is compatible with the snake path
actions in Lemma~\ref{lem:snake-h}.

\begin{thm}\label{thm:N=1} 
When $N=1$, the inverse map of $\phi$ is given by
$$  
  q_{j,i}^t 
  =
  f_c^{-\frac{1}{M}} \cdot a_j \cdot c_{i+j-1}^t \cdot 
  \frac{\theta_{i}^{t+\be_{j}}(P) \, \theta_{i-1}^{t+\be_{j-1}}(P)}
  {\theta_{i}^{t+\be_{j-1}}(P) \, \theta_{i-1}^{t+\be_{j}}(P)},
$$
where
\begin{align}\label{eq:q-a}
a_j = x(A_j)^{\frac{1}{n}} \cdot 
\mathrm{exp}\left[\frac{2 \pi \sqrt{-1}}{n} \, b \cdot \iota(P-A_j)\right],
\end{align}
and 
$b \in \Z^g$ is determined by
$a,b \in \Z^g$ such that $a + b \Omega := \iota(\OO_n)$.
\end{thm}

\begin{proof}
Let $p_{j,i}^t = q_{j,i}^t q_{j,i-1}^t \cdots q_{j,i-N+1}^t$.
By Lemma~\ref{lem:Psi-OP} we get the equations 
$$
  \alpha^t_j := \frac{p_{j,i}^t}{\Psi_{i}^{t+\be_{j}}(P)}
  \frac{d_{i}^{t+\be_{j}}}{d_{i-N}^{t+\be_{j}}}
  = 
  \frac{p_{j,i-1}^t}{\Psi_{i-1}^{t+\be_{j}}(P)}
  \frac{d_{i-1}^{t+\be_{j}}}{d_{i-N-1}^{t+\be_{j}}}
  = 
  \cdots.
$$
We show that $\alpha^t_j$ does not depend on $t$.
Consider the product 
$$
  x(A_j) = p_{j,i}^t p_{j,i+N}^t \cdots p_{j,i+(n'-1)N}^t,
$$
which is equal to 
$$
  (\alpha^t_j)^{n^\prime} 
  \frac{d_{i-N}^{t+\be_{j}}}{d_{i+n-N}^{t+\be_{j}}}\,
  \frac{\theta_{i-N}^{t+\be_{j-1}}(P) \,\theta_{i+n-N}^{t+\be_{j}}(P)}
  {\theta_{i+n-N}^{t+\be_{j-1}}(P) \,\theta_{i-N}^{t+\be_{j}}(P)}.
$$
Since $\OO_n$ is equivalent to $0$ in $\Pic^0(C_f)$,
there exists $a,b \in \Z^g$ such that 
$a + b \Omega = \iota(\OO_n)$.
Due to the quasi-periodicity of $\Theta(z)$, we have
\begin{align}\label{theta-quasi1}
  \theta_{i+n}^t(P) 
  = \exp \left[-2 \pi \sqrt{-1} b \cdot
    \left(\beta^t - \iota(\OO_{n-i}-\OO_1) + b \Omega / 2 \right)    
    \right]\cdot \theta_i^t(P).
\end{align} 
By using \eqref{eq:sneke-d} and \eqref{theta-quasi1} we obtain
$\alpha^t_j$ independent of $t$ as
$$
  \alpha_j^t
  = 
  f_c^{-\frac{N}{M}}
  a_j^N.
$$ 
Hence we get 
\begin{align}\label{eq:p-theta}
  p_{j,i}^t 
  =
  f_c^{-\frac{N}{M}} a_j^N
  \cdot \prod_{\ell=1}^N c_{i+j-\ell}^t \cdot
  \frac{\theta_{i}^{t+\be_{j}}(P) \, \theta_{i-N}^{t+\be_{j-1}}(P)}
  {\theta_{i}^{t+\be_{j-1}}(P) \, \theta_{i-N}^{t+\be_{j}}(P)}.
\end{align}   
When $N=1$, it is nothing but the claim.
\end{proof}

\subsection{Conditional solution for $N>1$}

When $N > 1$, by factorizing \eqref{eq:p-theta} we obtain 
\begin{align}\label{eq:q-gamma}
  q_{j,i}^t 
  =
  f_c^{-\frac{1}{M}} \cdot \gamma_{j,i} \cdot  a_j \cdot 
  c_{i+j-1}^t
  \frac{\theta_{i}^{t+\be_{j}}(P) \, \theta_{i-1}^{t+\be_{j-1}}(P)}
  {\theta_{i}^{t+\be_{j-1}}(P) \, \theta_{i-1}^{t+\be_{j}}(P)}.
\end{align}
Here $\gamma_{j,i}$ satisfies
\begin{align}\label{gamma-period}
  \prod_{\ell=1}^N \gamma_{j,i+\ell} = 1,
\end{align}
which denotes that $\gamma_{j,i} = \gamma_{j,i+N}$.

\begin{lem}
For $i$ such that $i \equiv r \mod N$ we have
\begin{align}\label{gamma-m}
  \prod_{j=1}^m \gamma_{j,i} 
  =
  \sigma_r^\frac{1}{n'} \cdot \prod_{j=1}^m a_j^{-1} \cdot 
  \mathrm{exp}\left(2 \pi \sqrt{-1} (b' - b \, \frac{k'}{n'}) 
  \cdot \iota(P-O_r)\right).
\end{align}
Here $b'$ is given by $a', b' \in \Z^g$ such that
$a' + b' \Omega := \iota(\mathcal{A}_m + \OO_k)$.
\end{lem}

\begin{proof}
From \eqref{def:sigma_r} and \eqref{eq:q-gamma} it follows that 
\begin{align}\label{sigma-gamma}
\sigma_r = \left(\prod_{j=1}^m \gamma_{j,i} \,a_j \,f_c^{-\frac{1}{M}}\right)^{n'}
\cdot \prod_{j=1}^m \prod_{p=0}^{n'-1} c_{j+i+pN-1}^t \cdot
\prod_{j=0}^{n'-1}   
\frac{\theta^t_{i+jN-1}(P) \, \theta_{i+jN}^{t+\be_{m}}(P)}
  {\theta_{i+jN-1}^{t+\be_{m}}(P) \, \theta^t_{i+jN}(P)},
\end{align}
for $r=1,\ldots,N$ and $i \equiv r \mod N$.

First we show that the second factor of \eqref{sigma-gamma} is equal to
$f_c^{\frac{m n'}{M}}.$
It is enough to prove in the case of $i=1$, where the second factor is
$
\prod_{p=0}^{n'-1}(c_{1+pN} \cdots c_{m+pN}).
$  
When $n'=k'$, $M$ is a divisor of $m$ and we have 
$c_{1+pN} \cdots c_{m+pN} = f_c^{m/M}$.
When $n' \neq k'$, define an automorphism $\nu$ of 
$\mathcal{N}' := \{0,\ldots,n'-1\}$ by 
$\nu : p \mapsto p+n'-k' \mod n'$.
Since $n'$ and $k'$ are coprime, a set 
$\{\nu^\ell(1) ~|~ \ell \in \mathcal{N}' \}$ coincides with 
$\mathcal{N}'$.
Thus, from $M=\gcd(n,k+m)$ and $m+pN \equiv \nu(p) N \mod M$, 
it follows that
the product $\prod_{p=0}^{n'-1}(c_{1+pN} \cdots c_{m+pN})$  
is reordered to be $\prod_{j=1}^{mn'} c_j$. 
Further, since $M$ is a divisor of $mn'$, we obtain the claim.
Next, we can show that the third factor of \eqref{sigma-gamma} is equal to
$
\mathrm{exp}\left(2 \pi \sqrt{-1} (-b'n' + b k') \cdot \iota(P-O_r)\right),
$
by using \eqref{theta-quasi1} and 
$$
\theta_{i}^{t+\be_m}(P) = \exp \left[-2 \pi \sqrt{-1} b' \cdot \left(\beta^t 
- \iota(\OO_{n-i+k} - \OO_1) + b' \Omega / 2 \right) \right]\cdot \theta_{i-k}^t(P).
$$ 
Here $b'$ is given by $a', b' \in \Z^g$ such that
$a' + b' \Omega := \iota(\mathcal{A}_m + \OO_k)$.
Finally from \eqref{sigma-gamma} we obtain 
$$
  \left(\prod_{j=1}^m \gamma_{j,i}\right)^{n'} 
  = 
  \sigma_r \cdot \prod_{j=1}^m a_j^{-n'} \cdot
  \mathrm{exp}\left(2 \pi \sqrt{-1} (n'b' - b k') \cdot
   \iota(P-O_r)\right),
$$
and the claim follows.
\end{proof}

Unfortunately, we have not been able to obtain $\gamma_{j,i}$ from
\eqref{gamma-period} and \eqref{gamma-m}.
Nevertheless, if we assume that $\gamma_{j,i}$ is constant on $\Pic^g(C_f)$, then we obtain the following conditional result.

\begin{prop}\label{prop:N>1}
Suppose that $\gamma_{j,i}$ is a constant function on $\Pic^g(C_f)$. 
Then we have
\begin{align}\label{eq:gamma}
  \gamma_{j,i} 
  = 
  \sigma_r^\frac{1}{n'm} \cdot \prod_{j=1}^m a_j^{-\frac{1}{m}} 
  \cdot 
  \mathrm{exp}\left(\frac{2 \pi \sqrt{-1}}{n'm} (n'b' - k' b)
   \cdot \iota(P-O_r)\right).
\end{align}
In particular, the inverse of $\phi$ is given by
$$
  q_{j,i}^t 
  =
  Q \cdot a_j \cdot o_i \cdot c_{i+j-1}^t \, 
  \frac{\theta_{i}^{t+\be_{j}}(P) \, \theta_{i-1}^{t+\be_{j-1}}(P)}
  {\theta_{i}^{t+\be_{j-1}}(P) \, \theta_{i-1}^{t+\be_{j}}(P)},
$$
where $a_j$ is defined by \eqref{eq:q-a}, and 
\begin{align*}
&Q = f_c^{-\frac{1}{M}} \cdot \prod_{j=1}^m x(A_j)^{-\frac{1}{nm}},
\\
&o_i = \left((-1)^{n'+1} \frac{y(O_r)^{n'}}{x(O_r)^{k'}}
\right)^{\frac{1}{n'm}}
\cdot  
\mathrm{exp}\left[\frac{2 \pi \sqrt{-1}}{nm}
\left((n b'- kb) \cdot \iota(P-O_i) - b \cdot \iota(\mathcal{A}_m) \right) 
\right], 
\\ 
&\hspace*{12cm} i \equiv r \mod N.
\end{align*}
\end{prop}

\begin{proof}
To have \eqref{eq:q-gamma} compatible with the snake path action 
\eqref{eq:snake-q} and \eqref{eq:snake-c},
$\gamma_{i,j}$ has to be constant on $\mathcal{S}_f$.
So $\gamma_{j,i}$ is regarded as a function on $R_A \times R_O$.
If $\gamma_{j,i}$ is not constant on $R_A$, \eqref{gamma-m} implies
that $\gamma_{j,i}$ depends on $a_j^{-1}$, but this contradicts 
\eqref{gamma-period}. 
Thus $\gamma_{j,i}$ is constant on $R_A$, and we have \eqref{eq:gamma}
which fulfills \eqref{gamma-period}.
By substituting \eqref{eq:gamma} in \eqref{eq:q-gamma} and 
using \eqref{eq:o_r}, we obtain the final claim.
\end{proof}

\section{Octahedron recurrence}
\label{sec:fay}

We shall prove that the function $\theta_i^t(P)$ 
satisfies a family of octahedron recurrences \cite{Spe},
as a specialization of Fay's trisecant identity.
We follow the definitions in \S \ref{sec:theta}.
In the following we assume $N = \gcd(n,k) = 1$ and 
write $O$ for the unique special point $O_1$ over $(0,0) \in \tilde{C}_f$.

\subsection{Fay's trisecant identity}

We introduce Fay's trisecant identity in our setting.
For $\alpha, \,\beta \in \R^g$,
we define a generalization of the Riemann theta function:
\begin{align}\label{def:gen-theta}
  \Theta[\alpha,\beta](z) 
  :=
  \exp \left(\pi \sqrt{-1} \beta \cdot (\beta \Omega + 2 z + 2 \alpha
  \right)
  \cdot \Theta(z + \Omega \beta + \alpha).
\end{align}
We call $[\alpha,\beta]$ {\it a half period} when
$\alpha, \, \beta \in (\Z/2)^g$. Furthermore, a half period $[\alpha,\beta]$
is {\it odd} when  
$\Theta[\alpha,\beta](z)$ is an odd function of $z$,
that is, $\Theta[\alpha,\beta](-z) = - \Theta[\alpha,\beta](z)$.
We note that the Riemann theta function itself is an even function:
$\Theta(z) = \Theta(-z)$.
It is easy to check that a half period $[\alpha,\beta]$ is odd
if and only if $4 \alpha \cdot \beta \equiv 1 \mod 2$.

\begin{thm}\cite[\S II]{Fay73}
\label{thm:Fay}
For four points $P_1,P_2,P_3,P_4$ on the 
universal cover of $C_f$, $z\in \C^g$, and 
an odd half period $[\alpha,\beta]$, the formula
\begin{align*}
&\Theta(z+\iota(P_3-P_1)) \, \Theta(z+\iota(P_4-P_2)) \,
\Theta[\alpha,\beta](\iota(P_2-P_3)) \,\Theta[\alpha,\beta](\iota(P_4-P_1))\\
&+\Theta(z+\iota(P_3-P_2)) \, \Theta(z+\iota(P_4-P_1)) \,
\Theta[\alpha,\beta](\iota(P_1-P_3)) \, \Theta[\alpha,\beta](\iota(P_2-P_4))\\
&=
\Theta(z+\iota(P_3+P_4-P_1-P_2)) \, \Theta({\bf z}) \,
\Theta[\alpha,\beta](\iota(P_4-P_3)) \, \Theta[\alpha,\beta](\iota(P_2-P_1)).
\end{align*}
holds.
\end{thm}

\subsection{$m=2$ case}

When $m=2$, the vertical actions $e_u ~(u=1,2)$ are written as 
difference equations expressed as 
\begin{align}
  \label{eq:evolm=2-1}
  &q_{2,i}^t q_{1,i-k}^t = q_{1,i}^{t+\be_1} q_{2,i}^{t+\be_1},
  \\
  \label{eq:evolm=2-2}
  &q_{2,i+1}^t + q_{1,i-k}^t = q_{1,i+1}^{t+\be_1} + q_{2,i}^{t+\be_1}, 
  \\
  \label{eq:evolm=2-3}
  &q_{j,i-k}^t = q_{j,i}^{t+\be_2} \qquad (j=1,2).
\end{align}

For simplicity we write $\theta_i^t$ for 
$\theta_i^t(P) = \Theta(\beta^t + \iota((n-i-1) (O-P)))$. 
By construction the theta function solution of $q_{j,i}^t$ 
(Theorem \ref{thm:N=1}),
\begin{align}\label{eq:theta-sol}
  q_{j,i}^t 
  =
  f_c^{-\frac{1}{M}} \cdot a_j \cdot c_{i+j-1}^t \cdot 
  \frac{\theta_{i}^{t+\be_{j}} \, \theta_{i-1}^{t+\be_{j-1}}}
  {\theta_{i}^{t+\be_{j-1}} \, \theta_{i-1}^{t+\be_{j}}},
\end{align}  
satisfies \eqref{eq:evolm=2-1}--\eqref{eq:evolm=2-3}.

\begin{thm}\label{thm:octahedron}
For any $t \in \Z^2$ and $i \in \Z$, the $\theta_i^t$ satisfy  
an octahedron recurrence relation, 
\begin{align}\label{eq:m=2-oct}
  a_2 \,\theta_{i+1}^{t+\be_2} \theta_i^{t+2 \be_1}
  - a_1 \,\theta_{i+1}^{t+2 \be_1} \theta_i^{t+\be_2}
  =
  c \,\theta_{i}^{t+\be_1 + \be_2} \theta_{i+1}^{t+\be_1}.
\end{align}  
Here $c$ is a constant given by
$$
  c = a_2 \, \frac{\Theta(p_0+\iota(A_1-A_2)) \Theta(p_0+\iota(O-P))}
                  {\Theta(p_0+\iota(O-A_2)) \Theta(p_0+\iota(A_1-P))},
$$ 
where $p_0$ is a zero of the Riemann theta function: $\Theta(p_0) = 0$. 
\end{thm}

\begin{proof}
By setting $(P_1,P_2,P_3,P_4) = (A_2,O,P,A_1)$ and using \eqref{def:gen-theta},
we obtain:
\begin{align}\label{Fay-basic}
  \begin{split}
  &T_1 \, \Theta(z+\iota(P-A_2)) \, \Theta(z+\iota(A_1-O)) 
  + T_2 \, \Theta(z+\iota(P-O)) \, \Theta(z+\iota(A_1-A_2))\,
    \\
  &\qquad = T_3 \, \Theta(z+\iota(P+A_1-A_2-O)) \, \Theta(z) \,
  \end{split}
\end{align} 
where $p := \Omega \beta + \alpha$, and 
\begin{align}\label{eq:T}
  \begin{split}
  &T_1 := \Theta(p+\iota(O-P)) \, \Theta(p+\iota(A_1-A_2)),
  \\
  &T_2 := \mathrm{e}^{4 \pi \sqrt{-1} \beta \cdot \iota(A_2-A_1)} \,
     \Theta(p+\iota(A_2-P)) \, \Theta(p+\iota(O-A_1)),
  \\
  &T_3 := \Theta(p+\iota(A_1-P)) \, \Theta(p+\iota(O-A_2)).
  \end{split}
\end{align}
By setting $z = \beta^t + \iota((n-i-1)(O-P) + 2 \mathcal{A}_1)$, 
\eqref{Fay-basic} turns out to be
\begin{align}\label{eq:oct}
  T_1 \, \theta_i^{\be_1+\be_2} \, \theta_{i+1}^{\be_1}
  + T_2 \, \theta_{i+1}^{2\be_1} \, \theta_{i}^{\be_2}
  = 
  T_3 \, \theta_{i+1}^{\be_2} \, \theta_{i}^{2 \be_1}.
\end{align}

\begin{lem}\label{lem:T23}
We have 
$$
  \frac{T_2}{T_3} = \frac{a_1}{a_2}.
$$
\end{lem}

\begin{proof}
For $q^0 \in \psi^{-1}(f)$, 
take $t_0 \in \Z^m$ such that $\beta^{t_0} \in \C^g$ is a zero of 
the Riemann theta function,
which is always possible by choosing $q^0$ appropriately.
We write $-p_0$ for such $\beta^{t_0}$.
Then we have $\theta^{t_0}_{n-1} = \Theta(-p_0) = 0$, and 
$q_{1,n}^{t_0} = q_{2,n}^{t_0-\be_1} = 0$.
From \eqref{eq:evolm=2-1} and \eqref{eq:evolm=2-2},
we obtain 
\begin{align}\label{eq:rel-at-t0} 
  q_{2,n-1}^{t_0-\be_1} q_{1,n-1-k}^{t_0-\be_1} 
  = q_{1,n-1}^{t_0} q_{2,n-1}^{t_0},
  \qquad 
  q_{1,n-k-1}^{t_0-\be_1} = q_{2,n-1}^{t_0}.
\end{align}

On the other hand, 
when $z = p_0 + \iota(A_2-P)$, \eqref{Fay-basic} becomes
$$
  T_2 \, \Theta(p_0+\iota(A_2-O)) \, \Theta(p_0+\iota(A_1-P))
  = T_3 \, \Theta(p_0+\iota(A_1-O)) \, \Theta(p_0+\iota(A_2-P)). 
$$   
It is rewritten as (using that $\Theta(z)$ is even function) 
$$
  \frac{T_2}{T_3}
  = 
  \frac{\theta^{t_0 + \be_1}_{n-2} \, \theta^{t_0 + \be_2-\be_1}_{n-1}}
       {\theta^{t_0 + \be_2-\be_1}_{n-2} \, \theta^{t_0 + \be_1}_{n-1}}
  =
  \frac{a_1}{a_2} \cdot \frac{q_{2,n-1}^{t_0-\be_1}}{q_{1,n-1}^{t_0}},
$$
where we use \eqref{eq:theta-sol} to get the last equality.
It follows from \eqref{eq:rel-at-t0} that 
$q_{2,n-1}^{t_0-\be_1}/q_{1,n-1}^{t_0} = 1$, and we obtain the claim. 
\end{proof}

We continue the proof of the theorem.
By setting
$z = p_0 + \iota(O-P)$ in \eqref{Fay-basic}, we obtain
$$
  T_1 \, \Theta(p_0+\iota(O-A_2)) \, \Theta(p_0+\iota(A_1-P))
  = T_3 \, \Theta(p_0+\iota(A_1-A_2)) \, \Theta(p_0+\iota(O-P)). 
$$   
Using this and the above lemma, \eqref{eq:oct} is shown to be
\begin{equation*}
  c \, \theta_i^{\be_1+\be_2} \, \theta_{i+1}^{\be_1}
  + a_1 \, \theta_{i+1}^{2\be_1} \, \theta_{i}^{\be_2}
  = 
  a_2 \, \theta_{i+1}^{\be_2} \, \theta_{i}^{2 \be_1}. \qedhere
\end{equation*}
\end{proof}

Conversely, we have the following. 

\begin{prop}\label{prop:thetaq}
Suppose that $\theta_i^t$ satisfy \eqref{eq:m=2-oct}.  Then $q_{j,i}^t$ defined by \eqref{eq:theta-sol}
  satisfies \eqref{eq:evolm=2-1}--\eqref{eq:evolm=2-3}.
\end{prop}
\begin{proof}
It is very easy to see that \eqref{eq:theta-sol} satisfies 
\eqref{eq:evolm=2-1} and \eqref{eq:evolm=2-3}.
We check \eqref{eq:evolm=2-2}. 
We consider a ratio
$$
  f_i^t 
  := \frac{q_{2,i+1}^t -q_{1,i+1}^{t+\be_1}}
       {q_{2,i}^{t+\be_1} - q_{1,i}^{t+\be_2}}.
$$ 
By substituting \eqref{eq:theta-sol} in $f_i^t$ we obtain
\begin{align}
  f_i^t
  =
  \frac{
  a_2 c_{i+2} \frac{\theta_{i+1}^{\be_2} \, \theta_i^{\be_1}}
                 {\theta_{i+1}^{\be_1} \, \theta_i^{\be_2}}
  - a_1 c_{i+1}^{\be_1} \frac{\theta_{i+1}^{2 \be_1} \, \theta_{i}^{\be_1}}  
                 {\theta_{i+1}^{\be_1} \, \theta_{i}^{2 \be_1}}
  }
  {
  a_2 c_{i+1}^{\be_1} \frac{\theta_{i}^{\be_1+\be_2} \, \theta_{i-1}^{2 \be_1}}
                 {\theta_{i}^{2 \be_1} \, \theta_{i-1}^{\be_1+\be_2}}
  - a_1 c_{i}^{\be_2} \frac{\theta_{i}^{\be_1+\be_2} \, \theta_{i-1}^{\be_2}}
                 {\theta_{i}^{\be_2} \, \theta_{i-1}^{\be_1+\be_2}}
  }
  =
  \frac{
  \theta_i^{\be_1} \, \theta_{i-1}^{\be_1+\be_2}
  \left( a_2 \theta_{i+1}^{\be_2} \, \theta_{i}^{2 \be_1}
         - a_1 \theta_{i+1}^{2 \be_1} \, \theta_{i}^{\be_2} \right)
  }
  {
  \theta_i^{\be_1 + \be_2} \, \theta_{i+1}^{\be_1}
  \left( a_2 \theta_{i}^{\be_2} \, \theta_{i-1}^{2 \be_1}
         - a_1 \theta_{i}^{2 \be_1} \, \theta_{i-1}^{\be_2} \right)
  },
\end{align}  
where we omit the superscripts $t$ for simplicity.
At the second equality we have canceled all the $c_i$ 
using $c_i^{\be_u} = c_{i+u}$.
Due to \eqref{eq:m=2-oct}, we obtain $f_i^t = 1$.
Thus, using \eqref{eq:evolm=2-3}, we obtain \eqref{eq:evolm=2-2} .
\end{proof}

\subsection{General $m$ case}

The vertical actions $e_u ~(u=1,\ldots,m)$ are expressed as
matrix equations:
\begin{align}
  \label{Q-m-1}
  Q_1^{t+\be_{u}}\cdots Q_m^{t+\be_{u}}
  =
  Q_{u+1}^t \cdots Q_m^t P(x)^k Q_1^t \cdots Q_{u}^t P(x)^{-k}.
\end{align}
Among the family of difference equations we will use the following 
ones later:  
\begin{align}
  \label{eq:evol1}
  &\prod_{j=1}^m q_{j,i}^{t+\be_{u}} 
  = 
  \prod_{j=u+1}^m q_{j,i}^t \cdot \prod_{j=1}^u q_{j,i-k}^t,
  \\
  \label{eq:evol2}
  &\sum_{j=1}^m q_{1,i}^{t+\be_u} \cdots q_{j-1,i}^{t+\be_u} \,
   q_{j+1,i-1}^{t+\be_u} \cdots q_{m,i-1}^{t+\be_u}
  =
  \sum_{j=1}^m q_{s(1,i)}^{t} \cdots q_{s(j-1,i)}^{t} \,
   q_{s(j+1,i-1)}^{t} \cdots q_{s(m,i-1)}^{t}, 
\end{align}
for $u=1,\cdots,m$. Here we define
$$
  s(j,i) = \begin{cases} 
           (j+u,i) & j=1,\ldots m-u, \\
           (j+u-m,i-k) & j=m-u+1,\ldots,m.
           \end{cases}
$$

\begin{thm}\label{thm:octahedron-general}
For any $t \in \Z^m$, $i \in \Z$ and $1 \leq p < r \leq m$, 
the $\theta_i^t$ satisfy  
an octahedron recurrence relation, 
\begin{align}\label{eq:m-oct}
  \delta_{p,r} \,\theta_{i+1}^{t+\be_{p-1}+\be_{r-1}}
  \theta_{i}^{t+\be_{p}+\be_{r}} 
  + a_p \,\theta_{i+1}^{t+\be_{p}+\be_{r-1}} \theta_i^{t+\be_{p-1}+\be_{r}}
  =
  a_r \,\theta_{i+1}^{t+\be_{p-1}+\be_{r}} \theta_i^{t+\be_{p}+\be_{r-1}}.
\end{align}  
Here $\delta_{p.r}$ is a constant given by
$$
  \delta_{p,r} = a_r \, \frac{\Theta(p_0+\iota(A_p-A_r)) \Theta(p_0+\iota(O-P))}
                  {\Theta(p_0+\iota(O-A_r)) \Theta(p_0+\iota(A_p-P))},
$$ 
and $p_0$ is a zero of the Riemann theta function: $\Theta(p_0) = 0$. 
\end{thm}

\begin{proof}
The proof is similar to the $m=2$ case. We explain the outline.
In the case $(P_1,P_2,P_3,P_4) = (A_r,O,P,A_p)$ of Theorem~\ref{thm:Fay},
we obtain
\begin{align}\label{Fay-basic-m}
  \begin{split}
  &T_1 \, \Theta(z+\iota(P-A_r)) \, \Theta(z+\iota(A_p-O)) 
  + T_2 \, \Theta(z+\iota(P-O)) \, \Theta(z+\iota(A_p-A_r))\,
    \\
  &\qquad = T_3 \, \Theta(z+\iota(P+A_p-A_r-O)) \, \Theta(z) \,
  \end{split}
\end{align} 
where $T_1$, $T_2$ and $T_3$ are given by \eqref{eq:T}, but 
we replace $A_1$ (resp. $A_2$) with $A_p$ (resp. $A_r$).
By setting 
$z = \beta^t + \iota((n-i-1)(O-P) + \mathcal{A}_p + \mathcal{A}_{r-1})$ at
\eqref{Fay-basic-m}, we get
$$
  T_1 \,\theta_{i+1}^{t+\be_{p-1}+\be_{r-1}}
  \theta_{i}^{t+\be_{p}+\be_{r}} 
  + T_2 \,\theta_{i+1}^{t+\be_{p}+\be_{r-1}} \theta_i^{t+\be_{p-1}+\be_{r}}
  =
  T_3 \,\theta_{i+1}^{t+\be_{p-1}+\be_{r}} \theta_i^{t+\be_{p}+\be_{r-1}}.
$$
We take $t_0 \in \Z^m$ and define $p_0 := - \beta^{t_0} \in \C^g$ 
in the same manner as in the proof of Lemma~\ref{lem:T23}. 
Then $q_{u,n}^{t_0-\be_{u-1}} = 0$ holds for $u=1,\ldots,m$.
From \eqref{eq:evol1} (resp. \eqref{eq:evol2}) 
of $i=n-1$ (resp. $i=n$) and $u=p-1$ or $r-1$, it follows that
$$
  \frac{q_{p,n-1}^{t_0-\be_{p-1}}}{q_{1,n-1}^{t_0}} 
  = \frac{q_{r,n-1}^{t_0-\be_{r-1}}}{q_{1,n-1}^{t_0}} = 1.  
$$
On the other hand, 
by setting $z=p_0 + \iota(A_r-P)$ at \eqref{Fay-basic-m}, we obtain 
$$
  \frac{T_2}{T_3} 
  =
  \frac{\theta^{t_0 + \be_{p} - \be_{p-1}}_{n-2} \, \theta^{t_0 + \be_{r}-\be_{r-1}}_{n-1}}
       {\theta^{t_0 + \be_{r}-\be_{r-1}}_{n-2} \, \theta^{t_0 + \be_{p}-\be_{p-1}}_{n-1}} 
  =
  \frac{q_{r,n-1}^{t_0-\be_{r-1}}}{q_{p,n-1}^{t_0-\be_{p-1}}} \cdot \frac{a_p}{a_r}.
$$
From the above two relations, $T_2/T_3 = a_p/a_r$ follows.
Finally, by setting $z=p_0+\iota(O-P)$ at \eqref{Fay-basic-m}
we obtain the formula of $\delta_{p,r}$. 
\end{proof}

The following conjecture extends Proposition \ref{prop:thetaq}.
\begin{conjecture}
Suppose $\theta_i^t$ satisfy \eqref{eq:m-oct}.  Then $q_{j,i}^t$ defined via \eqref{eq:theta-sol}
satisfy all the difference equations \eqref{Q-m-1}.
\end{conjecture}

We have checked the conjecture for $m \leq 3$.

\section{The transposed network}
\label{sec:transpose}

\subsection{Transposed Lax matrix and spectral curve}
Corresponding to $\tilde Q$ of \eqref{eq:q-tildeq},
we also identify $q \in \mM$ with an $N$-tuple of $n'm \times n'm$ matrices
$\tilde Q := (\tilde Q_i(x))_{i \in [N]}$ as \S\ref{sec:WWonM}.
The matrices $\tilde Q_i(x)$ are given in terms of the $q_{ji}$ by 
\begin{align*}
&\tilde Q_{N+1-i}(x) := 
\tilde P(x) + \mathrm{diag}[q_{m,i},\ldots,q_{1,i},q_{m,i+k},\ldots,q_{1,i+k},\ldots,
                     q_{m,i-k},\ldots,q_{1,i-k}], 
\end{align*}
where $\tilde P(x)$ is the $mn' \times mn'$ matrix:
$$
\tilde P(x) := \left(\begin{array}{cccc} 
0 & 0 & 0 &x\\
1&0 &0 &0 \\
0&\ddots & \ddots&0 \\
0&0&1&0
\end{array} \right).
$$
We define $\tilde L(x) := \tilde Q_1(x) \cdots \tilde Q_N(x)
\tilde P(x)^{\bar k^\prime m}$, which is another Lax matrix.

Also define a map $\tilde \psi : \mM \to \C[x,y]$ given as a composition, 
$$
  \tilde Q \mapsto \tilde L(x) \mapsto \det(\tilde L(x) - y).
$$
Consequently, for $q \in \mM$
we have two affine plane curves $\tilde C_{\psi(q)}$ and 
$\tilde C_{\tilde \psi(q)}$in $\C^2$, given by 
the zeros of $\psi(q)$ and $\tilde \psi(q)$ respectively. 
The proof of Proposition \ref{prop:transposehull} is similar to that of Proposition \ref{prop:hull}.

\begin{prop} \label{prop:transposehull}
The Newton polygon $N(\tilde \psi(q))$ is the triangle with 
vertices $(0,m n')$, $(m \bar k',0)$ and $(m \bar k'+N,0)$,
where the lower hull (resp. upper hull)  
consists of one edge with vertices $(m \bar k',0)$ and $(0,m n')$
(resp. $(m \bar k'+N,0)$ to $(0,m n')$).
\end{prop}

\begin{lem}
The affine transformation
\begin{align}\label{eq:aff-trans}
  \begin{pmatrix}
  i \\ j
  \end{pmatrix}
  \mapsto 
  \begin{pmatrix} \bar k'(k + m) \\ -n' k \end{pmatrix}
  +
  \begin{pmatrix} -\bar k^\prime & 
                  (1- k^\prime \bar k^\prime)/n'  \\ 
                  n' & k^\prime \end{pmatrix}
  \begin{pmatrix} i \\ j \end{pmatrix}
\end{align}
sends integer points of $N(\psi(q))$ into integer points
of $N(\tilde \psi(q))$.
\end{lem}

\begin{proof}
 It is easy to see that it sends the vertices correctly. By the
definition of 
$\bar k'$ we know $(1-k' \bar k')/n'$ is an integer. Thus
this transformation sends integer points to integer points.
 So does the inverse, as the determinant of the matrix involved is $-1$.
\end{proof}

\begin{example}
Let $(n,m,k) = (6,3,4)$.
The two Newton polygons $N(\psi(q))$ and $N(\tilde \psi(q))$ are illustrated in
Figure \ref{fig:mnk1}. Here $i$ labels the horizontal axis and $j$
labels the vertical axis.
\begin{figure}[ht]
    \begin{center}
\vspace{-.1in}
\input{mnk1.pstex_t}
\vspace{-.1in}
    \end{center}
    \caption{}
    \label{fig:mnk1}
\end{figure}
The dots of the same color show integer points inside the Newton polygon
that get sent to each other by the transformation \eqref{eq:aff-trans}. The formula
for the transformation in this case is
 $$
 \left(\begin{array}{c}
i \\
j
\end{array} \right)
\mapsto
 \left(\begin{array}{c}
 14 \\
-12
\end{array} \right)
+
 \left(\begin{array}{cc}
-2 & -1 \\
3 & 2
\end{array} \right)
 \left(\begin{array}{c}
i \\
j
\end{array} \right).
 $$
\end{example}

\begin{prop}\label{prop:two-polys}
For $q \in \mM$, the polynomials $\psi(q)$ and $\tilde \psi(q)$ 
coincide up to the monomial transformation induced by \eqref{eq:aff-trans}.
The signs of the new monomials are derived from the rule given in Theorem \ref{thm:mit}: the sign of $x^a y^b$ is $(-1)^{(mn'-b-1)a+b}$.
\end{prop}

See \S~\ref{proof:two-polys} for the proof.

\subsection{Special points on the transposed curve}

We write $\tilde f(x,y)$ for the polynomial obtained from 
the fixed polynomial $f(x,y)$ in \S\ref{subsec:special-pts} via 
the affine transformation \eqref{eq:aff-trans}.
Let $C_{\tilde{f}}$ be the smooth compactification of the affine plane curve 
$\tilde{C}_{\tilde f}$ given by $\tilde f(x,y) = 0$.
As for $C_{f}$ (Lemma~\ref{lem:inf}), 
$C_{\tilde f}$ has a unique point $\tilde P$ lying over $\infty$. 
Due to this fact and Proposition~\ref{prop:two-polys}, the two curves $C_{f}$ and $C_{\tilde{f}}$ are isomorphic.
Let $\tau : C_{\tilde f} \to C_f$ be the birational isomorphism, which is given by
$$
  (x,y) \mapsto ( y^{n'} x^{k'}, x^\frac{1-k' \bar k'}{n'} y^{k'} ) 
$$
when $(x,y) \in C_{\tilde f} \cap (\C^\ast)^2$.
Besides $\tilde P$, on $C_{\tilde{f}}$ we have special points 
$\tilde O_r = ((-1)^{m n'} \sigma_r,0)$ for $r \in [N]$, and $\tilde A_i$ for $i \in [m]$,
where near $\tilde A_i$ there is a local coordinate $u$ such that
$$
  (x,y) \sim (u^{n'}, -(- \epsilon_i)^{\frac{1}{n'}} u^{\bar k'}).
$$
We see that $\tau (\tilde P) = P$, $\tau (\tilde A_i) = A_i$ and  
$\tau (\tilde O_r) = O_r$.

\subsection{Proof of Proposition~\ref{prop:two-polys}}
\label{proof:two-polys}

 The terms contributing to a particular coefficient of the spectral
curve are described by Theorem \ref{thm:mit}. We shall exhibit a bijection between
 the terms of the coefficient associated with the lattice point $(i,j)$
inside the Newton polygon $N(\psi(q))$ and the terms contributing to the 
coefficient associated with the image of that point in $N(\tilde \psi(q))$ 
 under the affine transformation \eqref{eq:aff-trans}.

An \emph{underway path} in the network $G$ is the mirror symmetric version of a highway path.  Thus the weights of an underway path are given by Figure \ref{fig:highway} with $0$ and $q_{ij}$ swapped.  The crucial observation is that families of highway paths on the
network associated with $\tilde Q$ are families of underway paths
on the original network,
 parsed in the opposite direction. This is because by the construction
of the transpose map between $Q$ and $\tilde Q$, the corresponding
toric networks are the same but are viewed
 from opposite sides of the torus (inside vs outside).

 Now, assume we have a closed family of highway paths contributing to
one of the coefficients of $\psi(q)$.
 Simply complement all edges that ended up on our family of closed
highway paths inside the set of all edges of the network. We claim
that the result can be parsed as the desired
 family of closed underway paths.

 Indeed, the original family can be viewed as a number of horizontal
``steps through'' picking up a weight $q_{ij}$ at some node,
connected by intervals of staircase paths that
 do not pick up any weight. We can interpret our procedure as
complementing used intervals of staircase paths, that is, making them not
used, and vice versa. As a result, locally around
 each weight $q_{ij}$ that was picked up the new path will look
like what is shown in Figure \ref{fig:mnk2}. Thus, it will be an underway path, and it will
pick up exactly such $q_{ij}$. In other words, the
 weight of the original highway family is the same as the weight of
this new underway family.

 \begin{figure}[ht]
    \begin{center}
\vspace{-.1in}
\input{mnk2.pstex_t}
\vspace{-.1in}
    \end{center}
    \caption{}
    \label{fig:mnk2}
\end{figure}

It is also easy to see that the new underway family is closed.
This completes the proof.

\begin{example}
 Let $(\a,\b,\c,\d)= (3,2,3,2)$ as in Example \ref{ex:3232}. Figure
\ref{fig:mnk3} shows an example of a family of paths contributing to
the purple term of the spectral curve
 as marked in Figure \ref{fig:mnk1}. The first step takes the
complement of edges of this family inside the set of all edges of this
network.
 \begin{figure}[ht]
    \begin{center}
\vspace{-.1in}
\input{mnk3.pstex_t}
\vspace{-.1in}
    \end{center}
    \caption{}
    \label{fig:mnk3}
\end{figure}
The second step does not change the network or the paths, it just
changes the point of view and reverses paths' directions.
\end{example}

\section{Relation to the dimer model}
\label{sec:cluster}

In this section, we give the explicit relation between 
the $R$-matrix dynamics on our toric network and
cluster transformations on the honeycomb dimer on a torus. 
See \cite{GonchaKenyon13} for background on the dimer model. 

\subsection{Cluster transformations on the honeycomb dimer}
\label{subsec:cluster-trans}
The calculation in this section is the dimer analogue of the highway network computation of the geometric $R$-matrix (cf. \cite[Theorem 6.2]{LP}), which was explained earlier in \S \ref{sec:dynamics_proof} (Figure \ref{fig:wire20}).

Fix a positive integer $L$, and consider a honeycomb bipartite graph on a cylinder as in
Figure~\ref{fig:honeycone-dimer},
where $\alpha_i, \beta_i, \gamma_i ~(i \in \Z/L \Z)$ are the weights of the faces in three cyclically consecutive rows.
We write $\alpha := (\alpha_i, \beta_i, \gamma_i)_{i \in \Z / L \Z}$.

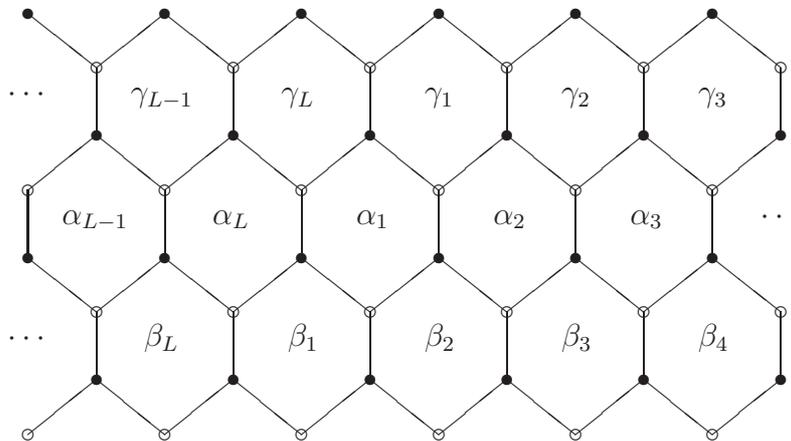
\begin{figure}[h]
\unitlength=0.9mm
\begin{picture}(100,80)(5,10)

\multiput(10,68)(20,0){6}{\line(0,-1){10}}
\multiput(0,50)(20,0){6}{\line(0,-1){10}}
\multiput(10,32)(20,0){6}{\line(0,-1){10}}

\multiput(0,76)(20,0){6}{\line(5,-4){10}}
\multiput(10,68)(20,0){5}{\line(5,4){10}}
\multiput(10,58)(20,0){5}{\line(5,-4){10}}
\multiput(0,50)(20,0){6}{\line(5,4){10}}
\multiput(0,40)(20,0){6}{\line(5,-4){10}}
\multiput(10,32)(20,0){5}{\line(5,4){10}}
\multiput(10,22)(20,0){5}{\line(5,-4){10}}
\multiput(0,14)(20,0){6}{\line(5,4){10}}

\multiput(0,76)(20,0){6}{\circle*{1.5}}
\multiput(10,68)(20,0){6}{\circle{1.5}}
\multiput(10,58)(20,0){6}{\circle*{1.5}}
\multiput(0,50)(20,0){6}{\circle{1.5}}
\multiput(0,40)(20,0){6}{\circle*{1.5}}
\multiput(10,32)(20,0){6}{\circle{1.5}}
\multiput(10,22)(20,0){6}{\circle*{1.5}}
\multiput(0,14)(20,0){6}{\circle{1.5}}

\put(-3,63){$\cdots$}
\put(15,63){$\gamma_{L-1}$}
\put(37,63){$\gamma_L$}
\put(58,63){$\gamma_1$}
\put(78,63){$\gamma_2$}
\put(98,63){$\gamma_3$}

\put(5,45){$\alpha_{L-1}$}
\put(27,45){$\alpha_L$}
\put(48,45){$\alpha_1$}
\put(68,45){$\alpha_2$}
\put(88,45){$\alpha_3$}
\put(107,45){$\cdots$}

\put(-3,27){$\cdots$}
\put(17,27){$\beta_L$}
\put(38,27){$\beta_1$}
\put(58,27){$\beta_2$}
\put(78,27){$\beta_3$}
\put(98,27){$\beta_4$}

\end{picture}
    \caption{Honeycomb dimer on a cylinder}
    \label{fig:honeycone-dimer}
\end{figure}

For the weights $\alpha$ of the honeycomb,
we define a transformation $R_\alpha$ of $\alpha$ in a following way:
first we split the $\alpha_1$-face into two, by inserting a digon 
of weight $-1$ as the top of Figure~\ref{fig:dimer-mutation}.  We thank R. Kenyon for explaining this operation to us.

Set the weights of the new two faces to be
$-c$ and $\frac{\alpha_1}{c}$, where $c$ is a nonzero parameter which 
will be determined.
Let $D$ be the quiver dual to the bipartite graph, drawn in blue in the figure. 
The weights of faces are to be regarded as {\em coefficient variables} 
associated to each vertex of $D$.
With $1,\ldots,L,a$ and $b$, 
we assign the vertices of $D$ in the middle row, as depicted.
Next, we apply the cluster mutations $\mu_1$, $\mu_2, \ldots, \mu_L$ to $(D,\alpha)$ 
in order, recursively defining $\omega_i ~(i=1,\ldots,L)$ by
$$
  \omega_1 := \frac{\alpha_1}{c}, \qquad \omega_i := \alpha_i(1+\omega_{i-1}).
$$  
The condition that the digon's weight is again $-1$ after the 
$L$ mutations gives an equation for $c$ as $-c (1+\omega_L) = -1$.
By solving it, $c$ is determined to be 
$$
  c = \frac{1-\prod_{s=1}^L \alpha_s}
{\sum_{t=0}^{L-1} \prod_{s=1}^t \alpha_{1-s}}, 
$$
and $\omega_i := \omega_i(\alpha_1,\ldots,\alpha_L)$ is obtained as 
\begin{align}\label{eq:omega}
  \omega_i(\alpha_1,\ldots,\alpha_L) 
  = \frac{\alpha_i \sum_{t=0}^{L-1} \prod_{s=1}^t \alpha_{i-s}}
                {1- \prod_{s=1}^L \alpha_s}.
\end{align}

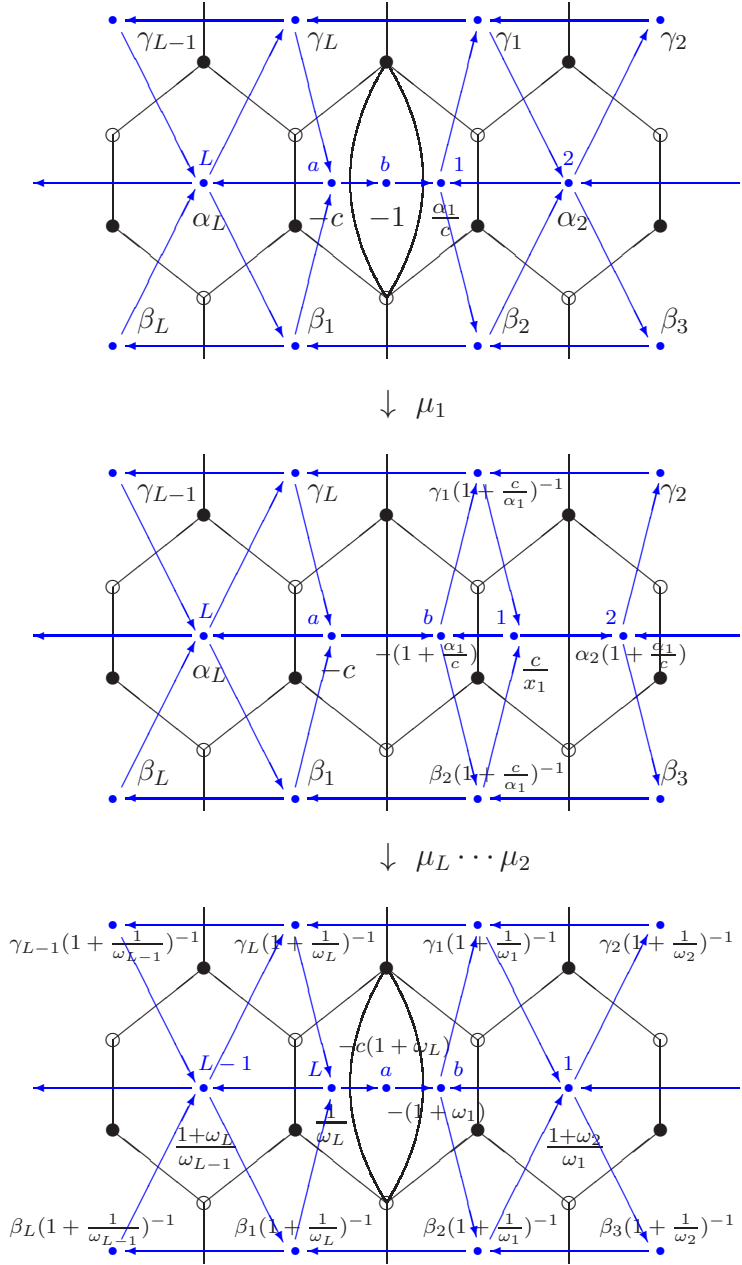
\begin{figure}
\unitlength=0.8mm
\begin{picture}(100,250)(5,-150)

\multiput(15,72)(30,0){3}{\line(0,-1){10}}
\multiput(0,50)(30,0){4}{\line(0,-1){15}}
\multiput(15,23)(30,0){3}{\line(0,-1){10}}

\multiput(0,50)(30,0){3}{\line(5,4){15}}
\multiput(30,50)(30,0){3}{\line(-5,4){15}}
\multiput(15,23)(30,0){3}{\line(5,4){15}}
\multiput(15,23)(30,0){3}{\line(-5,4){15}}

\multiput(15,62)(30,0){3}{\circle*{2}}
\multiput(0,50)(30,0){4}{\circle{2}}
\multiput(0,35)(30,0){4}{\circle*{2}}
\multiput(15,23)(30,0){3}{\circle{2}}

\qbezier(45,62)(33,42.5)(45,23)
\qbezier(45,62)(57,42.5)(45,23)

\put(13,35){\small $\alpha_L$}
\put(42,35){\small $-1$}
\put(32,35){\small $-c$}
\put(52,35){${\frac{\alpha_1}{c}}$}
\put(73,35){$\alpha_2$}

\put(4,65){\small $\gamma_{L-1}$}
\put(32,65){\small $\gamma_L$}
\put(63,65){\small $\gamma_1$}
\put(90,65){\small $\gamma_2$}

\put(4,18){\small $\beta_L$}
\put(32,18){\small $\beta_1$}
\put(64,18){\small $\beta_2$}
\put(90,18){\small $\beta_3$}

{\color{blue}

\multiput(0,69)(30,0){4}{\circle*{1.5}}
\multiput(15,42)(30,0){3}{\circle*{1.5}}
\multiput(36,42)(18,0){2}{\circle*{1.5}}
\multiput(0,15)(30,0){4}{\circle*{1.5}}

\put(32,44){\tiny $a$}
\put(44,44){\tiny $b$}
\put(56,44){\tiny $1$}
\put(74,45){\tiny $2$}
\put(14,45){\tiny $L$}

\multiput(28,69)(30,0){3}{\vector(-1,0){26}}
\multiput(28,15)(30,0){3}{\vector(-1,0){26}}
\put(37.5,42){\vector(1,0){6}}
\put(34.5,42){\vector(-1,0){18}}
\put(73.5,42){\vector(-1,0){18}}
\put(46.5,42){\vector(1,0){6}}
\put(13,42){\vector(-1,0){26}}
\put(103,42){\vector(-1,0){26}}
\put(1.5,67){\vector(1,-2){12}}
\multiput(16,43.5)(60,0){2}{\vector(1,2){12}}
\multiput(16,40.5)(60,0){2}{\vector(1,-2){12}}
\put(1.5,17){\vector(1,2){12}}
\put(30,67){\vector(1,-4){5.8}}
\put(54,44){\vector(1,4){5.8}}
\put(54,40.5){\vector(1,-4){5.8}}
\put(30,17){\vector(1,4){5.8}}
\put(62,67){\vector(1,-2){12}}
\put(62,17){\vector(1,2){12}}
}


\put(44,4){$\downarrow ~ \mu_1$}


\multiput(15,-3)(30,0){3}{\line(0,-1){10}}
\multiput(0,-25)(30,0){4}{\line(0,-1){15}}
\multiput(15,-52)(30,0){3}{\line(0,-1){10}}

\multiput(0,-25)(30,0){3}{\line(5,4){15}}
\multiput(30,-25)(30,0){3}{\line(-5,4){15}}
\multiput(15,-52)(30,0){3}{\line(5,4){15}}
\multiput(15,-52)(30,0){3}{\line(-5,4){15}}

\multiput(15,-13)(30,0){3}{\circle*{2}}
\multiput(0,-25)(30,0){4}{\circle{2}}
\multiput(0,-40)(30,0){4}{\circle*{2}}
\multiput(15,-52)(30,0){3}{\circle{2}}

\put(45,-13){\line(0,-1){39}}
\put(75,-13){\line(0,-1){39}}

\put(13,-40){\small $\alpha_L$}
\put(34,-40){\small $-c$}
\put(43,-37){\tiny{$-(1+\frac{\alpha_1}{c})$}}

\put(67,-40){${\frac{c}{x_1}}$}
\put(76,-37){\tiny{$\alpha_2(1+\frac{\alpha_1}{c})$}}

\put(4,-10){\small $\gamma_{L-1}$}
\put(32,-10){\small $\gamma_L$}
\put(52,-10){\tiny{$\gamma_1(1+\frac{c}{\alpha_1})^{-1}$}}
\put(90,-10){\small $\gamma_2$}

\put(4,-57){\small $\beta_L$}
\put(32,-57){\small $\beta_1$}
\put(52,-57){\tiny $\beta_2(1+\frac{c}{\alpha_1})^{-1}$}
\put(90,-57){\small $\beta_3$}

{\color{blue}

\multiput(0,-6)(30,0){4}{\circle*{1.5}}
\multiput(15,-33)(60,0){1}{\circle*{1.5}}
\multiput(36,-33)(18,0){2}{\circle*{1.5}}
\multiput(66,-33)(18,0){2}{\circle*{1.5}}
\multiput(0,-60)(30,0){4}{\circle*{1.5}}

\put(32,-31){\tiny $a$}
\put(51,-31){\tiny $b$}
\put(63,-31){\tiny $1$}
\put(81,-31){\tiny $2$}
\put(14,-30){\tiny $L$}

\multiput(28,-6)(30,0){3}{\vector(-1,0){26}}
\multiput(28,-60)(30,0){3}{\vector(-1,0){26}}
\put(37.5,-33){\vector(1,0){15}}
\put(34.5,-33){\vector(-1,0){18}}
\put(64,-33){\vector(-1,0){8}}
\put(67,-33){\vector(1,0){15}}
\put(46.5,42){\vector(1,0){6}}
\put(13,-33){\vector(-1,0){26}}
\put(103,-33){\vector(-1,0){17}}

\put(1.5,-8){\vector(1,-2){12}}
\multiput(16,-31.5)(60,0){1}{\vector(1,2){12}}
\multiput(16,-34.5)(60,0){1}{\vector(1,-2){12}}
\put(1.5,-58){\vector(1,2){12}}

\multiput(30,-8)(31,0){2}{\vector(1,-4){5.8}}
\multiput(54,-31)(30,0){2}{\vector(1,4){5.8}}
\multiput(54,-34.5)(30,0){2}{\vector(1,-4){5.8}}
\multiput(30,-58)(31,0){2}{\vector(1,4){5.8}}

}

\put(44,-71){$\downarrow ~ \mu_{L} \cdots \mu_2$}


\multiput(15,-78)(30,0){3}{\line(0,-1){10}}
\multiput(0,-100)(30,0){4}{\line(0,-1){15}}
\multiput(15,-127)(30,0){3}{\line(0,-1){10}}

\multiput(0,-100)(30,0){3}{\line(5,4){15}}
\multiput(30,-100)(30,0){3}{\line(-5,4){15}}
\multiput(15,-127)(30,0){3}{\line(5,4){15}}
\multiput(15,-127)(30,0){3}{\line(-5,4){15}}

\multiput(15,-88)(30,0){3}{\circle*{2}}
\multiput(0,-100)(30,0){4}{\circle{2}}
\multiput(0,-115)(30,0){4}{\circle*{2}}
\multiput(15,-127)(30,0){3}{\circle{2}}

\qbezier(45,-88)(33,-107.5)(45,-127)
\qbezier(45,-88)(57,-107.5)(45,-127)

\put(10,-119){$\frac{1+\omega_L}{\omega_{L-1}}$}
\put(33,-115){$\frac{1}{\omega_L}$}
\put(37,-102){\tiny $-c(1+\omega_L)$}
\put(45,-113){\tiny $-(1+\omega_1)$}
\put(71,-119){$\frac{1+\omega_2}{\omega_1}$}

\put(-17,-85){\tiny $\gamma_{L-1}(1+\frac{1}{\omega_{L-1}})^{-1}$}
\put(20,-85){\tiny $\gamma_L (1+\frac{1}{\omega_{L}})^{-1}$}
\put(51,-85){\tiny $\gamma_1 (1+\frac{1}{\omega_{1}})^{-1}$}
\put(80,-85){\tiny $\gamma_2 (1+\frac{1}{\omega_{2}})^{-1}$}

\put(-17,-132){\tiny $\beta_{L}(1+\frac{1}{\omega_{L-1}})^{-1}$}
\put(20,-132){\tiny $\beta_1 (1+\frac{1}{\omega_{L}})^{-1}$}
\put(51,-132){\tiny $\beta_2 (1+\frac{1}{\omega_{1}})^{-1}$}
\put(80,-132){\tiny $\beta_3 (1+\frac{1}{\omega_{2}})^{-1}$}

{\color{blue}

\multiput(0,-81)(30,0){4}{\circle*{1.5}}
\multiput(15,-108)(30,0){3}{\circle*{1.5}}
\multiput(36,-108)(18,0){2}{\circle*{1.5}}
\multiput(0,-135)(30,0){4}{\circle*{1.5}}

\put(32,-106){\tiny $L$}
\put(44,-106){\tiny $a$}
\put(56,-106){\tiny $b$}
\put(74,-105){\tiny $1$}
\put(14,-105){\tiny $L-1$}

\multiput(28,-81)(30,0){3}{\vector(-1,0){26}}
\multiput(28,-135)(30,0){3}{\vector(-1,0){26}}
\put(37.5,-108){\vector(1,0){6}}
\put(34.5,-108){\vector(-1,0){18}}
\put(73.5,-108){\vector(-1,0){18}}
\put(46.5,-108){\vector(1,0){6}}
\put(13,-108){\vector(-1,0){26}}
\put(103,-108){\vector(-1,0){26}}
\multiput(1.5,-83)(60,0){2}{\vector(1,-2){12}}
\multiput(16,-106.5)(60,0){2}{\vector(1,2){12}}
\multiput(16,-109.5)(60,0){2}{\vector(1,-2){12}}
\put(1.5,-133){\vector(1,2){12}}
\put(30,-83){\vector(1,-4){5.8}}
\put(54,-106){\vector(1,4){5.8}}
\put(54,-109.5){\vector(1,-4){5.8}}
\put(30,-133){\vector(1,4){5.8}}
\put(62,-133){\vector(1,2){12}}
}
\end{picture}
    \caption{Honeycomb dimer model on a cylinder.}
    \label{fig:dimer-mutation}
\end{figure}

\begin{definition}\label{prop:Rcluster}
Let $R_{\alpha}$ be the transformation of $\alpha$ given by
  \begin{align}\label{eq:Rtrans}
    (\alpha_i,\beta_i,\gamma_i) \mapsto 
    \left( \omega_{i-1}^{-1} (1+\omega_i),~
           \beta_i (1+\omega_{i-1}^{-1})^{-1},~
           \gamma_i (1+\omega_{i}^{-1})^{-1} \right),
  \end{align}
which is induced by 
the sequence  of mutations $\mu_L \mu_{L-1} \cdots \mu_1$,
\end{definition}

Due to the expression of $\omega_i$,
we see that  
$R_\alpha$ does not depend on which $\alpha_j$-face we split 
at the first stage. 

\begin{remark}
Note that our derivation of $R_\alpha$ is not entirely a cluster algebra computation since we began with the ``digon insertion" operation, which does not have a clear cluster algebra interpretation.  In an upcoming work \cite{ILP}, we plan to further clarify the cluster nature of the geometric $R$-matrix.
\end{remark}

\subsection{Relation with the $(n,m,k)$-network}

We consider the honeycomb bipartite graph on a torus as in Figure~\ref{fig:Toda-dimer}.
We assign each face (resp. edge) with a weight $x_{ji}$ 
(resp. $q_{ji}$ or $1$),
satisfying the periodicity conditions $x_{j,i+n} = x_{j,i}$ and $x_{m+j,i} = x_{j,i-k}$ 
(resp. $q_{j,i+n} = q_{j,i}$ and $q_{m+j,i} = q_{j,i-k}$).
In the figure we omit weights that are equal to $1$. 
The $x_{ji}$ and $q_{ji}$ are related by
\begin{align}\label{eq:x-q}
  x_{j,i} = \frac{q_{j,i+1}}{q_{j+1,i}} ~(j=1,\ldots,m-1),
  \qquad 
  x_{m,i} = \frac{q_{m,i+1}}{q_{1,i-k}},
\end{align}
hence the product of all the $x_{ji}$ is equal to $1$.

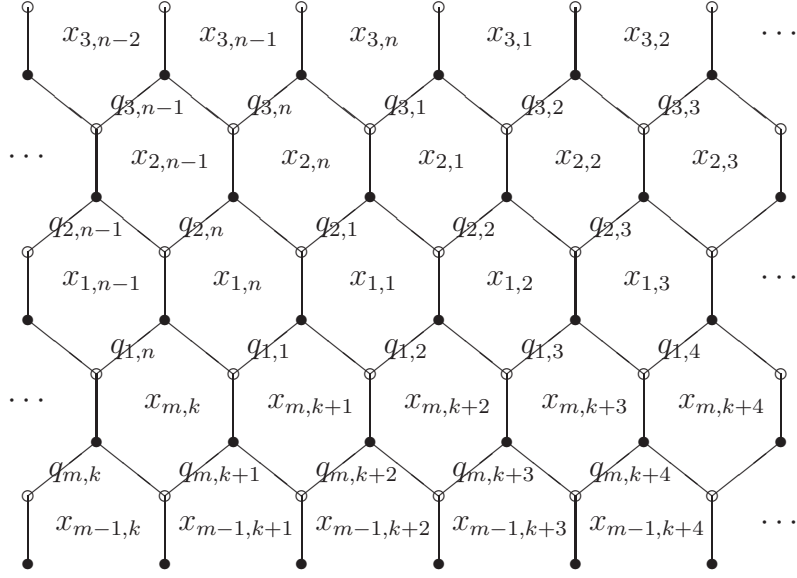
\begin{figure}
\unitlength=0.9mm
\begin{picture}(100,100)(5,0)

\multiput(0,86)(20,0){6}{\line(0,-1){10}}
\multiput(10,68)(20,0){6}{\line(0,-1){10}}
\multiput(0,50)(20,0){6}{\line(0,-1){10}}
\multiput(10,32)(20,0){6}{\line(0,-1){10}}
\multiput(0,14)(20,0){6}{\line(0,-1){10}}

\multiput(0,76)(20,0){6}{\line(5,-4){10}}
\multiput(10,68)(20,0){5}{\line(5,4){10}}
\multiput(10,58)(20,0){5}{\line(5,-4){10}}
\multiput(0,50)(20,0){6}{\line(5,4){10}}
\multiput(0,40)(20,0){6}{\line(5,-4){10}}
\multiput(10,32)(20,0){5}{\line(5,4){10}}
\multiput(10,22)(20,0){5}{\line(5,-4){10}}
\multiput(0,14)(20,0){6}{\line(5,4){10}}

\multiput(0,86)(20,0){6}{\circle{1.5}}
\multiput(0,76)(20,0){6}{\circle*{1.5}}
\multiput(10,68)(20,0){6}{\circle{1.5}}
\multiput(10,58)(20,0){6}{\circle*{1.5}}
\multiput(0,50)(20,0){6}{\circle{1.5}}
\multiput(0,40)(20,0){6}{\circle*{1.5}}
\multiput(10,32)(20,0){6}{\circle{1.5}}
\multiput(10,22)(20,0){6}{\circle*{1.5}}
\multiput(0,14)(20,0){6}{\circle{1.5}}
\multiput(0,4)(20,0){6}{\circle*{1.5}}

\put(5,81){$x_{3,n-2}$}
\put(25,81){$x_{3,n-1}$}
\put(47,81){$x_{3,n}$}
\put(67,81){$x_{3,1}$}
\put(87,81){$x_{3,2}$}
\put(107,81){$\cdots$}

\put(-3,63){$\cdots$}
\put(15,63){$x_{2,n-1}$}
\put(37,63){$x_{2,n}$}
\put(57,63){$x_{2,1}$}
\put(77,63){$x_{2,2}$}
\put(97,63){$x_{2,3}$}

\put(5,45){$x_{1,n-1}$}
\put(27,45){$x_{1,n}$}
\put(47,45){$x_{1,1}$}
\put(67,45){$x_{1,2}$}
\put(87,45){$x_{1,3}$}
\put(107,45){$\cdots$}

\put(-3,27){$\cdots$}
\put(17,27){$x_{m,k}$}
\put(35,27){$x_{m,k+1}$}
\put(55,27){$x_{m,k+2}$}
\put(75,27){$x_{m,k+3}$}
\put(95,27){$x_{m,k+4}$}

\put(4,9){$x_{m-1,k}$}
\put(22,9){$x_{m-1,k+1}$}
\put(42,9){$x_{m-1,k+2}$}
\put(62,9){$x_{m-1,k+3}$}
\put(82,9){$x_{m-1,k+4}$}
\put(107,9){$\cdots$}

\put(12,71){$q_{3,n-1}$}
\put(32,71){$q_{3,n}$}
\put(52,71){$q_{3,1}$}
\put(72,71){$q_{3,2}$}
\put(92,71){$q_{3,3}$}

\put(3,53){$q_{2,n-1}$}
\put(22,53){$q_{2,n}$}
\put(42,53){$q_{2,1}$}
\put(62,53){$q_{2,2}$}
\put(82,53){$q_{2,3}$}

\put(12,35){$q_{1,n}$}
\put(32,35){$q_{1,1}$}
\put(52,35){$q_{1,2}$}
\put(72,35){$q_{1,3}$}
\put(92,35){$q_{1,4}$}

\put(3,17){$q_{m,k}$}
\put(22,17){$q_{m,k+1}$}
\put(42,17){$q_{m,k+2}$}
\put(62,17){$q_{m,k+3}$}
\put(82,17){$q_{m,k+4}$}

\end{picture}
    \caption{$(n,m,k)$-dimer}
    \label{fig:Toda-dimer}
\end{figure}

Let $\mathcal{X} \simeq \C^{mn}$ be a space of the weights of faces whose 
coordinates are given by 
$\underline{x} := (x_{ji})_{j \in \Z/m \Z,~ i \in \Z/n\Z}$.
Let $\rho$ be an embedding map from $\C(\mathcal{X})$ to $\C(\mathcal{M})$
given by \eqref{eq:x-q}.

We define three types of actions $\bar R_j$, $\tilde R_i$ and 
$\hat R_i$ on $\mathcal{X}$, corresponding to the three directions that the honeycomb bipartite graph can be arranged into rows.
For $j =1,\ldots,m$, define 
$\bar x(j) := (x_{j,i}, x_{j-1,i}, x_{j+1,i})_{i \in \Z/ n \Z}$. 
Let $\bar R_j$ be the action on $\mathcal{X}$ induced by $R_{\bar x(j)}$ 
\eqref{eq:Rtrans} 
with $L=n$.  More precisely, $R_j(\underline{x}) = \underline{x}^\prime$ is given by
$$
  x_{li}^\prime = 
  \begin{cases}
    \omega_{i-1}^{-1}(1+\omega_i) & l = j,
    \\
    x_{j-1,i} (1+\omega_{i-1}^{-1})^{-1} & l=j-1,
    \\
    x_{j+1,i} (1+\omega_{i}^{-1})^{-1} & l=j+1,
    \\
    x_{li} & \text{otherwise},
  \end{cases}
$$
with $\omega_i := \omega_i(x_{j,1},\ldots,x_{j,n})$ \eqref{eq:omega}.
In a similar way,
for $i=1,\ldots,N$, define
$\tilde x(i) := (x_{m-j,N-i}, x_{m-j,N-i-1}, x_{m-j,N-i+1})_{j \in \Z/ m n'\Z}$ 
and let $\tilde R_i$ be the action on $\mathcal{X}$ given by $R_{\tilde x(i)}$
\eqref{eq:Rtrans} 
with $L=mn'$.
For $i =1,\ldots,M$, define
$\hat x(i) := (x_{j,i+1-j}, x_{j,i+2-j}, x_{j,i-j})_{j \in \Z/ \frac{mn}{M} \Z}$, and let $\hat R_i$ be the action on $\mathcal{X}$ 
given by $R_{\hat x(i)}$ \eqref{eq:Rtrans} 
with $L=\frac{mn}{M}$.

Recall that we also have three types of actions on $\mathcal{M}$ given by 
the elements $s_j ~(j=0,1,\ldots,m-1)$ of 
the extended symmetric group $W$,
the elements $\tilde s_i ~(i=0,1,\ldots,N-1)$ of 
the extended symmetric group $\tilde W$,
and the snake path action $T_s ~(s=1,\ldots,M)$.

\begin{prop}\
\begin{enumerate}
\item[(i)] 
The actions $s_j$ and $\tilde s_i$ are compatible with 
the actions $\bar R_j$ and $\tilde R_j$ respectively:
for $\underline{x} \in \mathcal{X}$, we have
\begin{align}
  \label{R-s-1}
  &\rho \circ \bar R_j (\underline{x}) 
  = s_j^\ast \circ \rho (\underline{x}),
  \quad j \in \Z / m \Z,
  \\
  \label{R-s-2}
  &\rho \circ \tilde R_i (\underline{x})
  = \tilde s_i^\ast \circ \rho (\underline{x}),
  \quad i \in \Z / N \Z.
\end{align}
\item[(ii)] For $\underline{x} \in \mathcal{X}$,
the action $\hat R_i$ satisfies
\begin{align}
  \label{R-T}
  &\rho \circ \hat R_i (\underline{x})= \rho (\underline{x}),
  \quad i=1,\ldots,M.
\end{align}
\item[(iii)] 
For $\underline{x} \in \mathcal{X}$, the snake path action $T_i^\ast$ satisfies
$$T_i^\ast \circ \rho(\underline{x}) = \rho(\underline{x}), \quad i=1,\ldots,M.$$
\end{enumerate}
\end{prop}

\begin{proof}
(i) To show \eqref{R-s-1}, it is enough to prove the $j=1$ case.
We write $E_i$ for the energy $E(P^{-i} Q_1 P^i, P^{-i} Q_2 P^i)$.
The operator $s_1$ acts on $\mathcal{M}$ as 
$s_1(q) = q^\prime$: 
\begin{align}
  q_{1,i}^\prime 
  = 
  \displaystyle{q_{2,i} \,\frac{E_{i}}{E_{i-1}},}
  \qquad
  q_{2,i}^\prime 
  = 
  \displaystyle{q_{1,i} \, \frac{E_{i-1}}{E_{j,i}}},
\end{align}
and the other $q_{ji}$ do not change.
By definition, $R_1(\underline{x}) = \underline{x}^\prime$ is obtained as
$$
   x_{ji}^\prime
   = 
   \begin{cases} 
   \omega_{i-1}^{-1} (1+\omega_i) & j=1,
   \\
   x_{m,i+k} (1+\omega_{i-1}^{-1})^{-1} & j=m,
   \\
   x_{2,i} (1+\omega_{i}^{-1})^{-1} & j=2,
   \\
   x_{ji} & \text{otherwise}.
   \end{cases} 
$$
where $\omega_i := \omega_i(x_{1,1},\ldots,x_{1,n})$.
On the other hand, by direct computation we obtain
$$
  \rho(\omega_i) = \frac{q_{1,i+1} E_{i}}
  {\prod_{s=1}^n q_{2,s} - \prod_{s=1}^n q_{1,s}},
  \qquad
  \rho(1+\omega_i) = \frac{q_{2,i+1} E_{i+1}}
  {\prod_{s=1}^n q_{2,s} - \prod_{s=1}^n q_{1,s}}.
$$ 
Thus we get
\begin{align*}
  &\rho \left(\omega_{i-1}^{-1} (1+\omega_i)\right)
  = \frac{q_{2,i+1} E_{i+1}}{q_{1,i} E_{i-1}} 
  = s_1^\ast \left(\frac{q_{1,i+1}}{q_{2,i}}\right) 
  = s_1^\ast \circ \rho(x_{1,i}), 
  \\
  &\rho \left(x_{m,i+k} (1+\omega_{i-1}^{-1})^{-1}\right)
  = \frac{q_{m,i+k+1} E_{i-1}}{q_{2,i} E_{i}} 
  = s_1^\ast \left(\frac{q_{m,i+k+1}}{q_{1,i}}\right)
  = s_1^\ast \circ \rho(x_{m,i+k}), 
  \\
  &\rho \left(x_{2,i} (1+\omega_{i}^{-1})^{-1}\right)
  = \frac{q_{1,i+1} E_{i}}{q_{3,i} E_{i+1}} 
  = s_1^\ast \left(\frac{q_{2,i+1}}{q_{3,i}}\right)
  = s_1^\ast \circ \rho(x_{2,i}), 
\end{align*}
and \eqref{R-s-1} follows.

We can prove \eqref{R-s-2} in a similar manner,
by replacing $Q_i$ with $\tilde Q_i$, $P$ with $\tilde P$, and so on.

\noindent
(ii) Again, it is enough to prove the case of $i=1$.
For simplicity, we write $L$ for $\frac{mn}{M}$, and set 
$x_j := x_{j,2-j}$, $q_j := q_{j,3-j}$ for $j \in \Z /L \Z$. 
From \eqref{eq:x-q}, it follows that 
\begin{align}\label{eq:x-s-rho}
  \rho(x_j) = \frac{q_j}{q_{j+1}}, \qquad
  \rho\left(\prod_{j=1}^L x_j\right) = 1.
\end{align}
Define $\omega_j = \omega_j(x_1,\ldots,x_L)$ for $j=1,\ldots,L$.
Using the definition of $\omega_j$ and  \eqref{eq:x-s-rho},
we obtain
\begin{align}
  \rho\left(\frac{1 + \omega_j}{\omega_{j-1}}\right) 
  = \frac{q_j}{q_{j+1}},
  \qquad
  \rho\left(\frac{\omega_j}{1 + \omega_j}\right) = 1,
\end{align}
by the following calculation:
\begin{align*}
  \frac{1 + \omega_j}{\omega_{j-1}}
  &= 
  \frac{1 - \prod_{s=1}^L x_s + x_i \sum_{t=0}^{L-1} \prod_{s=1}^t x_{i-s}}
       {x_{i-1} \sum_{t=0}^{L-1} \prod_{s=1}^t x_{i-1-s}} 
  \\               
  &\stackrel{\rho}{\longmapsto}
  \frac{1 - 1 + \sum_{t=0}^{n-1} \frac{q_{i-t}}{q_{i+1}}}
       {\sum_{t=1}^{n} \frac{q_{i-t}}{q_{i}}}
  = \frac{q_i}{q_{i+1}},
  \\
  \frac{\omega_{j}}{1 + \omega_j}
  &= 
  \frac{x_{i} \sum_{t=0}^{L-1} \prod_{s=1}^t x_{i-s}} 
       {1 - \prod_{s=1}^L x_s + x_i \sum_{t=0}^{L-1} \prod_{s=1}^t x_{i-s}}
  \stackrel{\rho}{\longmapsto} 1.
\end{align*}
Thus we have 
$\rho \circ R_{\hat x(1)}(\hat x(1)) 
= \rho(\hat x(1))$, 
and \eqref{R-T} follows. 

\noindent
(iii) 
The snake path action $T_s$ 
changes $q_{j,i}$ and $q_{j^\prime,i^\prime}$ in the same way 
if $i+j \equiv i^\prime + j^\prime \mod M$.
This condition is satisfied by $q_{j,i+1}$ and $q_{j+1,i}$,
then the change is cancelled in $\rho(x_{j,i})$. 
Thus we see that $T_s \circ \rho(x_{j,i}) = \rho(x_{j,i})$.
\end{proof}


\end{document}

%% file: toda3.pstex_t
\begin{picture}(0,0)%
\includegraphics{toda3.pstex}%
\end{picture}%
\setlength{\unitlength}{1657sp}%
\begingroup\makeatletter\ifx\SetFigFont\undefined%
\gdef\SetFigFont#1#2#3#4#5{%
  \reset@font\fontsize{#1}{#2pt}%
  \fontfamily{#3}\fontseries{#4}\fontshape{#5}%
  \selectfont}%
\fi\endgroup%
\begin{picture}(9916,5647)(10703,-5476)
\end{picture}%

%% file: toda1.pstex_t
\begin{picture}(0,0)%
\includegraphics{toda1.pstex}%
\end{picture}%
\setlength{\unitlength}{2368sp}%
\begingroup\makeatletter\ifx\SetFigFont\undefined%
\gdef\SetFigFont#1#2#3#4#5{%
  \reset@font\fontsize{#1}{#2pt}%
  \fontfamily{#3}\fontseries{#4}\fontshape{#5}%
  \selectfont}%
\fi\endgroup%
\begin{picture}(6555,4581)(5086,-6694)
\put(5551,-3136){\makebox(0,0)[lb]{\smash{{\SetFigFont{12}{14.4}{\rmdefault}{\mddefault}{\updefault}{\color[rgb]{0,0,0}$1$}%
}}}}
\put(11626,-3136){\makebox(0,0)[lb]{\smash{{\SetFigFont{12}{14.4}{\rmdefault}{\mddefault}{\updefault}{\color[rgb]{0,0,0}$1$}%
}}}}
\put(5551,-4336){\makebox(0,0)[lb]{\smash{{\SetFigFont{12}{14.4}{\rmdefault}{\mddefault}{\updefault}{\color[rgb]{0,0,0}$k$}%
}}}}
\put(5551,-6136){\makebox(0,0)[lb]{\smash{{\SetFigFont{12}{14.4}{\rmdefault}{\mddefault}{\updefault}{\color[rgb]{0,0,0}$n$}%
}}}}
\put(11626,-4336){\makebox(0,0)[lb]{\smash{{\SetFigFont{12}{14.4}{\rmdefault}{\mddefault}{\updefault}{\color[rgb]{0,0,0}$n-k$}%
}}}}
\put(11626,-4936){\makebox(0,0)[lb]{\smash{{\SetFigFont{12}{14.4}{\rmdefault}{\mddefault}{\updefault}{\color[rgb]{0,0,0}$n-k+1$}%
}}}}
\put(11626,-6136){\makebox(0,0)[lb]{\smash{{\SetFigFont{12}{14.4}{\rmdefault}{\mddefault}{\updefault}{\color[rgb]{0,0,0}$n$}%
}}}}
\put(5101,-4936){\makebox(0,0)[lb]{\smash{{\SetFigFont{12}{14.4}{\rmdefault}{\mddefault}{\updefault}{\color[rgb]{0,0,0}$1+k$}%
}}}}
\put(9751,-2236){\makebox(0,0)[lb]{\smash{{\SetFigFont{12}{14.4}{\rmdefault}{\mddefault}{\updefault}{\color[rgb]{0,0,0}$x$-line}%
}}}}
\end{picture}%

%% file: toda2.pstex_t
\begin{picture}(0,0)%
\includegraphics{toda2.pstex}%
\end{picture}%
\setlength{\unitlength}{1973sp}%
\begingroup\makeatletter\ifx\SetFigFont\undefined%
\gdef\SetFigFont#1#2#3#4#5{%
  \reset@font\fontsize{#1}{#2pt}%
  \fontfamily{#3}\fontseries{#4}\fontshape{#5}%
  \selectfont}%
\fi\endgroup%
\begin{picture}(8277,8724)(5536,-12673)
\put(5551,-9136){\makebox(0,0)[lb]{\smash{{\SetFigFont{10}{12.0}{\rmdefault}{\mddefault}{\updefault}{\color[rgb]{0,0,0}$1$}%
}}}}
\put(5551,-12136){\makebox(0,0)[lb]{\smash{{\SetFigFont{10}{12.0}{\rmdefault}{\mddefault}{\updefault}{\color[rgb]{0,0,0}$n$}%
}}}}
\put(5551,-9736){\makebox(0,0)[lb]{\smash{{\SetFigFont{10}{12.0}{\rmdefault}{\mddefault}{\updefault}{\color[rgb]{0,0,0}$i$}%
}}}}
\put(8026,-7711){\makebox(0,0)[lb]{\smash{{\SetFigFont{10}{12.0}{\rmdefault}{\mddefault}{\updefault}{\color[rgb]{0,0,0}$i$}%
}}}}
\put(9826,-5911){\makebox(0,0)[lb]{\smash{{\SetFigFont{10}{12.0}{\rmdefault}{\mddefault}{\updefault}{\color[rgb]{0,0,0}$i$}%
}}}}
\put(11626,-4111){\makebox(0,0)[lb]{\smash{{\SetFigFont{10}{12.0}{\rmdefault}{\mddefault}{\updefault}{\color[rgb]{0,0,0}$i$}%
}}}}
\put(11626,-7711){\makebox(0,0)[lb]{\smash{{\SetFigFont{10}{12.0}{\rmdefault}{\mddefault}{\updefault}{\color[rgb]{0,0,0}$i$}%
}}}}
\put(13426,-5911){\makebox(0,0)[lb]{\smash{{\SetFigFont{10}{12.0}{\rmdefault}{\mddefault}{\updefault}{\color[rgb]{0,0,0}$i$}%
}}}}
\put(13426,-9511){\makebox(0,0)[lb]{\smash{{\SetFigFont{10}{12.0}{\rmdefault}{\mddefault}{\updefault}{\color[rgb]{0,0,0}$i$}%
}}}}
\put(9826,-9511){\makebox(0,0)[lb]{\smash{{\SetFigFont{10}{12.0}{\rmdefault}{\mddefault}{\updefault}{\color[rgb]{0,0,0}$i$}%
}}}}
\put(13651,-10636){\makebox(0,0)[lb]{\smash{{\SetFigFont{10}{12.0}{\rmdefault}{\mddefault}{\updefault}{\color[rgb]{0,0,0}$x$-line}%
}}}}
\put(13651,-7036){\makebox(0,0)[lb]{\smash{{\SetFigFont{10}{12.0}{\rmdefault}{\mddefault}{\updefault}{\color[rgb]{0,0,0}$x$-line}%
}}}}
\end{picture}%

%% file: wire11.pstex_t
\begin{picture}(0,0)%
\includegraphics{wire11.pstex}%
\end{picture}%
\setlength{\unitlength}{1579sp}%
\begingroup\makeatletter\ifx\SetFigFont\undefined%
\gdef\SetFigFont#1#2#3#4#5{%
  \reset@font\fontsize{#1}{#2pt}%
  \fontfamily{#3}\fontseries{#4}\fontshape{#5}%
  \selectfont}%
\fi\endgroup%
\begin{picture}(14744,1244)(4479,-6083)
\put(5026,-5161){\makebox(0,0)[lb]{\smash{{\SetFigFont{10}{12.0}{\rmdefault}{\mddefault}{\updefault}{\color[rgb]{0,0,0}$x$}%
}}}}
\put(6226,-5161){\makebox(0,0)[lb]{\smash{{\SetFigFont{10}{12.0}{\rmdefault}{\mddefault}{\updefault}{\color[rgb]{0,0,0}$y$}%
}}}}
\put(13876,-5161){\makebox(0,0)[lb]{\smash{{\SetFigFont{10}{12.0}{\rmdefault}{\mddefault}{\updefault}{\color[rgb]{0,0,0}$0$}%
}}}}
\put(9601,-5161){\makebox(0,0)[lb]{\smash{{\SetFigFont{10}{12.0}{\rmdefault}{\mddefault}{\updefault}{\color[rgb]{0,0,0}$x+y$}%
}}}}
\end{picture}%

%% file: wire8.pstex_t
\begin{picture}(0,0)%
\includegraphics{wire8.pstex}%
\end{picture}%
\setlength{\unitlength}{2368sp}%
\begingroup\makeatletter\ifx\SetFigFont\undefined%
\gdef\SetFigFont#1#2#3#4#5{%
  \reset@font\fontsize{#1}{#2pt}%
  \fontfamily{#3}\fontseries{#4}\fontshape{#5}%
  \selectfont}%
\fi\endgroup%
\begin{picture}(8444,3194)(2679,-4883)
\put(3901,-2611){\makebox(0,0)[lb]{\smash{{\SetFigFont{11}{13.2}{\rmdefault}{\mddefault}{\updefault}{\color[rgb]{0,0,0}$x$}%
}}}}
\put(4651,-3886){\makebox(0,0)[lb]{\smash{{\SetFigFont{11}{13.2}{\rmdefault}{\mddefault}{\updefault}{\color[rgb]{0,0,0}$y$}%
}}}}
\put(4951,-2611){\makebox(0,0)[lb]{\smash{{\SetFigFont{11}{13.2}{\rmdefault}{\mddefault}{\updefault}{\color[rgb]{0,0,0}$z$}%
}}}}
\put(8251,-3586){\makebox(0,0)[lb]{\smash{{\SetFigFont{11}{13.2}{\rmdefault}{\mddefault}{\updefault}{\color[rgb]{0,0,0}$\frac{yz}{x+z}$}%
}}}}
\put(9976,-3586){\makebox(0,0)[lb]{\smash{{\SetFigFont{11}{13.2}{\rmdefault}{\mddefault}{\updefault}{\color[rgb]{0,0,0}$\frac{xy}{x+z}$}%
}}}}
\put(9451,-2911){\makebox(0,0)[lb]{\smash{{\SetFigFont{11}{13.2}{\rmdefault}{\mddefault}{\updefault}{\color[rgb]{0,0,0}$x+z$}%
}}}}
\end{picture}%

%% file: wire20a.pstex_t
\begin{picture}(0,0)%
\includegraphics{wire20a.pstex}%
\end{picture}%
\setlength{\unitlength}{1184sp}%
\begingroup\makeatletter\ifx\SetFigFont\undefined%
\gdef\SetFigFont#1#2#3#4#5{%
  \reset@font\fontsize{#1}{#2pt}%
  \fontfamily{#3}\fontseries{#4}\fontshape{#5}%
  \selectfont}%
\fi\endgroup%
\begin{picture}(25844,6034)(579,-5623)
\put(2401,-4186){\makebox(0,0)[lb]{\smash{{\SetFigFont{8}{9.6}{\rmdefault}{\mddefault}{\updefault}{\color[rgb]{0,0,0}.}%
}}}}
\put(2401,-4411){\makebox(0,0)[lb]{\smash{{\SetFigFont{8}{9.6}{\rmdefault}{\mddefault}{\updefault}{\color[rgb]{0,0,0}.}%
}}}}
\put(2401,-4636){\makebox(0,0)[lb]{\smash{{\SetFigFont{8}{9.6}{\rmdefault}{\mddefault}{\updefault}{\color[rgb]{0,0,0}.}%
}}}}
\put(17701,-4786){\makebox(0,0)[lb]{\smash{{\SetFigFont{8}{9.6}{\rmdefault}{\mddefault}{\updefault}{\color[rgb]{0,0,0}.}%
}}}}
\put(17701,-5011){\makebox(0,0)[lb]{\smash{{\SetFigFont{8}{9.6}{\rmdefault}{\mddefault}{\updefault}{\color[rgb]{0,0,0}.}%
}}}}
\put(17701,-5236){\makebox(0,0)[lb]{\smash{{\SetFigFont{8}{9.6}{\rmdefault}{\mddefault}{\updefault}{\color[rgb]{0,0,0}.}%
}}}}
\put(17476,-136){\makebox(0,0)[lb]{\smash{{\SetFigFont{8}{9.6}{\rmdefault}{\mddefault}{\updefault}{\color[rgb]{0,0,0}$-p$}%
}}}}
\put(20701,-2461){\makebox(0,0)[lb]{\smash{{\SetFigFont{8}{9.6}{\rmdefault}{\mddefault}{\updefault}{\color[rgb]{0,0,0}.}%
}}}}
\put(20926,-2461){\makebox(0,0)[lb]{\smash{{\SetFigFont{8}{9.6}{\rmdefault}{\mddefault}{\updefault}{\color[rgb]{0,0,0}.}%
}}}}
\put(21151,-2461){\makebox(0,0)[lb]{\smash{{\SetFigFont{8}{9.6}{\rmdefault}{\mddefault}{\updefault}{\color[rgb]{0,0,0}.}%
}}}}
\put(24601,-4936){\makebox(0,0)[lb]{\smash{{\SetFigFont{8}{9.6}{\rmdefault}{\mddefault}{\updefault}{\color[rgb]{0,0,0}.}%
}}}}
\put(24601,-5161){\makebox(0,0)[lb]{\smash{{\SetFigFont{8}{9.6}{\rmdefault}{\mddefault}{\updefault}{\color[rgb]{0,0,0}.}%
}}}}
\put(24601,-5386){\makebox(0,0)[lb]{\smash{{\SetFigFont{8}{9.6}{\rmdefault}{\mddefault}{\updefault}{\color[rgb]{0,0,0}.}%
}}}}
\put(24526, 14){\makebox(0,0)[lb]{\smash{{\SetFigFont{8}{9.6}{\rmdefault}{\mddefault}{\updefault}{\color[rgb]{0,0,0}$p$}%
}}}}
\put(7501,-4936){\makebox(0,0)[lb]{\smash{{\SetFigFont{8}{9.6}{\rmdefault}{\mddefault}{\updefault}{\color[rgb]{0,0,0}.}%
}}}}
\put(7501,-5161){\makebox(0,0)[lb]{\smash{{\SetFigFont{8}{9.6}{\rmdefault}{\mddefault}{\updefault}{\color[rgb]{0,0,0}.}%
}}}}
\put(7501,-5386){\makebox(0,0)[lb]{\smash{{\SetFigFont{8}{9.6}{\rmdefault}{\mddefault}{\updefault}{\color[rgb]{0,0,0}.}%
}}}}
\put(7426,-1561){\makebox(0,0)[lb]{\smash{{\SetFigFont{8}{9.6}{\rmdefault}{\mddefault}{\updefault}{\color[rgb]{0,0,0}$p$}%
}}}}
\put(7276, 14){\makebox(0,0)[lb]{\smash{{\SetFigFont{8}{9.6}{\rmdefault}{\mddefault}{\updefault}{\color[rgb]{0,0,0}$-p$}%
}}}}
\put(12601,-4786){\makebox(0,0)[lb]{\smash{{\SetFigFont{8}{9.6}{\rmdefault}{\mddefault}{\updefault}{\color[rgb]{0,0,0}.}%
}}}}
\put(12601,-5011){\makebox(0,0)[lb]{\smash{{\SetFigFont{8}{9.6}{\rmdefault}{\mddefault}{\updefault}{\color[rgb]{0,0,0}.}%
}}}}
\put(12601,-5236){\makebox(0,0)[lb]{\smash{{\SetFigFont{8}{9.6}{\rmdefault}{\mddefault}{\updefault}{\color[rgb]{0,0,0}.}%
}}}}
\put(12376,-136){\makebox(0,0)[lb]{\smash{{\SetFigFont{8}{9.6}{\rmdefault}{\mddefault}{\updefault}{\color[rgb]{0,0,0}$-p$}%
}}}}
\put(12526,-2311){\makebox(0,0)[lb]{\smash{{\SetFigFont{8}{9.6}{\rmdefault}{\mddefault}{\updefault}{\color[rgb]{0,0,0}$p'$}%
}}}}
\put(1051,-811){\makebox(0,0)[lb]{\smash{{\SetFigFont{8}{9.6}{\rmdefault}{\mddefault}{\updefault}{\color[rgb]{0,0,0}$x_{1}$}%
}}}}
\put(3151,-811){\makebox(0,0)[lb]{\smash{{\SetFigFont{8}{9.6}{\rmdefault}{\mddefault}{\updefault}{\color[rgb]{0,0,0}$y_{1}$}%
}}}}
\put(1051,-2011){\makebox(0,0)[lb]{\smash{{\SetFigFont{8}{9.6}{\rmdefault}{\mddefault}{\updefault}{\color[rgb]{0,0,0}$x_{2}$}%
}}}}
\put(3151,-2011){\makebox(0,0)[lb]{\smash{{\SetFigFont{8}{9.6}{\rmdefault}{\mddefault}{\updefault}{\color[rgb]{0,0,0}$y_{2}$}%
}}}}
\put(1051,-3211){\makebox(0,0)[lb]{\smash{{\SetFigFont{8}{9.6}{\rmdefault}{\mddefault}{\updefault}{\color[rgb]{0,0,0}$x_{3}$}%
}}}}
\put(3151,-3211){\makebox(0,0)[lb]{\smash{{\SetFigFont{8}{9.6}{\rmdefault}{\mddefault}{\updefault}{\color[rgb]{0,0,0}$y_{3}$}%
}}}}
\put(6151,-1561){\makebox(0,0)[lb]{\smash{{\SetFigFont{8}{9.6}{\rmdefault}{\mddefault}{\updefault}{\color[rgb]{0,0,0}$x_{1}$}%
}}}}
\put(8251,-1561){\makebox(0,0)[lb]{\smash{{\SetFigFont{8}{9.6}{\rmdefault}{\mddefault}{\updefault}{\color[rgb]{0,0,0}$y_{1}$}%
}}}}
\put(6151,-2761){\makebox(0,0)[lb]{\smash{{\SetFigFont{8}{9.6}{\rmdefault}{\mddefault}{\updefault}{\color[rgb]{0,0,0}$x_{2}$}%
}}}}
\put(8251,-2761){\makebox(0,0)[lb]{\smash{{\SetFigFont{8}{9.6}{\rmdefault}{\mddefault}{\updefault}{\color[rgb]{0,0,0}$y_{2}$}%
}}}}
\put(8251,-3961){\makebox(0,0)[lb]{\smash{{\SetFigFont{8}{9.6}{\rmdefault}{\mddefault}{\updefault}{\color[rgb]{0,0,0}$y_{3}$}%
}}}}
\put(6151,-3961){\makebox(0,0)[lb]{\smash{{\SetFigFont{8}{9.6}{\rmdefault}{\mddefault}{\updefault}{\color[rgb]{0,0,0}$x_{3}$}%
}}}}
\put(11176,-811){\makebox(0,0)[lb]{\smash{{\SetFigFont{8}{9.6}{\rmdefault}{\mddefault}{\updefault}{\color[rgb]{0,0,0}$x'_{1}$}%
}}}}
\put(13351,-811){\makebox(0,0)[lb]{\smash{{\SetFigFont{8}{9.6}{\rmdefault}{\mddefault}{\updefault}{\color[rgb]{0,0,0}$y'_{1}$}%
}}}}
\put(11176,-2611){\makebox(0,0)[lb]{\smash{{\SetFigFont{8}{9.6}{\rmdefault}{\mddefault}{\updefault}{\color[rgb]{0,0,0}$x_{2}$}%
}}}}
\put(13351,-2611){\makebox(0,0)[lb]{\smash{{\SetFigFont{8}{9.6}{\rmdefault}{\mddefault}{\updefault}{\color[rgb]{0,0,0}$y_{2}$}%
}}}}
\put(11176,-3811){\makebox(0,0)[lb]{\smash{{\SetFigFont{8}{9.6}{\rmdefault}{\mddefault}{\updefault}{\color[rgb]{0,0,0}$x_{3}$}%
}}}}
\put(13351,-3811){\makebox(0,0)[lb]{\smash{{\SetFigFont{8}{9.6}{\rmdefault}{\mddefault}{\updefault}{\color[rgb]{0,0,0}$y_{3}$}%
}}}}
\put(16276,-811){\makebox(0,0)[lb]{\smash{{\SetFigFont{8}{9.6}{\rmdefault}{\mddefault}{\updefault}{\color[rgb]{0,0,0}$x'_{1}$}%
}}}}
\put(18451,-811){\makebox(0,0)[lb]{\smash{{\SetFigFont{8}{9.6}{\rmdefault}{\mddefault}{\updefault}{\color[rgb]{0,0,0}$y'_{1}$}%
}}}}
\put(16276,-2011){\makebox(0,0)[lb]{\smash{{\SetFigFont{8}{9.6}{\rmdefault}{\mddefault}{\updefault}{\color[rgb]{0,0,0}$x'_{2}$}%
}}}}
\put(18451,-2011){\makebox(0,0)[lb]{\smash{{\SetFigFont{8}{9.6}{\rmdefault}{\mddefault}{\updefault}{\color[rgb]{0,0,0}$y'_{2}$}%
}}}}
\put(16276,-3811){\makebox(0,0)[lb]{\smash{{\SetFigFont{8}{9.6}{\rmdefault}{\mddefault}{\updefault}{\color[rgb]{0,0,0}$x_{3}$}%
}}}}
\put(18451,-3811){\makebox(0,0)[lb]{\smash{{\SetFigFont{8}{9.6}{\rmdefault}{\mddefault}{\updefault}{\color[rgb]{0,0,0}$y_{3}$}%
}}}}
\put(24376,-1486){\makebox(0,0)[lb]{\smash{{\SetFigFont{8}{9.6}{\rmdefault}{\mddefault}{\updefault}{\color[rgb]{0,0,0}$-p$}%
}}}}
\put(23176,-1561){\makebox(0,0)[lb]{\smash{{\SetFigFont{8}{9.6}{\rmdefault}{\mddefault}{\updefault}{\color[rgb]{0,0,0}$x'_{1}$}%
}}}}
\put(25351,-1561){\makebox(0,0)[lb]{\smash{{\SetFigFont{8}{9.6}{\rmdefault}{\mddefault}{\updefault}{\color[rgb]{0,0,0}$y'_{1}$}%
}}}}
\put(23176,-2761){\makebox(0,0)[lb]{\smash{{\SetFigFont{8}{9.6}{\rmdefault}{\mddefault}{\updefault}{\color[rgb]{0,0,0}$x'_{2}$}%
}}}}
\put(25351,-2761){\makebox(0,0)[lb]{\smash{{\SetFigFont{8}{9.6}{\rmdefault}{\mddefault}{\updefault}{\color[rgb]{0,0,0}$y'_{2}$}%
}}}}
\put(23176,-3961){\makebox(0,0)[lb]{\smash{{\SetFigFont{8}{9.6}{\rmdefault}{\mddefault}{\updefault}{\color[rgb]{0,0,0}$x'_{3}$}%
}}}}
\put(25351,-3961){\makebox(0,0)[lb]{\smash{{\SetFigFont{8}{9.6}{\rmdefault}{\mddefault}{\updefault}{\color[rgb]{0,0,0}$y'_{3}$}%
}}}}
\end{picture}%

%% file: wire25a.pstex_t
\begin{picture}(0,0)%
\includegraphics{wire25a.pstex}%
\end{picture}%
\setlength{\unitlength}{1579sp}%
\begingroup\makeatletter\ifx\SetFigFont\undefined%
\gdef\SetFigFont#1#2#3#4#5{%
  \reset@font\fontsize{#1}{#2pt}%
  \fontfamily{#3}\fontseries{#4}\fontshape{#5}%
  \selectfont}%
\fi\endgroup%
\begin{picture}(10244,10244)(19179,-9383)
\end{picture}%

%% file: mnk1.pstex_t
\begin{picture}(0,0)%
\includegraphics{mnk1.pstex}%
\end{picture}%
\setlength{\unitlength}{2072sp}%
\begingroup\makeatletter\ifx\SetFigFont\undefined%
\gdef\SetFigFont#1#2#3#4#5{%
  \reset@font\fontsize{#1}{#2pt}%
  \fontfamily{#3}\fontseries{#4}\fontshape{#5}%
  \selectfont}%
\fi\endgroup%
\begin{picture}(11724,4974)(1339,-9073)
\end{picture}%

%% file: mnk2.pstex_t
\begin{picture}(0,0)%
\includegraphics{mnk2.pstex}%
\end{picture}%
\setlength{\unitlength}{1243sp}%
\begingroup\makeatletter\ifx\SetFigFont\undefined%
\gdef\SetFigFont#1#2#3#4#5{%
  \reset@font\fontsize{#1}{#2pt}%
  \fontfamily{#3}\fontseries{#4}\fontshape{#5}%
  \selectfont}%
\fi\endgroup%
\begin{picture}(12174,4974)(1114,-8398)
\put(3826,-6541){\makebox(0,0)[lb]{\smash{{\SetFigFont{9}{10.8}{\rmdefault}{\mddefault}{\updefault}{\color[rgb]{0,0,0}$\bar q_{ij}$}%
}}}}
\put(11026,-6541){\makebox(0,0)[lb]{\smash{{\SetFigFont{9}{10.8}{\rmdefault}{\mddefault}{\updefault}{\color[rgb]{0,0,0}$\bar q_{ij}$}%
}}}}
\end{picture}%

%% file: mnk3.pstex_t
\begin{picture}(0,0)%
\includegraphics{mnk3.pstex}%
\end{picture}%
\setlength{\unitlength}{2072sp}%
\begingroup\makeatletter\ifx\SetFigFont\undefined%
\gdef\SetFigFont#1#2#3#4#5{%
  \reset@font\fontsize{#1}{#2pt}%
  \fontfamily{#3}\fontseries{#4}\fontshape{#5}%
  \selectfont}%
\fi\endgroup%
\begin{picture}(11094,8124)(1339,-9523)
\put(1891,-3931){\makebox(0,0)[lb]{\smash{{\SetFigFont{12}{14.4}{\rmdefault}{\mddefault}{\updefault}{\color[rgb]{0,0,0}$q_{121}$}%
}}}}
\put(2791,-4831){\makebox(0,0)[lb]{\smash{{\SetFigFont{12}{14.4}{\rmdefault}{\mddefault}{\updefault}{\color[rgb]{0,0,0}$q_{212}$}%
}}}}
\put(1891,-6721){\makebox(0,0)[lb]{\smash{{\SetFigFont{12}{14.4}{\rmdefault}{\mddefault}{\updefault}{\color[rgb]{0,0,0}$q_{113}$}%
}}}}
\put(10891,-3481){\makebox(0,0)[lb]{\smash{{\SetFigFont{12}{14.4}{\rmdefault}{\mddefault}{\updefault}{\color[rgb]{0,0,0}$q_{121}$}%
}}}}
\put(11791,-6271){\makebox(0,0)[lb]{\smash{{\SetFigFont{12}{14.4}{\rmdefault}{\mddefault}{\updefault}{\color[rgb]{0,0,0}$q_{113}$}%
}}}}
\put(11791,-7981){\makebox(0,0)[lb]{\smash{{\SetFigFont{12}{14.4}{\rmdefault}{\mddefault}{\updefault}{\color[rgb]{0,0,0}$q_{212}$}%
}}}}
\put(6391,-4021){\makebox(0,0)[lb]{\smash{{\SetFigFont{12}{14.4}{\rmdefault}{\mddefault}{\updefault}{\color[rgb]{0,0,0}$q_{121}$}%
}}}}
\put(7291,-4921){\makebox(0,0)[lb]{\smash{{\SetFigFont{12}{14.4}{\rmdefault}{\mddefault}{\updefault}{\color[rgb]{0,0,0}$q_{212}$}%
}}}}
\put(6391,-6721){\makebox(0,0)[lb]{\smash{{\SetFigFont{12}{14.4}{\rmdefault}{\mddefault}{\updefault}{\color[rgb]{0,0,0}$q_{113}$}%
}}}}
\end{picture}%